\documentclass[UTF-8,reqno]{amsart}
\usepackage[margin=1.2in]{geometry}
\usepackage{enumerate}
\usepackage{enumitem}
\usepackage{amssymb,url,color, booktabs,nccmath}
\usepackage{mathrsfs}
\usepackage{tikz}
\usepackage{mhequ}
\usepackage{mhsymb}
\usepackage{longtable}

\usepackage{color}
\usepackage[colorlinks=true]{hyperref}
\hypersetup{
    linkcolor=blue,          
    citecolor=red,        
    filecolor=blue,      
    urlcolor=cyan
}

\usetikzlibrary{shapes.misc}
\usetikzlibrary{shapes.symbols}
\usetikzlibrary{snakes}
\usetikzlibrary{decorations}
\usetikzlibrary{decorations.markings}

\def\drawx{\draw[-,solid] (-3pt,-3pt) -- (3pt,3pt);\draw[-,solid] (-3pt,3pt) -- (3pt,-3pt);}
\tikzset{
	root/.style={circle, draw=black, fill=white, inner sep=0pt, minimum size=0.7mm},
	dot/.style={circle,fill=black,draw=black, solid,inner sep=0pt,minimum size=0.5mm},
	square/.style={rectangle,fill=black,draw=black, solid,inner sep=0pt,minimum size=1mm},
	empty/.style={circle,fill=white,draw=white, solid,inner sep=0pt,minimum size=0.5mm},
	var/.style={circle,fill=black!10,draw=black,inner sep=0pt, minimum size=
	2mm},
	symb/.style={circle,fill=symbols,draw=symbols, solid,inner sep=0pt,minimum size=0.5mm},
	yy/.style={circle,fill=gray!20,draw=black,inner sep=0pt,minimum size=0.8mm},
	>=stealth,
	dotred/.style={circle,fill=black!50,inner sep=0pt, minimum size=2mm},
	generic/.style={semithick,shorten >=1pt,shorten <=1pt},
	dist/.style={ultra thick,draw=testcolor,shorten >=1pt,shorten <=1pt},
	testfcn/.style={ultra thick,testcolor,shorten >=1pt,shorten <=1pt,<-},
	testfcnx/.style={ultra thick,testcolor,shorten >=1pt,shorten <=1pt,<-,
		postaction={decorate,decoration={markings,mark=at position 0.6 with {\drawx}}}},
	kprime/.style={semithick,shorten >=1pt,shorten <=1pt,densely dashed,->},
	kprimex/.style={semithick,shorten >=1pt,shorten <=1pt,densely dashed,->,
		postaction={decorate,decoration={markings,mark=at position 0.4 with {\drawx}}}},
	kernel/.style={semithick,shorten >=1pt,shorten <=1pt,->},
	multx/.style={shorten >=1pt,shorten <=1pt,
		postaction={decorate,decoration={markings,mark=at position 0.5 with {\drawx}}}},
	kernelx/.style={semithick,shorten >=1pt,shorten <=1pt,->,
		postaction={decorate,decoration={markings,mark=at position 0.4 with {\drawx}}}},
	kernel1/.style={->,semithick,shorten >=1pt,shorten <=1pt,postaction={decorate,decoration={markings,mark=at position 0.45 with {\draw[-] (0,-0.1) -- (0,0.1);}}}},
	kernel2/.style={->,semithick,shorten >=1pt,shorten <=1pt,postaction={decorate,decoration={markings,mark=at position 0.45 with {\draw[-] (0.05,-0.1) -- (0.05,0.1);\draw[-] (-0.05,-0.1) -- (-0.05,0.1);}}}},
	kernelBig/.style={semithick,shorten >=1pt,shorten <=1pt,decorate, decoration={zigzag,amplitude=1.5pt,segment length = 3pt,pre length=2pt,post length=2pt}},
	rho/.style={dotted,semithick,shorten >=1pt,shorten <=1pt},
	renorm/.style={shape=circle,fill=white,inner sep=1pt},
	labl/.style={shape=rectangle,fill=white,inner sep=1pt},
	xi/.style={circle,fill=symbols!10,draw=symbols,inner sep=0pt,minimum size=1.2mm},
	xix/.style={crosscircle,fill=symbols!10,draw=symbols,inner sep=0pt,minimum size=1.2mm},
	xib/.style={circle,fill=symbols!10,draw=symbols,inner sep=0pt,minimum size=1.6mm},
	xibx/.style={crosscircle,fill=symbols!10,draw=symbols,inner sep=0pt,minimum size=1.6mm},
	not/.style={circle,fill=symbols,draw=symbols,inner sep=0pt,minimum size=0.5mm},
	>=stealth,
	}

\colorlet{symbols}{blue!90!black}

\makeatletter
\def\DeclareSymbol#1#2#3{\expandafter\gdef\csname MH@symb@#1\endcsname{\tikz[baseline=#2,scale=0.15]{#3}}%
\expandafter\gdef\csname MH@symb@#1s\endcsname{\scalebox{0.6}{\tikz[baseline=#2,scale=0.15]{#3}}}}
\def\<#1>{\csname MH@symb@#1\endcsname}
\makeatother

\DeclareSymbol{X}{-2.4}{\node[dot] {};}
\DeclareSymbol{1}{0}{\draw[white] (-.4,0) -- (.4,0); \draw (0,0)  -- (0,1.2) node[dot] {};}
\DeclareSymbol{2}{0}{\draw (-0.5,1.2) node[dot] {} -- (0,0.2) -- (0.5,1.2) node[dot] {};}
\DeclareSymbol{3}{0}{\draw (0,0) -- (0,1.2) node[dot] {}; \draw (-.7,1) node[dot] {} -- (0,0) -- (.7,1) node[dot] {};}
\DeclareSymbol{31}{-3}{\draw (0,0) -- (0,-0.7) node[root] {} -- (0.9,0.1) node[dot] {}; \draw (0,0) -- (0,1.2) node[dot] {}; \draw (-.7,1) node[dot] {} -- (0,0) -- (.7,1) node[dot] {};}
\DeclareSymbol{30}{-3}{\draw (0,0) -- (0,-0.7); \draw (0,0) -- (0,1.2) node[dot] {}; \draw (-.7,1) node[dot] {} -- (0,0) -- (.7,1) node[dot] {};}
\DeclareSymbol{10}{-3}{\draw (0,0) -- (0,-1); \draw (-.8,1) node[dot] {} -- (0,0);}
\DeclareSymbol{32}{-3}{\draw (0,0) -- (0,-0.7) -- (0.7,0) node[dot] {}; \draw  (0,-0.7)  -- (-0.7,0) node[dot] {}; \draw (0,0) -- (0,1) node[dot] {}; \draw (-.7,0.8) node[dot] {} -- (0,0) -- (.7,0.8) node[dot] {}; \node [root] at (0,-0.7) {};}
\DeclareSymbol{12}{-3}{\draw (0,0) -- (0,-1) -- (1,0) node[dot] {}; \draw (0,0) -- (0,-1) -- (-1,0) node[dot] {}; \draw (-.8,1) node[dot] {} -- (0,0);}
\DeclareSymbol{22}{-3}{\draw (0,0) -- (0,-0.7) -- (0.7,0) node[dot] {}; \draw (0,0) -- (0,-0.7) -- (-0.7,0) node[dot] {};\draw (-.5,0.8) node[dot] {} -- (0,0) -- (.5,0.8) node[dot] {};  \node [root] at (0,-0.7) {};}
\DeclareSymbol{20}{-3}{\draw (0,0) -- (0,-0.8);\draw (-.6,0.8) node[dot] {} -- (0,0) -- (.6,0.8) node[dot] {};}
\DeclareSymbol{32W}{-3}{\draw  [densely dotted, semithick] (-1.5,0)node[dot] {}  -- (0,-1.5) node[yy]{} -- (1.5,0) node[dot] {}; \draw (.8,-2.3)node{\tiny $y$} ;\draw[semithick] (0,0) -- (0,-1.5) node[yy] {}; \draw [densely dotted, semithick] (0,0) -- (0,1.8) node[dot] {}; \draw[densely dotted, semithick] (-1,1.5) node[dot] {} -- (0,0) -- (1,1.5) node[dot] {};}
\DeclareSymbol{32WW}{-3}{\draw  [densely dotted, semithick] (-1.5,0)node[dot] {}  -- (0,-1.5) node[yy]{} -- (1.5,0) node[dot] {}; \draw (.8,-2.3)node{\tiny $\bar{y}$} ;\draw[semithick] (0,0) -- (0,-1.5) node[yy] {}; \draw [densely dotted, semithick] (0,0) -- (0,1.8) node[dot] {}; \draw[densely dotted, semithick] (-1,1.5) node[dot] {} -- (0,0) -- (1,1.5) node[dot] {};}

\DeclareSymbol{s1}{0}{\draw[white] (-.4,0) -- (.4,0); \draw[symbols] (0,0)  -- (0,1.2) node[symb] {};}
\DeclareSymbol{s2}{0}{\draw[symbols] (-0.5,1.2) node[symb] {} -- (0,0) -- (0.5,1.2) node[symb] {};}
\DeclareSymbol{s3}{0}{\draw[symbols] (0,0) -- (0,1.2) node[symb] {}; \draw[symbols] (-.7,1) node[symb] {} -- (0,0) -- (.7,1) node[symb] {};}
\DeclareSymbol{s31}{-3}{\draw[symbols] (0,0) -- (0,-1) -- (1,0) node[symb] {}; \draw[symbols] (0,0) -- (0,1.2) node[symb] {}; \draw[symbols] (-.7,1) node[symb] {} -- (0,0) -- (.7,1) node[symb] {};}
\DeclareSymbol{s30}{-3}{\draw[symbols] (0,0) -- (0,-1); \draw[symbols] (0,0) -- (0,1.2) node[symb] {}; \draw[symbols] (-.7,1) node[symb] {} -- (0,0) -- (.7,1) node[symb] {};}
\DeclareSymbol{s32}{-3}{\draw[symbols] (0,0) -- (0,-1) -- (1,0) node[symb] {}; \draw[symbols] (0,0) -- (0,-1) -- (-1,0) node[symb] {}; \draw[symbols] (0,0) -- (0,1.2) node[symb] {}; \draw[symbols] (-.7,1) node[symb] {} -- (0,0) -- (.7,1) node[symb] {};}
\DeclareSymbol{s22}{-3}{\draw[symbols] (0,0.3) -- (0,-1) -- (1,0) node[symb] {}; \draw[symbols] (0,0.3) -- (0,-1) -- (-1,0) node[symb] {};\draw[symbols] (-.7,1) node[symb] {} -- (0,0.3) -- (.7,1) node[symb] {};}
\DeclareSymbol{s20}{-3}{\draw[symbols] (0,0) -- (0,-1);\draw[symbols] (-.7,1) node[symb] {} -- (0,0) -- (.7,1) node[symb] {};}
\DeclareSymbol{s21}{-3}{\draw[symbols] (0,0) -- (0,-1);\draw[symbols] (-.7,1) node[symb] {} -- (0,0) -- (.7,1) node[symb] {}; \draw[symbols] (0,0) -- (0,-1) -- (1,0) node[symb] {};}
\DeclareSymbol{s12}{-3}{\draw[symbols] (0,0) -- (0,-1) -- (1,0) node[symb] {}; \draw[symbols] (0,0) -- (0,-1) -- (-1,0) node[symb] {}; \draw (-.8,1) node[symb] {} -- (0,0);}
\DeclareSymbol{s11}{-3}{\draw[symbols] (0,0) -- (0,-1) -- (-1,0) node[symb] {}; \draw (-.8,1) node[symb] {} -- (0,0);}
\DeclareSymbol{s10}{-3}{\draw[symbols] (0,0) -- (0,-1); \draw[symbols] (-.8,1) node[symb] {} -- (0,0);}

\definecolor{darkergreen}{rgb}{0.0, 0.5, 0.0}

\numberwithin{equation}{section}
\def\theequation{\arabic{section}.\arabic{equation}}
\newcommand{\be}{\begin{eqnarray}}
\newcommand{\ee}{\end{eqnarray}}
\newcommand{\ce}{\begin{eqnarray*}}
\newcommand{\de}{\end{eqnarray*}}
\newtheorem{theorem}{Theorem}[section]
\newtheorem{lemma}[theorem]{Lemma}
\newtheorem{remark}[theorem]{Remark}
\newtheorem{definition}[theorem]{Definition}
\newtheorem{proposition}[theorem]{Proposition}
\newtheorem{Examples}[theorem]{Example}
\newtheorem{corollary}[theorem]{Corollary}

\newcommand{\rmk}[1]{\textcolor{red}{#1}}

\newcommand{\LL}{\mathscr{L}}

\newcommand{\UU}{\mathscr{U}}

\def\Wick#1{\,\colon\! #1 \colon}

\def\eps{\varepsilon}

\def\p{\partial}

\def\l{\lambda}
\def\la{\langle}
\def\ra{\rangle}
\def\[{{\Big[}}
\def\]{{\Big]}}
\def\({{\Big(}}
\def\){{\Big)}}

\def\bx{{\mathbf{x}}}

\def\dif{{\mathord{{\rm d}}}}

\def\no{\nonumber}
\def\={&\!\!=\!\!&}

\def\qi{{ [i]}}
\def\qj{{ [j]}}
\def\ql{{ [l]}}

\def\bC{{\mathbf C}}

\def\cI{{\mathcal I}}

\def\cN{{\mathcal N}}

\def\cZ{{\mathcal Z}}

\def\mR{{\mathbb R}}

\def\mZ{{\mathbb Z}}

\def\1{{\mathbf{1}}}

\def\sD{{\mathscr D}}

\def\sL{{\mathscr L}}

\def\E{\mathbf E}

\def\geq{\geqslant}
\def\leq{\leqslant}
\def\ge{\geqslant}
\def\le{\leqslant}

\def\eps{\varepsilon}

\def\p{\partial}

\def\l{\lambda}
\def\la{\langle}
\def\ra{\rangle}
\def\[{{\Big[}}
\def\]{{\Big]}}
\def\({{\Big(}}
\def\){{\Big)}}

\def\bx{{\mathbf{x}}}

\def\dif{{\mathord{{\rm d}}}}

\def\no{\nonumber}
\def\={&\!\!=\!\!&}
\def\bt{\begin{theorem}}
\def\et{\end{theorem}}
\def\bl{\begin{lemma}}
\def\el{\end{lemma}}
\def\br{\begin{remark}}
\def\er{\end{remark}}
\def\bx{\begin{Examples}}
\def\ex{\end{Examples}}
\def\bd{\begin{definition}}
\def\ed{\end{definition}}
\def\bp{\begin{proposition}}
\def\ep{\end{proposition}}
\def\bc{\begin{corollary}}
\def\ec{\end{corollary}}

\def\geq{\geqslant}
\def\leq{\leqslant}
\def\ge{\geqslant}
\def\le{\leqslant}

 \def\R{\mathbb R}
 \def\R{\mathbb R}

\allowdisplaybreaks

\begin{document}
\begin{abstract}
In this paper we study
the large $N$ limit of the $O(N)$-invariant linear sigma model, which is a vector-valued
generalization of the $\Phi^4$ quantum field theory, on the three dimensional torus.
We study the problem via its stochastic quantization, which yields
a coupled system of $N$ interacting SPDEs.
We prove  tightness of the invariant measures in the
large $N$ limit. For large enough mass or small enough coupling constant, they converge to the (massive) Gaussian free field  at a rate of order $1/\sqrt N$ with respect to the Wasserstein distance.
We also obtain tightness results for certain $O(N)$ invariant observables.
These generalize some of the results in \cite{SSZZ20} from two dimensions to three dimensions.
The proof leverages the method recently developed by \cite{GH18}
and combines many new techniques such as uniform in $N$ estimates on perturbative objects as well as the solutions.
\end{abstract}

\title{Large $N$ limit  of the $O(N)$ linear sigma model in 3D}

\author{Hao Shen}
\address[H. Shen]{Department of Mathematics, University of Wisconsin - Madison, USA}
\email{pkushenhao@gmail.com}

\author{Rongchan Zhu}
\address[R. Zhu]{Department of Mathematics, Beijing Institute of Technology, Beijing 100081, China}
\email{zhurongchan@126.com}

\author{Xiangchan Zhu}
\address[X. Zhu]{ Academy of Mathematics and Systems Science,
Chinese Academy of Sciences, Beijing 100190, China}
\email{zhuxiangchan@126.com}

\date{\today}

\maketitle

\setcounter{tocdepth}{1}
\tableofcontents

\section{Introduction}


In this paper, we continue the study initiated in \cite{SSZZ20}
on the application of singular SPDE methods to large $N$ problems in quantum field theory (QFT).
Large $N$ problems in QFT generally refer to the study of asymptotic behaviors  of QFT models as the dimensionality of the target space where the quantum fields take values tends to infinity.
Physicists' study of large N problems in QFT originates
from the seminal work by Wilson \cite{wilson1973quantum} and Gross--Neveu \cite{gross1974dynamical}, and soon flourished following the work by t'Hooft
\cite{tHooft1974planar} who applied it to gauge theories, which influenced many aspects of probability. 
We refer to \cite[Section~1.1]{SSZZ20} for a more thorough exposition of the background and motivation for large $N$ methods in QFT.
A prototype model, which we consider in this paper,
is the
 $O(N)$-invariant linear sigma model given by the  (formal) measure
\begin{equation}\label{e:Phi_i-measure}
\dif\nu^N(\Phi)\eqdef \frac{1}{C_N}\exp\bigg(-\int_{\mathbb T^d} \sum_{j=1}^N|\nabla \Phi_j|^2+m \sum_{j=1}^N\Phi_j^2
+\frac{\l}{2N} \Big(\sum_{j=1}^N\Phi_j^2\Big)^2 \dif x\bigg)\mathcal D \Phi
\end{equation}
over $\R^N$ valued fields $\Phi=(\Phi_1,\Phi_2,...,\Phi_N)$,
where $\mathbb T^d$ is the $d$-dimensional torus and $C_N$ is a normalization constant (partition function).
This is an $N$-component generalization of the $\Phi^4_d$ model,
which is symmetric under the action by the orthogonal group $O(N)$ on the field $\Phi$.
In this paper we focus on $d=3$.

Our main tool to investigate the large $N$ behavior of the above model \eqref{e:Phi_i-measure} is
the system of SPDEs arising as its stochastic quantization:
\begin{equation}\label{eq:main}
\LL \Phi_i= -\frac{\l}{N}\sum_{j=1}^N \Phi_j^2\Phi_i+\xi_i,  
\end{equation}
where $\LL=\p_t-\Delta+m$ with $m\geq0$, and $i\in\{1,\cdots,N\}$.
The collection $(\xi_i)_{i=1}^N$ consists of $N $ independent space-time white noises on a probability space $(\Omega,\mathcal{F},\mathbf{P})$.
The utilization of the
dynamic \eqref{eq:main} to probe large $N$ properties of
the QFT model \eqref{e:Phi_i-measure} lead to connections with dynamical mean field theory, see \cite[Section~1.2]{SSZZ20} for more discussions on this connection.  We also note that for finite $N$, the measure $\nu^N$ formally given by \eqref{e:Phi_i-measure} can be constructed rigorously as the unique invariant measure of SPDEs \eqref{eq:main} -- see Lemma \ref{lem:zz1}, and we henceforth denote by $\nu^N$ this well-defined measure. 

In the 2D paper \cite{SSZZ20}, since the SPDE is less singular,
 a series of  results was obtained.
It is proved that the non-stationary dynamics
converge to a mean field dynamic
$\LL \Psi_i= -\l \E[\Psi_i^2]\Psi_i+\xi_i$
(which needs to be interpreted via a suitable renormalization)
as $N\to\infty$.
Also, for large enough mass, the invariant measures converge to the (massive) Gaussian free field which is the
unique invariant measure of the mean-field dynamic.
Moreover, this paper obtained tightness in suitable Besov spaces of   some $O(N)$ invariant observables as $N\to \infty$,
and also proved exact correlation formulae for these observables.

In $d=3$, the SPDE becomes much more singular.  For $N=1$, namely the dynamical $\Phi^4_3$ model,
the construction of  local solutions was achieved by \cite{Hairer14} using the theory of regularity structures and then %
\cite{MR3846835} using paracontrolled distributions
developed in
\cite{GIP15}; and global solutions were studied in \cite{MW18,GubCMP19,AK17,moinat2020space,GH18}.
One should be able to extend these constructions and estimates
to the vector-valued case for a finite and {\it fixed} $N>1$.

The goal of  this article is to study the asymptotic behavior as $N\to \infty$ of the  invariant measures \eqref{e:Phi_i-measure} as well as $O(N)$-invariant observables, in $d=3$. Given the more singular nature in $d=3$,
it is not the purpose of this article to extend all the results in the 2D paper \cite{SSZZ20} to 3D. Instead we only focus on large $N$ limit of the measure $\nu^N$ in \eqref{e:Phi_i-measure} and tightness of certain observables in $d=3$.
Our proof is self-contained based on the techniques from mean field theory and recent progress of singular SPDEs, and does not rely on methods in constructive field theory such as cluster expansion and correlation inequalities. Our approach may bring a new perspective to study the large $N$ limit problem.

\subsection{Main results}\label{sec:main-res}

Our first result shows that assuming $m$ to be sufficiently large or $\lambda$ to be sufficiently small,
the invariant measures converge to the (massive) Gaussian free field.
Note that since the interaction term in the measure \eqref{e:Phi_i-measure}
has a sum besides the factor $1/N$,
it is far from  being obvious that the large $N$ limit of the measures
is a Gaussian free field.
This was only heuristically predicted by physicists (e.g.\cite{wilson1973quantum}) at the level of perturbation theory: namely, after suitable renormalization, all the Feynman diagrams are of order $N^\alpha$ for some $\alpha<0$ except for the diagrams corresponding to the Gaussian free field (see Section~\ref{sec:heu} below for more discussion).

To state the result, let $\nu\eqdef \mathcal{N}(0,\frac12(m-\Delta)^{-1})$ be the (massive) Gaussian free field.
Consider the projection onto the $i^{th}$ component of the field $\Phi$
\begin{equ}[e:Pr-Pi]
	\Pi_i:\mathcal{S}'(\mathbb{T}^d)^N\rightarrow \mathcal{S}'(\mathbb{T}^d),
	\qquad \Pi_i(\Phi)\eqdef \Phi_i.
\end{equ}
Noting that $\nu^N$ is a measure on $\mathcal{S}'(\mathbb{T}^d)^N$, we define the marginal law $\nu^{N,i}\eqdef \nu^N\circ \Pi_i^{-1}$. Furthermore, consider
\begin{equ}[e:Pr-Pik]
	\Pi^{(k)}:\mathcal{S}'(\mathbb{T}^d)^N\rightarrow \mathcal{S}'(\mathbb{T}^d)^k,
	\qquad
	\Pi^{(k)}(\Phi)=(\Phi_i)_{1\leq i\leq k}
\end{equ}
and define the marginal law of the first $k$ components by $\nu^{N}_k \eqdef \nu^N\circ (\Pi^{(k)})^{-1}$.
We denote by $\mathbb{W}_{2,k}$  the
		Wasserstein distance, with the precise definition given in \eqref{e:W2}.
We recall the definitions of Besov spaces $B^\alpha_{p,q}$, $\bC^\alpha$ and $H^\alpha$ in
Section~\ref{sec:nota} and Appendix~\ref{sec:pre}.


\begin{theorem} 
\label{th:1.2}
		For  any $(m,\l) \in (0,\infty)\times [0,\infty)$ and  every $i\ge 1$, the sequence of probability measures $(\nu^{N,i})_{N \geq 1}$ is tight on $H^{-\frac12-\kappa}$ for any $\kappa>0$.
		Moreover, there exists a constant $c_0>0$ such that for all $(m,\l) \in [1,\infty)\times [0,\infty)$ satisfying
		 $m\geq 1+c_0\l(1+\l^{59})$ and every  $k\ge 1$, there exists a constant $C_k>0$ such that $\mathbb{W}_{2,k}(\nu^{N}_k,\nu^{\otimes k})\le C_k N^{-\frac{1}{2}}$. 
\end{theorem}

Our next result is concerned with $O(N)$-invariant observables for the invariant measure.
We refer to  \cite[Section~1]{SSZZ20} for more discussion on the motivation of studying  observables for QFT models with continuous symmetries.
Here a natural quantity  that is invariant under $O(N)$-rotations is the squared ``length'' of $\Phi$, that is suitably renormalized and scaled with respect to $N$:
\begin{equ}[e:twoObs]
	\frac{1}{\sqrt{N}} \sum_{i=1}^N\Wick{\Phi^2_i},
	\qquad \Phi = (\Phi_i)_{1\leq i\leq N}\thicksim \nu^N.
\end{equ}
The precise definition  is given by \eqref{eq:o1} in Section~\ref{sec:inv}.
We establish the large $N$ tightness of these observables as random fields in suitable Besov spaces.

\begin{theorem}\label{th:1.3}
	 For $(m,\l) \in [1,\infty)\times [0,\infty)$ satisfying $m\geq 1+c_0\l(1+\l^{59})$ as in Theorem \ref{th:1.2},
	 the sequence of random variables
	  $\big( \frac{1}{\sqrt{N}} \sum_{i=1}^N\Wick{\Phi^2_i} \big)_{N\ge 1}$ is tight on the Besov space $B^{-1-\kappa}_{1,1}$  for any $\kappa>0$.
\end{theorem}



The proofs of Theorem~\ref{th:1.2} (split into Theorems~\ref{th:con} and \ref{th:main})
and Theorem~\ref{th:1.3} will be given in
Section~\ref{sec:inv}. 

Let us briefly mention some of the earlier related results 
on large N problems in QFT, and 
we refer to \cite[Sec.~1.1 and 1.2]{SSZZ20} for a more complete
historical remarks and literature
on the physical and mathematical results.

On a  {\it fixed} lattice,
Kupiainen \cite{MR574175} studied the $1/N$ expansion 
for the nonlinear sigma (classical Heisenberg) model, see also \cite{MR678004},
and more recently, 
Chatterjee and Jafarov 
 \cite{MR3919447,chatterjee2016} 
obtained the $1/N$ expansion for lattice gauge theories.
The lattice cutoff allows one to 
avoid the ultraviolet or singularity problem,
which is  convenient especially when the continuum limit of the model is not constructed or may not exist.
On the other hand, 
L\'evy \cite{Levy2011} and  Anshelevich--Sengupta \cite{MR2864481}
constructed the so called master field for the two-dimensional Yang--Mills model in continuum
which is the large $N$ limit of 
 Wilson loop observables.
While the two-dimensional Yang--Mills field 
 in continuum is very singular (and the random connection field was only constructed recently by Cheryvev \cite{MR4034782}),
the model in $d=2$ has a certain solvability property, so these large $N$ results 
are very difficult to be generalized to 3D continuum space, if possible, since the solvability would be lost.
Back to the vector $\Phi^4$ or $O(N)$ linear sigma model in continuum,
the first large $N$ result was achieved by Kupiainen \cite{MR578040} (see also the review \cite{MR582622})
in which $1/N$ expansion for the
pressure (vacuum energy) of the model  \eqref{e:Phi_i-measure} in $d=2$ was obtained
using constructive field theory methods. See also  \cite{MR661137}.
Generalizing the results in  \cite{MR578040} to 3D using constructive field theory methods
would be very interesting but, as far as we understand, might be difficult.

To our best knowledge nothing was known in a very singular setting
such as the model being considered here in $d=3$,
which is now analyzable
 thanks to the new SPDE methods developed more recently.
See
 Section~\ref{sec:future} for more discussion.

\subsection{Methodology and difficulties}
To obtain our main results, we outline the main steps of our strategy. To prove tightness of $\nu^{N,i}$ and the observables we turn to uniform in $N$ moment estimates for the stationary solutions $\Phi$ to the stochastic quantization \eqref{eq:main}. 
Here $\Phi$ is only a distribution and \eqref{eq:main} requires renormalization.
Our general strategy is to apply  SPDE techniques, especially {\it energy method} to derive uniform in $N$ estimates of $\Phi$. Due to the singularity of the noise $\xi$ and also $\Phi$, we first decompose
	\begin{align}\label{eq:dec}
		\Phi_i=Z_i+X_i+Y_i,
	\end{align}
	to isolate parts of different regularities. 
Here $Z_i$ is the stationary solution to the linear equation
 $\sL Z_i=\xi_i$ and is the most irregular part in $\Phi_i$, and $X_i$ involves the second irregular part $\tilde\cZ^{\<30>}$ 
 (see Section \ref{sec:ren} for the definition of $\tilde\cZ^{\<30>}$ and \eqref{eq:211} for the definition of $X_i$).

With \eqref{eq:dec} at hand, $Y_i$ becomes more regular and we  prove uniform in $N$ energy estimate for $Y_i$, 
which solves an equation driven by $X$ and renormalized powers of $Z$ (see \eqref{eq:22} below).
As in the 2D case, the extra factor $1/N$ before the nonlinear terms makes the damping effect from $Y_j^2Y_i$ weaker as $N$ becomes large, so we cannot exploit the strong damping effect at the level of a fixed component of $Y$. 
Instead we  consider aggregate quantities, and ultimately we focus on the empirical average of the $L^2$-norms. 
The advantage  
 is that the dissipation term $\|\frac{1}{N}\sum_{i=1}^N Y_i^2\|_{L^2}^2$ from $Y_j^2Y_i$ (see Section~\ref{sec:dec}) can be used to control the nonlinear terms.  
 We then show that the following averaged quantities
	\begin{align}\label{eq:YN}
		\frac1N\sum_{j=1}^N\|Y_j(t)\|_{L^2}^2+\frac1N\sum_{j=1}^N\|Y_j\|_{L_T^2H^{1-2\kappa}}+\frac{\lambda}{N^2}\Big\|\sum_{i=1}^NY_i^2\Big\|_{L_T^2L^2}
	\end{align}
are controlled pathwise in terms of the averages of the renormalized powers of $Z$ (see Theorem \ref{Y:T4a}), which is enough to derive  tightness of $\nu^{N,i}$ in Theorem \ref{th:1.2}.
Furthermore, we derive uniform in $N$ bounds for   the quantities in \eqref{eq:YN} multiplied by $N$, i.e. the sum of $L^2$-norm of $Y_i$  (see Theorem \ref{Y:T4}).

We also remark that in order to establish uniform in $N$ bounds,
it would be a natural attempt to exploit the method in \cite{MW18} or \cite{GubCMP19}
(which obtained global a priori estimates for dynamical $\Phi^4_3$).
However, it is not clear  how to 
 exploit the aforementioned dissipation effect to deduce uniform in $N$ bounds with $L^p$ ($p>2$) estimates developed in \cite{MW18}; also  it seems difficult to apply the maximum principle used in \cite{GubCMP19, moinat2020space} to the vector valued case.
 In this paper, we exploit the approach developed in \cite{GH18} (which relies on interesting cancellations)
  to establish $L^2$ uniform  bounds;
 due to the  cancellation
certain higher regularity estimate which would be rather technical
 is not required anymore for $L^2$ uniform estimates (see Section \ref{sec:dec} for more details).
 On the other hand, compared to the dynamical $\Phi^4_3$ model, the dissipation effect from the term $\|\frac{1}{N}\sum_{i=1}^N Y_i^2\|_{L^2}^2$  is weaker than the $L^4$ norm. We will perform further decomposition and choose suitable parameter to balance the competing contributions for the estimate of the cubic term (see Lemma \ref{Cubic} and discussion before it for more details.)

An important ingredient in our proof is a cancelation mechanism
which emerges from Section~\ref{sec:unis}
and arises from the following reasons. In our decomposition of the solution to \eqref{eq:main},
the leading order terms consist of (polynomials of) $N$ independent Gaussian processes;
so when calculating moments of certain sums of these leading order terms in a suitably chosen Hilbert space,
many terms do not contribute, which allows us to gain ``factors of $1/N$''.
See Lemma~\ref{lem:sto} (and also Lemmas~\ref{lem:X2}-\ref{lem:X3}) for these effects.

As a final step, as in the 2D case, we follow the idea in \cite{GH18} and construct a jointly stationary process $(\Phi,Z)$ whose components satisfy \eqref{eq:main} and the linear equation \eqref{eq:li1} respectively. The law of $Z$ is the Gaussian free field $\nu$
and the law of $(\Phi,Z)$ gives a coupling between the measures $\nu^N$ and $\nu$.  We use this coupling to prove convergence of $\nu^{N,i}$ to $\nu$ by invoking the uniform in $N$ estimates on the stationary processes. 
More precisely, we further use mean field limit techniques to center the stochastic objects (i.e. $R_N^1, R_N^2$ and $Q_N^3$ in Theorem \ref{Y:T4}) in our uniform estimates; their means will contribute to a mass term, which  ultimately requires the assumptions on $m,\l$. 
The part involving the centered stochastic objects
could be absorbed by the dissipation terms from nonlinearity;
and we will gain a factor of $1/N$ when computing variance of 
the centered stochastic objects (essentially a generalization of the elementary fact that the variance of the average of $N$ mean-zero i.i.d. random variables is $O(1/N)$).
Finally tightness of observables almost immediately follows
	from the uniform in $N$ estimates in Theorem \ref{th:main}.

\subsection{Background and heuristics.}
\label{sec:heu}

We review here how
physicists (e.g. Wilson \cite{wilson1973quantum}) predicted convergence of the model
\eqref{e:Phi_i-measure} (after inserting suitable renormalization constants in a way that corresponds to \eqref{eq:ap})
to GFF as $N\to \infty$, and why their heuristic argument was far from being rigorous.
Since this is completely formal, we simply drop the renormalization and pretend that
our fields can be evaluated at spatial points.

The prediction was based on  viewing \eqref{e:Phi_i-measure}
as a perturbation of GFF, namely, to formally Taylor expand \eqref{e:Phi_i-measure} in $\lambda$.
As an example, one can calculate the two-point correlation $\E[\Phi_i(y_1)\Phi_i(y_2)]$ (for a fixed $i$)
in this way; the zeroth order term  in this Taylor expansion is obtained by simply taking $\l=0$,
which gives $\E[Z_i(y_1)Z_i(y_2)]$, where $Z_j \sim \mathcal{N}(0,\frac12(m-\Delta)^{-1})$.
Each of the higher order terms in this Taylor expansion can be calculated as an expectation of a product of Gaussians using Wick theorem:
for instance at order $\lambda^2$ {\it one of} the terms has the form
$$
\frac{\l^2}{N^2} \int \E \Big[ Z_i(y_1)Z_i(y_2)
\Big(\!\!\! \sum_{i_1,j_1=1}^N  \!\!\! Z_{i_1} (x_1)^2 Z_{j_1} (x_1)^2 \Big)
\Big(\!\!\! \sum_{i_2,j_2=1}^N \!\!\! Z_{i_2} (x_2)^2 Z_{j_2} (x_2)^2 \Big) \Big] \dif x_1\dif x_2.
$$
 {\it One of} the terms  obtained from applying Wick theorem to the above expectation has the form
\begin{equs}
{}  \frac{\l^2}{N^2}  \!\!\!\!
 \sum_{i_1,j_1,i_2,j_2=1}^N \!\!\!\!\!\!\!
\delta_{i,i_1}\delta_{i,i_2} \delta_{i_1,i_2} \delta_{j_1,j_2}
& \int
 \E[ Z_i(y_1)Z_{i_1} (x_1)] \E[ Z_i(y_2)Z_{i_2} (x_2)]
\\
&\times
\E[ Z_{i_1} (x_1) Z_{i_2} (x_2)]
\E[Z_{j_1} (x_1)Z_{j_2} (x_2)]^2
 \dif x_1\dif x_2
\end{equs}
where the Kronecker $\delta$'s come from
the independence of $(Z_j)_{j=1}^N$.  (This integral is convergent after introducing the renormalization constant $\tilde b_\eps$ as in  \eqref{eq:ap}.)
Thanks to these Kronecker $\delta$'s, the summation has $N$ terms (rather than $N^4$ terms),
so the above expression is of order $1/N$ and thus converge to $0$ as $N\to \infty$.
This argument is far from being rigorous because
even if it could be shown that  every order of the Taylor expansion (except for the zeroth order)
 converges to $0$ as $N\to \infty$,
 and even if this could be done for higher order correlations,
 the above Taylor expansion in $\lambda$ does not converge.
\footnote{For $\Phi^4_2$, the expansion in $\lambda$ was proved to be divergent in \cite{JaffeDivergence},
but proved to be asymptotic in  \cite{dimock1974,perturbation}.}
This type of predictions can be rigorously proved
now thanks to the  recent development of
singular SPDE techniques and the methods developed in this paper.
On the other hand, as we will see,
some of our calculations such as Lemma~\ref{lem:sto}
have a similar flavor with the formal perturbative arguments
discussed above.

\subsection{Future directions}
\label{sec:future}

We mention a few possible future directions.
The proof of tightness in Theorem~\ref{th:1.3} requires
establishing uniform bounds on  moment of the observables.
It will be more interesting to prove
exact formulae of correlations of the observables (as done in \cite{SSZZ20} in $d=2$) in the $N\to\infty$ limit, by first using integration by parts (i.e. Dyson--Schwinger equations)
to find the leading contribution to the formula and then establishing uniform bounds which show that the remainder terms vanish.
In 3D it would require more effort to interpret the Dyson--Schwinger equations
and prove the bounds for the remainder, so we leave it to future work. 
We  hope that the SPDE approach and uniform estimates derived in Section \ref{sec:uni} can be used to characterize the tightness limit
of $O(N)$ invariant observables, 
 prove  exact correlation formulas in the large $N$ limit or systematic $1/N$ expansions, 
for which we may need to combine with the methods from QFT, for instance, 
the dual field representation rigorously studied in \cite{MR578040}.

Another interesting question is whether our convergence results hold for a larger range of $(m,\l)$, or even over the entire $(m,\l)\in \R^2_+$, at least on the torus.
It would be interesting to generalize some of our results to infinite volume setting by introducing suitable weights. 
 Let us also mention that proving or disproving phase transitions for $N$ large is an important question,
see for instance \cite{MR3602816} for a different but related model and references therein,
and we mention \cite{CGW2020} on phase transition of $\Phi^4_3$ using stochastic analysis methods.

Finally, as discussed in Section~\ref{sec:main-res}, it would certainly be interesting to investigate large $N$ problems for more challenging models without lattice cutoff, such as Lie algebra valued
\cite{CCHS20, CCHS22} or manifold valued  \cite{hairer2019geo,hairer2016motion, RWZZ17, CWZZ18} models with large dimension of target space via stochastic quantization.

\subsection{Structure of the paper}
This paper is organized as follows.
In Section~\ref{sec:Sto}, we set up the decomposition of our SPDE system, discuss renormalization,
derive the energy balance identity, and prove uniform in $N$ bounds for stochastic terms.
Section~\ref{sec:uni} is devoted to the proof of  uniform in $N$
energy estimates for the solution.

Section \ref{sec:inv} is concerned with the proof of Theorem \ref{th:1.2} and Theorem \ref{th:1.3}.
The convergence of invariant measures from $\nu^{N,i}$ to the Gaussian free field $\nu$ is shown in Section \ref{sec:Con}. 
Section~\ref{sec:Ob} is devoted to the study of the observables and the proof of Theorem \ref{th:1.3}. In Appendix~\ref{sec:pre}
we introduce the basic notations and recall useful tools such as paraproducts
and commutator estimates used throughout the paper. In Appendix \ref{sec:extra} we give an extra estimate and finally we collect some notation  in Appendix \ref{sec:extra1}.

\subsection{Some notations.}\label{sec:nota}
 Throughout the paper, we use the notation $a\lesssim b$ if there exists a proportional constant $c>0$ such that $a\leq cb$, and we write $a\simeq b$ if $a\lesssim b$ and $b\lesssim a$. 
We say that $a\lesssim b$ uniformly in a parameter  if the proportional constant $c$ does not depend on this parameter.

Given a Banach space $E$ with a norm $\|\cdot\|_E$ and fixing $T>0$, we write $C_TE=C([0,T];E)$ for the space of continuous functions from $[0,T]$ to $E$, equipped with the supremum norm $\|f\|_{C_TE}=\sup_{t\in[0,T]}\|f(t)\|_{E}$.  For $p\in [1,\infty]$ we write $L^p_TE=L^p([0,T];E)$ for the space of $L^p$-integrable functions from $[0,T]$ to $E$, equipped with the usual $L^p$-norm.
Let $H$ be a separable Hilbert space with  norm $\|\cdot \|_H$ and inner product $\la\cdot,\cdot\ra_H$. Given $p>1$, $\alpha\in (0,1)$, let $W^{\alpha,p}_TH$ be the Sobolev space of all $f\in L^p_TH$ such that 
 the following norm is finite
$$\|f\|_{W^{\alpha,p}_TH}^p\eqdef\int_0^T\|f(t)\|_H^p\dif t+\int_0^T\int_0^T\frac{\|f(t)-f(s)\|_H^p}{|t-s|^{1+\alpha p}}\dif t\dif s.$$
Let $\mathcal{S}'$ be the space of distributions on $\mathbb{T}^d$. 
We write $B^\alpha_{p,q}$ for Besov spaces on the torus with general indices $\alpha\in \R$, $p,q\in[1,\infty]$
and the H\"{o}lder--Besov  space $\bC^\alpha$ is given by $\bC^\alpha=B^\alpha_{\infty,\infty}$.  We will often write $\|\cdot\|_{\bC^\alpha}$ instead of $\|\cdot\|_{B^\alpha_{\infty,\infty}}$. For $\alpha\in\R$, set $H^\alpha=B^\alpha_{2,2}$. Set $\Lambda=(1-\Delta)^{\frac12}$. We put the definition of Besov spaces and useful lemmas in Appendix \ref{sec:pre}.


\subsection*{Acknowledgments}
We would like to thank Scott Smith
for the numerous and very helpful discussions on mean field limits and large $N$ problems of singular SPDEs.
 H.S. gratefully acknowledges financial support from NSF grants
 DMS-1712684 / 1909525,  DMS-1954091 and CAREER DMS-2044415.
R.Z. is grateful to the nancial supports of the NSFC (No. 11922103).
X.Z. is grateful to the nancial supports in part by National Key R\&D Program of China (No. 2020YFA0712700)
and the NSFC (No. 12288201, 12090014) and the support by key Lab of Random Complex Structures and Data
Science, Youth Innovation Promotion Association (2020003), Chinese Academy of Science.

\section{Stochastic terms and decomposition of the equation}
\label{sec:Sto}
{Due to the singular noise, the solutions to \eqref{eq:main} only belong to the negative order Besov-H\"older space and we need renormalization and 
decomposition to understand the nonlinearity. In this section, we first introduce the renormalizations and the relevant  stochastic terms for the decomposition of equation \eqref{eq:main}. We then derive uniform in $N$ estimates for the stochastic terms in 
Section~\ref{sec:sto} and Section~\ref{sec:unis}. In Section \ref{sec:dec} we give the decomposition of equation \eqref{eq:main}, which  gives the important cancellations for the nonlinearity.}

 We start with the  renormalized version of equation \eqref{eq:main}. Let $\xi_{i,\varepsilon}$ be a space-time mollification of $\xi_i$ defined on $\R\times \mathbb{T}^3$.
The formal equation \eqref{eq:main} is then interpreted as the limit of the following approximate equation
\begin{align}\label{eq:ap}
\LL \Phi_{i,\eps}+\frac{\l}{N}\sum_{j=1}^N\Phi_{j,\eps}^2\Phi_{i,\eps}
+\Big(-\frac{N+2}{N}\l a_\eps+\frac{3(N+2)}{N^2}\l^2 \tilde{b}_\eps\Big)\Phi_{i,\eps}=\xi_{i,\eps},
\end{align}
where $a_\eps$ and $\tilde{b}_\eps$ are renormalization constants given below.
Let  $Z_i$ be a stationary solution to
\begin{equation}\label{eq:li1}\LL Z_i=\xi_i.\end{equation}

\begin{remark}
To briefly motivate the $N$ dependent coefficients in the renormalization constants in \eqref{eq:ap}, note that we can ``recombine'' the terms as
$$\frac{\l}{N}\Big(\Phi_{i,\eps}^3 - 3a_\eps  \Phi_{i,\eps} +
\sum_{j\neq i} \Big(\Phi_{j,\eps}^2\Phi_{i,\eps}
-a_\eps  \Phi_{i,\eps}\Big)\Big)
$$
and this is then consistent with our renormalization \eqref{e:wick-tilde} below.
The  $N$ dependance of the coefficient in front of $\tilde{b}_\eps$ will also be clear
by similar consideration based on renormalization computation in paracontrolled calculus
 in Section~\ref{sec:ren} below.
\end{remark}

\subsection{Renormalization}\label{sec:ren}

We now introduce the renormalized terms.
Let $Z_{i,\varepsilon}$ be the stationary solution to $\LL {Z}_{i,\varepsilon}=\xi_{i,\varepsilon}$.
For convenience, we assume that all the noises are mollified with a common bump function. In particular, $ Z_{i,\eps}$ are i.i.d. mean zero Gaussian.
The Wick products are defined as
\begin{equ}[e:wick-tilde]
\cZ_{ij}^{\<2>}=
\begin{cases}
\lim\limits_{\varepsilon\to0}(Z_{i,\varepsilon}^2-a_\varepsilon)  &  (i=j)\\
 \lim\limits_{\varepsilon\to0} Z_{i,\varepsilon}Z_{j,\varepsilon} & (i\neq j)
\end{cases}
\qquad
\cZ_{ijj,\eps}^{\<3>} =
\begin{cases}
Z_{i,\varepsilon}^3-3a_\varepsilon Z_{i,\varepsilon}   & (i=j)\\
Z_{i,\varepsilon}Z_{j,\varepsilon}^2-a_\varepsilon Z_{i,\varepsilon} & (i\neq j)
\end{cases}
\end{equ}
 where $a_\varepsilon=\mathbf{E}[ Z_{i,\varepsilon}^2(0,0)]$ is a divergent constant and is independent of $i$.
 The limit in the definition of $\cZ_{ij}^{\<2>}$ is
  in $C_T\bC^{-1-\kappa}$ for $\kappa>0$ (see \cite[Section~10]{Hairer14} or \cite{MR3746744} for more details). Here and in the following $T\in(0,\infty)$ denotes an arbitrary finite time.
Let $\tilde{\cZ}_{ijj,\eps}^{\<30>}$ be the stationary solution to $\LL \tilde{\cZ}_{ijj,\eps}^{\<30>}=\cZ_{ijj,\eps}^{\<3>}$, i.e.  \footnote{Recall that $\cI$ is defined similarly but with $\int_{-\infty}^t$  replaced by $\int_0^t$.}
$$\tilde{\cZ}_{ijj,\eps}^{\<30>}=\int_{-\infty}^tP_{t-s}\cZ_{ijj,\eps}^{\<3>}(s)\dif s:=\tilde{\cI}\cZ_{ijj,\eps}^{\<3>}.$$
Set
$\tilde{\cZ}_{ijj}^{\<30>}=\lim\limits_{\varepsilon\to0}\tilde{\cZ}_{ijj,\eps}^{\<30>}$,
where the limit is in  $C_T\bC^{\frac12-\kappa}$ for $\kappa>0$.
Here in the ``tree'' type superscripts, a dot denotes a noise, a line denotes a heat kernel, and the
subscripts specify the component indices of the noises showing in the trees.
In particular we can write
 $\tilde{\cZ}_{ijj}^{\<30>}=\tilde{\cZ}_{jij}^{\<30>}=\tilde{\cZ}_{jji}^{\<30>}$.

To introduce further stochastic objects,
let $\sD \eqdef m-\Delta$ and define
\begin{equs}[2]
\tilde{\cZ}_{ij,k\ell,\eps}^{\<22>} &\eqdef
\sD^{-1}(\cZ_{ij,\eps}^{\<2>})\circ \cZ_{k\ell,\eps}^{\<2>}- c_1 \tilde{b}_\eps \;,
\qquad
& {\cZ}_{ij,k\ell,\eps}^{\<22>} \eqdef
\cI(\cZ_{ij,\eps}^{\<2>})\circ \cZ_{k\ell,\eps}^{\<2>}- c_1 b_\eps(t) \;,
\\
\tilde{\cZ}_{ijj,k,\eps}^{\<31>} &\eqdef \tilde{\cZ}_{ijj,\eps}^{\<30>}\circ Z_{k,\eps} \;,
\qquad
&\tilde{\cZ}_{ijj,ik,\eps}^{\<32>}\eqdef
\tilde{\cZ}_{ijj,\eps}^{\<30>}\circ \cZ_{ik,\eps}^{\<2>}- c_2 \tilde{b}_\eps Z_{j,\eps}\;.
\end{equs}
Here $c_1$ equals $\frac12$ if $i=k\neq j=\ell$  or $i=\ell\neq j=k$,
equals $1$ if $i=k= j=\ell$, and is $0$ otherwise;
and $c_2 $ equals $1$ if $j=k\neq i$, equals $3$ if $j=k= i$, and equals $0$ otherwise.
Also, $b_{\eps}(t)=\E[\cI(\cZ_{ii,\eps}^{\<2>})\circ \cZ_{ii,\eps}^{\<2>}]$ and  $\tilde{b}_{\eps}=\E[\sD^{-1}(\cZ_{ii,\eps}^{\<2>})\circ \cZ_{ii,\eps}^{\<2>}]$ are  renormalization constants,
and one has $|b_\eps-\tilde{b}_\eps|\lesssim  t^{-\gamma}$  for any $\gamma>0$ uniformly in $\eps$.
We denote collectively
$$\mathbb{Z}_\eps\eqdef(Z_{i,\eps},\cZ_{ij,\eps}^{\<2>},\tilde{\cZ}_{ijj,\eps}^{\<30>}, \tilde{\cZ}_{ijj,k,\eps}^{\<31>}, \tilde{\cZ}_{ij,k\ell,\eps}^{\<22>}, {\cZ}_{ij,k\ell,\eps}^{\<22>}, \tilde{\cZ}_{ijj,ik,\eps}^{\<32>}).$$

To state the next lemma,
for $\kappa>0$ we define the ``homogeneities'' $\alpha_\tau \in \mathbb{R}$
as in the following table
\renewcommand{\arraystretch}{1.8}
\begin{align*} \begin{array}{|c c c c c c c c|}
     \hline
     \tau &  Z_{i,\eps}&
     \cZ_{ij,\eps}^{\<2>} &\tilde{\cZ}_{ijj,\eps}^{\<30>} &\tilde{\cZ}_{ijj,k,\eps}^{\<31>}&\tilde{\cZ}_{ij,k\ell,\eps}^{\<22>}&{\cZ}_{ij,k\ell,\eps}^{\<22>}&\tilde{\cZ}_{ijj,ik,\eps}^{\<32>}
     \\
     \hline
    \alpha_\tau &  - \frac{1}{2}-\kappa & - 1-\kappa &  \frac{1}{2}-\kappa & -\kappa&-\kappa & -\kappa& - \frac{1}{2}-\kappa\\
     \hline
   \end{array} \end{align*}

 \bl\label{thm:renorm43}
  For every $\kappa,\sigma>0$ and some $0<\delta<{1/2}$,  there exist random distributions
 \begin{align}\label{e:defZ}
 \mathbb{Z}\eqdef(Z_{i},\cZ_{ij}^{\<2>},\tilde{\cZ}_{ijj}^{\<30>}, \tilde{\cZ}_{ijj,k}^{\<31>}, \tilde{\cZ}_{ij,k\ell}^{\<22>}, {\cZ}_{ij,k\ell}^{\<22>}, \tilde{\cZ}_{ijj,ik}^{\<32>})
 \end{align}
 such that
 if $\tau_{\varepsilon}$ is a component in $\mathbb{Z}_\varepsilon$
 and $\tau$ is
the corresponding component  in $\mathbb{Z}$  then
$\tau_{\varepsilon }\rightarrow \tau$ in $ C_T\bC^{\alpha_{\tau}}\cap C_T^{\delta/2}\CC^{\alpha_{\tau}-\delta}$ a.s. as $\varepsilon\to0$.
Furthermore, for every $p>1$
\begin{align*}
\sup_{m\geq1}\E\|\tau_\eps\|_{C_T\bC^{\alpha_{\tau}}}^p+\sup_{m\geq1}\E\|\tau_\eps\|_{C_T^{\delta/2}\CC^{\alpha_{\tau}-\delta}}^p
&\lesssim1,
\\
\sup_{m\geq1}\E\|\tau\|_{C_T\bC^{\alpha_{\tau}}}^p+\sup_{m\geq1}\E\|\tau\|_{C_T^{\delta/2}\CC^{\alpha_{\tau}-\delta}}^p
&\lesssim1,
\end{align*}
 where the proportional constants in the inequalities are independent of $\eps, i, j, N$.
\el

We emphasize that Wick renormalization is used for the terms in \eqref{e:wick-tilde} as in the 2D setting. But for the convergence of the higher order stochastic objects, we need to use the renormalization procedure  in \cite[Section 10]{Hairer14} or \cite{MR3846835,MR3746744}. We also refer the proof of  Lemma \ref{thm:renorm43} to the above references.

 {\bf Convention.}
In the following if we do not need the precise powers of the renormalized terms in $\mathbb{Z}$, we denote by $Q(\mathbb{Z})$ a generic polynomial in terms of the above norms of $\tau$ with $\E Q(\mathbb{Z})^q\lesssim 1$ for any $q\geq1$. We note that $Q(\mathbb{Z})$ may depend on $N$ and summations from $1$ to $N$, for instance, it could take the form  $\frac1N \sum_{i=1}^N\tau_i$,  but the expectation of $Q(\mathbb{Z})$ is uniformly bounded and independent of $N$.

By Lemma \ref{thm:renorm43} there exists a measurable  $\Omega_0\subset \Omega$ with $\mathbf{P}(\Omega_0)=1$ such that for $\omega\in \Omega_0$
$$\|\tau\|_{C_T\bC^{\alpha_{\tau}}}+\|\tau\|_{C_T^{\delta/2}\CC^{\alpha_{\tau}-\delta}}<\infty,$$
and
$\tau_{\varepsilon }\rightarrow \tau$ in $ C_T\bC^{\alpha_{\tau}}\cap C_T^{\delta/2}\CC^{\alpha_{\tau}-\delta}$  as $\varepsilon\to0$
for every $\tau\in \mathbb{Z}$. We fix such $\omega\in \Omega_0$ in the following. Now we fix $\kappa$ small enough and $\delta=1/2-2\kappa$ in the above estimate.

\subsection{Additional stochastic terms}\label{sec:sto}

Recall the Littlewood--Paley blocks $\Delta_j$
in Section~\ref{sec:pre}.  
Let
\begin{align}\label{zmm02}\UU_>\eqdef\sum_{j>L}\Delta_j,
\qquad \UU_\leq \eqdef \sum_{j\leq L}\Delta_j ,\end{align}
for some constant $L>0$ to be chosen below.
Note that the limit $\Phi_{i}$ of $\Phi_{i,\eps}$ has the same regularity as $Z_{i}$ which is $C_T\bC^{-\frac12-\kappa}$, and it would be natural to decompose
	$$\Phi_{i}=Z_{i}-\frac\lambda{N}\sum_{j=1}^N\tilde{\cZ}^{\<30>}_{ijj}+\zeta_{i}$$
	with $\zeta_i$ being a function of better regularity.
	 Writing down the dynamics for $\zeta_i$,
the most irregular terms would be the paraproduct $\cZ^{\<2>}_{ij}\succ (\Phi_k-Z_k)$,
which indicates that $\zeta_i \in\bC^{1-\kappa}$ since $\cZ^{\<2>}_{ij}\in \bC^{-1-\kappa}$ as in Lemma~\ref{thm:renorm43},
but then  in an energy estimate for $\zeta_{i}$, the term $\la\zeta_i,\cZ^{\<2>}_{ij}\succ (\Phi_k-Z_k)\ra$ cannot be controlled.

To overcome this difficulty, as in \cite{GH18} (for $N=1$ case) we introduce one more stochastic object:
\begin{align}\label{eq:211}
X_i=
-\frac{\l}{N}\sum_{j=1}^N\Big(
2\cI(X_j\prec \UU_> \cZ^{\<2>}_{ij})
+\cI(X_i\prec \UU_>  \cZ^{\<2>}_{jj})
+\tilde{\cZ}^{\<30>}_{ijj}\Big).
\end{align}
In the Section \ref{sec:dec}  we will see that after a further decomposition, $Y_i=\Phi_i-Z_i-X_i$ satisfies a suitable energy inequality.
For fixed $N$ and $L>0$,  fixed point argument and Lemma \ref{lem:para}  easily imply local well-posedness of  \eqref{eq:211} in $C_T\bC^{\frac12-2\kappa}$. Furthermore, by a suitable choice of $L$, we have the following uniform in $N$ estimates, which imply  global well-posedness.  Note that the right-hand sides of the following bounds only involve stochastic terms appearing in Lemma~\ref{thm:renorm43}.

\bl\label{lem:X}
There exists $L=L(\l, N)>0$ such that
\begin{equation}\label{bdx}
\frac{1}{N}\sum_{i=1}^N\|X_i\|_{C_T\bC^{\frac12-\kappa}}^2
\lesssim
\frac{\l^2}{N^2}\sum_{i,j=1}^N \|\tilde{\cZ}^{\<30>}_{ijj}\|^2_{C_T\bC^{\frac12-\kappa}},
\end{equation}
and
\begin{equs}
\|X_i\|_{C_T\bC^{\frac12-\kappa}}
&\lesssim R_N^0+\frac{\l}{N}\sum_{j=1}^N \|\tilde{\cZ}^{\<30>}_{ijj}\|_{C_T\bC^{\frac12-\kappa}},	\label{bdx1}
\\
\|X_i\|_{C^{\frac14-\kappa}_TL^\infty}
&\lesssim
R_N^0+\frac{\l}{N}\sum_{j=1}^N \|\tilde{\cZ}^{\<30>}_{ijj}\|_{C^{\frac14-\kappa}_TL^\infty}
+\frac{\l}{N}\sum_{j=1}^N \|\tilde{\cZ}^{\<30>}_{ijj}\|_{C_T\bC^{\frac12-\kappa}},	\label{bdx2}
\end{equs}
uniformly in $N$, $\lambda$ and $m\geq1$,
where
$$
R_N^0
\eqdef\frac{\l^2}{N^2}\sum_{k,j=1}^N \|\tilde{\cZ}^{\<30>}_{kjj}\|^2_{C_T\bC^{\frac12-\kappa}}
	+\frac{\l^2}{N}\sum_{j=1}^N\|\cZ^{\<2>}_{ij}\|_{C_T\bC^{-1-\kappa}}^2 .
$$
\el
%
%

 {\bf Convention.} Below we will often have expressions of the form $2A_j B_{i k l} + A_i B_{j k l}$.
To streamline our notation, we will write this as $A_{[j]} B_{[i] k l}$,
which stands for the same expression without the brackets multiplied by $2$, plus the expression with the two indices in brackets
swapped.
For instance, in \eqref{eq:211}, we can write
$$
\cI(X_{[j]}\prec \UU_> \cZ^{\<2>}_{[i]j})=
2\cI(X_j\prec \UU_> \cZ^{\<2>}_{ij})
+\cI(X_i\prec \UU_>  \cZ^{\<2>}_{jj})
$$
Note that when we prove inequalities later, often the factor $2$ does not matter since those inequalities will have an implicit proportional constant anyway.

\begin{proof}
The proof is motivated by \cite[Lemma~4.1]{GH18}.
By definitions of $\UU_>$ and Besov spaces,
\begin{align}\label{eq:uz}
\|\UU_>\cZ^{\<2>}_{ij}\|_{C_T\bC^{-3/2-\kappa}}\lesssim 2^{-L/2}\|\cZ^{\<2>}_{ij}\|_{C_T\bC^{-1-\kappa}}.
\end{align}
This combined with Schauder estimate Lemma \ref{lemma:sch} and Lemma \ref{lem:para} implies that
\begin{align}
\|X_i\|_{C_T\bC^{\frac12-\kappa}}
&\lesssim
2^{- L/2}\frac{\l}{N}
\sum_{j=1}^N
\Big(
\|X_{\qj}\|_{C_T\bC^{\frac12-\kappa}}\|\cZ^{\<2>}_{\qi j}\|_{C_T\bC^{-1-\kappa}}
\Big)
+\frac{\l}{N}\sum_{j=1}^N \|\tilde{\cZ}^{\<30>}_{ijj}\|_{C_T\bC^{\frac12-\kappa}}\nonumber
\\
&\lesssim 2^{-L/2}\Big(\frac{1}{N}\sum_{j=1}^N \|X_j\|^2_{C_T\bC^{\frac12-\kappa}}\Big)^{\frac12}\Big(\frac{\l^2}{N}\sum_{j=1}^N\|\cZ^{\<2>}_{ij}\|_{C_T\bC^{-1-\kappa}}^2\Big)^{\frac12}	\label{z1}
\\
&\qquad+2^{-L/2}\|X_i\|_{C_T\bC^{\frac12-\kappa}}\Big(\frac{\l}{N}\sum_{j=1}^N\|\cZ^{\<2>}_{jj}\|_{C_T\bC^{-1-\kappa}}\Big)
+\frac{\l}{N}\sum_{j=1}^N \|\tilde{\cZ}^{\<30>}_{ijj}\|_{C_T\bC^{\frac12-\kappa}}.
\nonumber
\end{align}
Taking  square  on  both sides and summing over $i$ then dividing by $N$  we deduce
\begin{align*}
\frac{1}{N}\sum_{i=1}^N\|X_i\|_{C_T\bC^{\frac12-\kappa}}^2
&\lesssim 2^{- L}\Big(\frac{1}{N}\sum_{j=1}^N \|X_j\|^2_{C_T\bC^{\frac12-\kappa}}\Big)\Big(\frac{\l^2}{N^2}\sum_{i,j=1}^N\|\cZ^{\<2>}_{ij}\|_{C_T\bC^{-1-\kappa}}^2\Big)
\\&+2^{- L}\Big(\frac{1}{N}\sum_{i=1}^N\|X_i\|^2_{C_T\bC^{\frac12-\kappa}}\Big)
\Big(\frac{\l}{N}\sum_{j=1}^N\|\cZ^{\<2>}_{jj}\|_{C_T\bC^{-1-\kappa}}\Big)^2
+\frac{\l^2}{N^2}\sum_{i,j=1}^N \|\tilde{\cZ}^{\<30>}_{ijj}\|^2_{C_T\bC^{\frac12-\kappa}}.
\end{align*}
Choosing $L=L(\lambda,N)$ such that
\begin{equation}\label{L}2^{ L}=2C^2\Big(\frac{\l^2}{N^2}\sum_{i,j=1}^N\|\cZ^{\<2>}_{ij}\|_{C_T\bC^{-1-\kappa}}^2\Big)
+2C^2\Big(\frac{\l}{N}\sum_{j=1}^N\|\cZ^{\<2>}_{jj}\|_{C_T\bC^{-1-\kappa}}\Big)^2+1,\end{equation}
with $C$ given as the maximum of proportional constant in \eqref{z1} and \eqref{eq:z2} below, we easily deduce \eqref{bdx}.  Using \eqref{z1} and \eqref{bdx}, \eqref{bdx1} follows.
\eqref{bdx2} follows from using Schauder estimate Lemma~\ref{lemma:sch}
  to bound $\|X_i\|_{C^{\frac14-\kappa}_TL^\infty}$ in the analogous way as \eqref{z1} and then applying \eqref{bdx} \eqref{bdx1} and \eqref{L}.
\end{proof}

By Lemma \ref{lem:X} we have $X_i\in C_T\bC^{\frac12-\kappa}$. We also introduce $X_{i,\eps}$ defined by \eqref{eq:211} with the elements in $\mathbb{Z}$ replaced by $\mathbb{Z}_\eps$. It is easy to deduce that
$$
X_{i,\eps}\to X_i \qquad\mbox{ in } C_T\bC^{\frac12-\kappa} \qquad P-\mbox{a.s.,}\qquad\mbox{as }\eps\to0.
$$
In our energy estimates below, we will have terms $X_j\circ \cZ^{\<2>}_{ij}$, $X_i\circ\cZ^{\<2>}_{jj}$ and $X_j\circ Z_i$ which are not classical well-defined.
We will use the renormalized terms introduced in Section \ref{sec:ren} to define them.
We first consider $X_j\circ \cZ^{\<2>}_{ij}$.
 For $i\neq j$,  note that  $2\cZ^{\<22>}_{jl,ij}$ requires renormalization when $i=l$, 
 and  $\cZ^{\<22>}_{ll,ij}$ does not. Hence, for $i\neq j$
 \begin{equs}   \label{sto1}
 X_j &\circ \cZ^{\<2>}_{ij}
 \eqdef  \lim_{\eps\to0}(X_{j,\eps}\circ \cZ^{\<2>}_{ij,\eps}+\frac{\l\tilde{b}_\eps}N(Z_{i,\eps}+X_{i,\eps}))
 \\
 &=-\frac{\l}{N} \sum_{l=1}^N
  \lim_{\eps\to0}
  \bigg[
  \tilde{\cZ}^{\<32>}_{llj,ij,\eps}
+\cI( X_{\ql,\eps}\prec \cZ^{\<2>}_{\qj l,\eps})\circ \cZ^{\<2>}_{ij,\eps}
-\cI(X_{\ql,\eps}\prec \UU_\leq(\cZ^{\<2>}_{\qj l,\eps})  )\circ \cZ^{\<2>}_{ij,\eps}-\frac{\tilde{b}_\eps}N X_{i,\eps}\bigg]
\\
&=-\frac{\l}{N} \sum_{l=1}^N \bigg[\tilde{\cZ}^{\<32>}_{llj,ij}
+X_{\ql}\cZ^{\<22>}_{\qj l,ij}
+\bar{C}(X_{\ql},\cZ^{\<2>}_{\qj l},\cZ^{\<2>}_{ij})
-\cI(2X_{\ql} \prec \UU_\leq\cZ^{\<2>}_{\qj l} )\circ \cZ^{\<2>}_{ij}\bigg]+\frac{\l}{N}(\tilde{b}-b(t))X_i
\end{equs}
with
  \begin{align}\label{c:b}
  \tilde{b}-b(t)=\lim_{\eps\to0}(\tilde{b}_\eps-b_\eps(t)),
 \qquad
  |\tilde{b}-b(t)|\lesssim t^{-\gamma}
  \end{align}
for $t>0$ and $\gamma>0$.
Here $\bar{C}$ is the commutator introduced in Lemma \ref{lem:com2}.
 Using
  Lemma \ref{lem:para}, Lemma \ref{lem:com2} and Lemma \ref{thm:renorm43} the limit in  \eqref{sto1} is understood in $C((0,T];\bC^{-\frac12-\kappa})$ $P$-a.s..
For  $i=j$, we have similar decomposition as in \eqref{sto1} with $\tilde{b}_\eps$ replaced by $3 \tilde{b}_\eps$,
with limits also in $C((0,T]; \bC^{-\frac12-\kappa})$ $P$-a.s..
($2\cZ^{\<22>}_{jl,ij}$ gives an extra $2\tilde{b}$,
 and $\cZ^{\<22>}_{ll,ij}$ gives an extra $\tilde{b}$, in the case $i=j=l$.)

Similarly  for $i\neq j$ we define, again with limit  in $C((0,T]; \bC^{-\frac12-\kappa})$ $P$-a.s.
 \begin{equs}\label{sto2}
{}& X_i   \circ \cZ^{\<2>}_{jj}
 \eqdef \lim_{\eps\to0}(X_{i,\eps}\circ \cZ^{\<2>}_{jj,\eps}+\frac{\l\tilde{b}_\eps}N(Z_{i,\eps}+X_{i,\eps}))
\\
& =-\frac{\l}{N} \sum_{l=1}^N \bigg[\tilde{\cZ}^{\<32>}_{ill,jj}
+X_{\ql}\cZ^{\<22>}_{\qi l,jj}
+\bar{C}(X_{\ql},\cZ^{\<2>}_{\qi l},\cZ^{\<2>}_{jj})
-\cI(  X_{\ql}\prec \UU_\leq\cZ^{\<2>}_{\qi l} )\circ \cZ^{\<2>}_{jj}\bigg]
+\frac{\l}{N}(\tilde{b}-b(t))X_i\;.
\end{equs}
%
Finally we define, with the limit taken in $C_T\bC^{-\kappa}$ $P$-a.s.
\begin{equ}\label{sto3}
X_j\circ Z_i
=\lim_{\eps\to0}(X_{j,\eps}\circ Z_{i,\eps})
=-\frac{\l}{N} \sum_{l=1}^N
\bigg[\tilde{\cZ}^{\<31>}_{llj,i}
+\cI(X_{\ql}\prec \UU_>\cZ^{\<2>}_{\qj l})  \circ Z_i\bigg] \;.
\end{equ}

 In the following lemma we derive uniform bounds for the above stochastic terms. Regarding $X_jZ_i$ we have $X_jZ_i=X_j\prec Z_i+X_j\circ Z_i+X_j\succ Z_i$.

\bl\label{lem:X1} It holds that
\begin{equs}
\frac{1}{N^2}\sum_{i,j=1}^N\|X_jZ_i\|_{C_T\bC^{-\frac12-\kappa}}^2&\lesssim \l^2(1+\l^2) Q^{52}_N,
\\
\frac{1}{N}\sum_{j=1}^N\|X_jZ_j\|_{C_T\bC^{-\frac12-\kappa}}^2&\lesssim \l^2(1+\l^2) Q^{53}_N,
\end{equs}
where $Q^{52}_N, Q^{53}_N$ are given in the proof with $\E Q^{52}_N+\E Q^{53}_N\lesssim1$ 
uniformly in $N$ and $\l$.
\el

\begin{proof}
By \eqref{sto3} and Lemmas \ref{lemma:sch},  \ref{lem:para} we obtain
\begin{equ}
\|X_j\circ Z_i\|_{C_T\bC^{-\kappa}}
\lesssim
\frac{\l}{N} \sum_{l=1}^N \bigg[\|\tilde{\cZ}^{\<31>}_{llj,i}\|_{C_T\bC^{-\kappa}}
+\|X_{\ql}\|_{C_T\bC^{\frac12-\kappa}}\|\cZ^{\<2>}_{\qj l}\|_{C_T\bC^{-1-\kappa}}\| Z_i\|_{C_T\bC^{-\frac12-\kappa}}
\bigg]. \label{e:XcircZ}
\end{equ}
Moreover, by Lemma \ref{lem:para} we have
\begin{align*}
\|X_j\prec Z_i\|_{C_T\bC^{-\frac12-\kappa}}+\|X_j\succ Z_i\|_{C_T\bC^{-\frac12-\kappa}}\lesssim \|X_j\|_{C_T\bC^{\frac12-\kappa}}\|Z_i\|_{C_T\bC^{-\frac12-\kappa}}.
\end{align*}

By \eqref{bdx} we have
\begin{align*}
&\frac{1}{N^2}\sum_{i,j=1}^N\|X_jZ_i\|^2_{C_T\bC^{-\frac12-\kappa}}
\lesssim
\l^2(1+\l^2) \bigg[
\frac{1}{N^3}\sum_{i,j,l=1}^N\|Z_i\|_{C_T\bC^{-\frac{1}{2}-\kappa}}^2\|\tilde{\cZ}^{\<30>}_{ljj}\|_{C_T\bC^{\frac{1}{2}-2\kappa}}^2
\\&
+\frac{1}{N^3}\sum_{i,j,l=1}^N\|\tilde{\cZ}^{\<31>}_{llj,i}\|_{C_T\bC^{-\kappa}}^2
+\frac{1}{N^4}\sum_{i,j,k,l=1}^N\|\tilde{\cZ}^{\<30>}_{ljj}\|_{C_T\bC^{\frac{1}{2}-2\kappa}}^2\|\cZ^{\<2>}_{kk}\|_{C_T\bC^{-1-\kappa}}^2
\|Z_i\|_{C_T\bC^{-\frac{1}{2}-\kappa}}^2
\\&
+\frac{1}{N^5}\sum_{i,j,k,l,m=1}^N\|\tilde{\cZ}^{\<30>}_{ljj}\|_{C_T\bC^{\frac{1}{2}-2\kappa}}^2\|\cZ^{\<2>}_{km}\|_{C_T\bC^{-1-\kappa}}^2
\|Z_i\|_{C_T\bC^{-\frac{1}{2}-\kappa}}^2
\bigg]
\eqdef \l^2(1+\l^2)Q_N^{52},
\end{align*}
where the first line corresponds to $X_j\prec Z_i+X_j\succ Z_i$, and the second and third line are estimate for $X_j\circ Z_i$.

Moreover, we have the paraproduct decomposition
$X_jZ_j=X_j\prec Z_j+X_j\succ Z_j+X_j\circ Z_j$.
Using Lemma~\ref{lem:para},
we can bound  the first two terms by
 $\|X_j\|_{C_T\bC^{\frac{1}{2}-2\kappa}}\|Z_j\|_{C_T\bC^{-\frac{1}{2}-\kappa}}$.
By \eqref{bdx1} we obtain  
\begin{align*}
&\frac{1}{N}\sum_{j=1}^N \|X_j\|_{C_T\bC^{\frac{1}{2}-\kappa}}^2 \|Z_j\|_{C_T\bC^{-\frac{1}{2}-\kappa}}^2
\\& \lesssim
\frac{1}{N}\sum_{j=1}^N \Big(
\frac{\l^2}{N^2}\sum_{i,l=1}^N \|\tilde{\cZ}^{\<30>}_{ill}\|^2_{C_T\bC^{\frac12-\kappa}}
	+\frac{\l^2}{N}\sum_{l =1}^N\|\cZ^{\<2>}_{j l}\|_{C_T\bC^{-1-\kappa}}^2
\Big)^2 \|Z_j\|_{C_T\bC^{-\frac{1}{2}-\kappa}}^2
\\&
\qquad\qquad +\frac{\l^2}{N^2}\sum_{i,j=1}^N\|\tilde{\cZ}^{\<30>}_{ijj}\|_{C_T\bC^{\frac{1}{2}-2\kappa}}^2\|Z_i\|_{C_T\bC^{-\frac{1}{2}-\kappa}}^2
\lesssim \l^2(1+\l^2)Q_N^6,\end{align*}
with  \label{page Q6}
\begin{align*}
Q_N^6
&\eqdef\frac{1}{N^2}\sum_{i,j=1}^N\|\tilde{\cZ}^{\<30>}_{ijj}\|_{C_T\bC^{\frac{1}{2}-2\kappa}}^2\|Z_i\|_{C_T\bC^{-\frac{1}{2}-\kappa}}^2+\frac{1}{N^2}\sum_{l,j=1}^N \|\cZ^{\<2>}_{jl}\|_{C_T\bC^{-1-\kappa}}^4\|Z_j\|^2_{C_T\bC^{-\frac{1}{2}-\kappa}}
\\
&\qquad +\frac{1}{N^3}\sum_{i,j,l=1}^N\|\tilde{\cZ}^{\<30>}_{ill}\|_{C_T\bC^{\frac{1}{2}-2\kappa}}^4\|Z_j\|^2_{C_T\bC^{-\frac{1}{2}-\kappa}}.
\end{align*}
Using \eqref{sto3} and \eqref{e:XcircZ}  with $i=j$ 
we deduce
\begin{align*}
\frac{1}{N}\sum_{j=1}^N\|X_jZ_j\|^2_{C_T\bC^{-\frac12-\kappa}}
 & \lesssim 
 \l^2(1+\l^2)
 \bigg[
 Q_N^6\Big(1+\frac{1}{N}\sum_{l=1}^N \|\cZ^{\<2>}_{ll}\|_{C_T\bC^{-1-\kappa}}^2\Big)
+\frac{1}{N^2}\Big(\sum_{j,l=1}^N\|\tilde{\cZ}^{\<31>}_{jll,j}\|_{C_T\bC^{-\kappa}}^2\Big)
\\
&+\Big(\frac{1}{N^2}\sum_{i,j=1}^N\|\tilde{\cZ}^{\<30>}_{ijj}\|_{C_T\bC^{\frac{1}{2}-2\kappa}}^2\Big)\Big(\frac{1}{N^2}\sum_{l,k=1}^N\|\cZ^{\<2>}_{kl}\|_{C_T\bC^{-1-\kappa}}^2\|Z_k\|^2_{C_T\bC^{-\frac{1}{2}-\kappa}}\Big)
\bigg]
\\
& \eqdef \l^2(1+\l^2)Q_N^{53}.
\end{align*}
The result follows by the above estimates.
\end{proof}

\subsection{Improved uniform in $N$ estimates for stochastic terms}\label{sec:unis}

In this section we derive uniform in $N$ estimates for the stochastic terms introduced in Section \ref{sec:sto},
which shows that one obtains ``improved estimates'' by gaining  ``factors of $1/N$''.

Recall that
for mean-zero independent random variables $U_{1},\dots,U_{N}$ taking values in a Hilbert space $H$, we have
\begin{equation}
\E \Big \| \sum_{i=1}^{N}U_{i} \Big \|_{H}^{2} =\E \sum_{i=1}^{N}\|U_{i}\|_{H}^{2}.\label{eq:Ui}
\end{equation}
This simple fact is important for us since the square of the sum on
the LHS of \eqref{eq:Ui} appears to have ``$N^2$ terms'' but under expectation it's only a sum of $N$ terms, in a certain sense giving us a ``factor of $1/N$''.
This motivates us to derive the following uniform in $N$ estimate in suitable Hilbert spaces.
We first prove the following result for renormalization terms
$\cZ_{ijj}^{\<30>}$ and $\cZ_{ljj,ij}^{\<32>}$.

\bl\label{lem:sto} Set
 \begin{align*}
 Q^0_N\eqdef\frac{1}{N^2}\sum_{i=1}^N \Big\|\sum_{j=1}^N\tilde{\cZ}^{\<30>}_{ijj}\Big\|^2_{L_T^2H^{\frac12-2\kappa}},\quad
 Q_N^1\eqdef\frac{1}{N^2}\sum_{i=1}^N \Big\|\sum_{j=1}^N \tilde{\cZ}^{\<30>}_{ijj}\Big\|_{W_T^{\frac14-2\kappa,2}L^2}^2,
\end{align*}
\begin{align*}
Q_N^2\eqdef\frac{1}{N^2} \sum_{i=1}^N\bigg(\sum_{j=1}^N\frac{1}{N} \Big\|\sum_{l=1}^N \tilde{\cZ}^{\<32>}_{llj,ij}\Big\|_{L^2_TH^{-\frac12-2\kappa}}\bigg)^2.
\end{align*}
One has  $\E| Q_N^i|^q\lesssim1$
 for every $q\geq 1$ and $i=0,1,2$,
uniformly in $N$ and $m$.
\el

\begin{proof}
	Since we will have several similar calculations in the sequel, we first demonstrate
	such calculation in the case $q=1$. We have for $s=\frac12-2\kappa$
	\begin{align*}
	\E\frac{1}{N^2} \sum_{i=1}^{N}  \Big\|\sum_{j=1}^{N}\Lambda^{s}\tilde{\cZ}^{\<30>}_{ijj}\Big\|^2_{L_T^2L^2}
	=
	\frac{1}{N^2}\sum_{i,j_1,j_2=1}^N\E \Big\la\Lambda^{s} \tilde{\cZ}^{\<30>}_{ij_1j_1},\Lambda^{s} \tilde{\cZ}^{\<30>}_{ij_2j_2}\Big\ra_{L^2_TL^2},
	\end{align*}
where $\Lambda^s=(1-\Delta)^{\frac{s}2}$ is introduced in Section  \ref{sub:1} and we used permutation invariance in law to drop the sum over $i$.
	We have $3$ summation indices and a factor $1/N^2$.
	The contribution to the sum from the cases  $j_1=i$ or $j_2=i$ or  $j_1=j_2$
	is bounded by a constant in light of Lemma \ref{thm:renorm43}. If $i,j_1,j_2$ are all different, by independence
	and the fact that Wick products are mean zero, the terms are zero.

	For general $q\ge 1$,
	by Gaussian hypercontractivity and the fact that $Q_N^i$ is a random variable in finite Wiener chaos, we have for $i=0,1,2$
	\begin{align*}\E[(Q_N^i)^q]\lesssim\E[(Q_N^i)^2]^{q/2}.\end{align*}
So it suffices to consider $q=2$.
 We  write $ \E[(Q_N^0)^2]$ as
	\begin{align*}
	\frac{1}{N^4}\sum_{\substack{i_1,i_2,j_k=1\\ k=1\dots 4}}^N\E \Big\la\Lambda^{s} \tilde{\cZ}_{i_1j_1j_1}^{\<30>} ,\Lambda^{s} \tilde{\cZ}_{i_1j_2j_2}^{\<30>}\Big\ra_{L^2_TL^2}
	\Big\la \Lambda^{s}\tilde{\cZ}_{i_2j_3j_3}^{\<30>},\Lambda^{s} \tilde{\cZ}_{i_2j_4j_4}^{\<30>}\Big\ra_{L^2_TL^2},
	\end{align*}
	We have $6$ indices $i_1, i_2, j_1, ... ,j_4$ summing from $1$ to $N$ and an overall factor $1/{N^4}$. Using again Lemma \ref{thm:renorm43}, we reduce the problem to the cases where five or six of the indices are different.  However, in these two cases, at least one of $j_k$ is different from others. Then by independence the expectation is zero, so $\E[(Q_N^0)^2]\lesssim1$.

	Since $W^{\frac14-2\kappa,2}_TL^2$ is a Hilbert space with inner product given by
	$$\la f,g\ra_{W_T^{\frac14-2\kappa,2}L^2}=\la f,g\ra_{L^2_TL^2}+\int_0^T\int_0^T\frac{\la f(t)-f(r),g(t)-g(r)\ra_{L^2}}{|t-r|^{\frac32-4\kappa}}\dif t\dif r,$$
	 we deduce $\E|Q_N^1|^q\lesssim1$
	 by similar arguments.
	
	In the following we consider $\E|Q_N^2|^2$. First note that $Q_N^2$ is bounded by
	\begin{align*}
	\frac{1}{N^3} \sum_{i,j=1}^N \Big\|\sum_{l=1}^N \tilde{\cZ}^{\<32>}_{llj,ij}\Big\|_{L^2_TH^{-\frac12-2\kappa}}^2,
	\end{align*}
	which implies that $\E|Q_N^2|^2$ is bounded by
\begin{align*}
	\frac{1}{N^6}
	\sum_{\substack{i,j,i_1,j_1,l_k=1\\k=1\dots 4}}^N \E\Big\la \Lambda^{s}\tilde{\cZ}^{\<32>}_{l_1l_1j,ij},\Lambda^s\tilde{\cZ}^{\<32>}_{l_2l_2j,ij}\Big\ra_{L^2_TL^2}
	\Big\la \Lambda^{s}\tilde{\cZ}^{\<32>}_{l_3l_3j_1,i_1j_1},\Lambda^s\tilde{\cZ}^{\<32>}_{l_4l_4j_1,i_1j_1}\Big\ra_{L^2_TL^2},
\end{align*}
	for $s=-\frac12-2\kappa$. 	We have $8$ indices $i, i_1, j, j_1, l_1, ...,l_4$ summing from $1$ to $N$ and an overall factor $1/{N^6}$. Using again Lemma \ref{thm:renorm43}, we reduce the problem to the cases where seven or eight of the indices are different.  However, in these two cases at least one of $l_k$ is different from others. Then by independence the expectation is zero, so $\E[(Q_N^2)^2]\lesssim1$.
	\end{proof}

With the help of Lemma~\ref{lem:sto} we also have the following bounds for $X$,
which states that summing $N$ terms of suitable Hilbert norms of $X_i$ actually ``behaves like order 1''.

\bl\label{lem:X2} The following bounds hold
with proportional constants independent of $N$ and $\l$
\begin{equs}
\sum_{i=1}^N\|X_i\|_{L_T^2H^{\frac12-2\kappa}}^2 & \lesssim \l^2 Q_N^0,
\label{bdnx}
\\
\sum_{i=1}^N\|X_i\|_{W_T^{\frac14-3\kappa,2}L^2}^2 & \lesssim \l^2 Q(\mathbb{Z}),
\label{bdnx1}
\end{equs}
where $Q_N^0$ satisfies the moment bounds  in Lemma~\ref{lem:sto}, and
 $(\E Q(\mathbb{Z})^q)^{1/q}\lesssim1+\l^2$ for any $q\geq1$.
\el
\begin{proof}
The start-point of the proof is similar as that of Lemma~\ref{lem:X}, but now it's crucial that we estimate $X_i$ in a Hilbert space.
Using the Schauder estimate in Lemma~\ref{lem:z1} now (instead of Lemma~\ref{lemma:sch}) and Lemma \ref{lem:para} and \eqref{eq:uz} we have
\begin{equs}[eq:z2]
\|X_i  \|_{L_T^2H^{\frac12-2\kappa}}
& \lesssim
2^{- L/2}\frac{\l}{N}
\sum_{j=1}^N \Big(
\|X_{\qj}\|_{L_T^2H^{\frac12-2\kappa}}\|{\cZ}^{\<2>}_{\qi j}\|_{C_T\bC^{-1-\kappa}}
\Big)
 +\Big\|\frac{\l}{N}\sum_{j=1}^N \tilde{\cZ}^{\<30>}_{ijj}\Big\|_{L_T^2H^{\frac12-2\kappa}}
\\
&\lesssim
2^{- L/2}\bigg(\sum_{j=1}^N \|X_j\|^2_{L_T^2H^{\frac12-2\kappa}}\bigg)^{\frac12}\bigg(\frac{\l^2}{N^2}\sum_{j=1}^N\|{\cZ}^{\<2>}_{ij}\|_{C_T\bC^{-1-\kappa}}^2\bigg)^{\frac12}				
\\
&\quad+2^{-L/2}\|X_i\|_{L_T^2H^{\frac12-2\kappa}}\bigg(\frac{\l}{N}\sum_{j=1}^N\|{\cZ}^{\<2>}_{jj}\|_{C_T\bC^{-1-\kappa}}\bigg)
+\frac{\l}{N}\Big\|\sum_{j=1}^N \tilde{\cZ}^{\<30>}_{ijj}\Big\|_{L_T^2H^{\frac12-2\kappa}}.
\end{equs}
Note that a key difference between the bound here and the proof of Lemma~\ref{lem:X}
is that
by considering $\|\frac{\l}{N}\sum_{j=1}^N \tilde{\cZ}^{\<30>}_{ijj}\|_{L_T^2H^{\frac12-2\kappa}}$ instead of $\frac{\l}{N}\sum_{j=1}^N\| \tilde{\cZ}^{\<30>}_{ijj}\|_{L_T^2H^{\frac12-2\kappa}}$, we gain ``a factor of $1/N$'' by Lemma \ref{lem:sto} and the discussion before Lemma \ref{lem:sto}.
Then taking  square on  both sides and summing over $i$ we obtain
\begin{align*}
\sum_{i=1}^N\|X_i\|_{L_T^2H^{\frac12-2\kappa}}^2
&\lesssim 2^{- L}\bigg(\sum_{j=1}^N \|X_j\|^2_{L_T^2H^{\frac12-2\kappa}}\bigg)\bigg(\frac{\l^2}{N^2}\sum_{i,j=1}^N\|{\cZ}^{\<2>}_{ij}\|_{C_T\bC^{-1-\kappa}}^2\bigg)
\\
&+2^{- L}\bigg(\sum_{i=1}^N\|X_i\|^2_{L_T^2H^{\frac12-2\kappa}}\bigg)\bigg(\frac{\l}{N}\sum_{j=1}^N\|{\cZ}^{\<2>}_{jj}\|_{C_T\bC^{-1-\kappa}}\bigg)^2
+\frac{\l^2}{N^2}\sum_{i=1}^N \Big\|\sum_{j=1}^N \tilde{\cZ}^{\<30>}_{ijj}\Big\|^2_{L_T^2H^{\frac12-2\kappa}}.
\end{align*}
By the choice of $2^{L}$ in \eqref{L}, \eqref{bdnx} follows.

To show \eqref{bdnx1},
again by Lemma \ref{lem:z1} and Lemma \ref{lem:para} we have
\begin{align*}
 \|X_i\|_{W_T^{\frac14-3\kappa,2}L^2}^2
\lesssim
\frac{\lambda^2}{N^2} \Big\|\sum_{j=1}^N \tilde{\cZ}^{\<30>}_{ijj}\Big\|_{W_T^{\frac14-3\kappa,2}L^2}^2
+\Big(\frac{\lambda}{N}\sum_{j=1}^N
\|X_{\qj}\|_{L_T^2H^{\frac12-2\kappa}}\|{\cZ}^{\<2>}_{\qi j}\|_{C_T\bC^{-1-\kappa}}
\Big)^2,
\end{align*}
which combined with \eqref{bdnx} and H\"older's inequality for  the sum over $j$ implies that
\begin{align*}
\sum_{i=1}^N\|X_i\|_{W_T^{\frac14-3\kappa,2}L^2}^2
& \lesssim   \l^2 Q_N^1+\bigg(\sum_{j=1}^N \|X_j\|^2_{L_T^2H^{\frac12-2\kappa}}\bigg)\bigg(\frac{\l^2}{N^2}\sum_{i,j=1}^N\|{\cZ}^{\<2>}_{ij}\|^2_{C_T\bC^{-1-\kappa}}\bigg)
\\
&\qquad\qquad +\bigg(\sum_{i=1}^N\|X_i\|^2_{L_T^2H^{\frac12-\kappa}}\bigg)\bigg(\frac{\l}{N}\sum_{j=1}^N \|{\cZ}^{\<2>}_{jj}\|_{C_T\bC^{-1-\kappa}}\bigg)^2
\end{align*}
where $Q_N^0, Q_N^1$ are introduced in Lemma \ref{lem:sto}. 
Using \eqref{bdnx} with Lemma~\ref{thm:renorm43} and Lemma~\ref{lem:sto},   \eqref{bdnx1} follows.
\end{proof}

In the following lemma  we deduce the estimates in a suitable Hilbert space for  the  renormalized terms introduced  in \eqref{sto1}, \eqref{sto2} and \eqref{sto3} before Lemma~\ref{lem:X1}.

\bl\label{lem:X3} The following bounds hold
with $\E Q(\mathbb{Z})^q\lesssim 1+\l^4$  uniformly in $N$ and $\l$ for all $q\geq1$ 
\begin{equs}
\frac{1}{N^2} \sum_{i=1}^N\Big(\sum_{j=1}^N\| \; \bullet \; \|_{L_T^2H^{-\frac12-2\kappa}}\Big)^2  &\leq \l^4 Q(\mathbb{Z}),
\qquad
\bullet \in \{X^2_jZ_i  \;,\; X_jX_iZ_j\} \;,
\\
\frac{1}{N^2}\sum_{i=1}^N\Big(\sum_{j=1}^N\| \;\bullet\; \|_{L^2_TH^{-\frac12-2\kappa}}\Big)^2  &\leq \l^2 Q(\mathbb{Z}),
\qquad
\bullet \in \{ X_i\circ {\CZ}^{\<2>}_{jj} \;,\; \,X_j\circ {\cZ}^{\<2>}_{ij}\} \;.
\end{equs}
\el

\begin{proof} 
	 From the proof of Lemma~\ref{lem:X1} we have
	\begin{align}\label{eq:l1}
	\frac{1}{N^2}\sum_{i,j=1}^N\|X_j\circ Z_i\|^2_{C_T\bC^{-\kappa}}\lesssim \l^2(1+\l^2) Q(\mathbb{Z}),
	\end{align}
with $Q(\mathbb{Z})$ satisfying $\E Q(\mathbb{Z})^q\lesssim 1$ for any $q\geq1$.
	Now we consider $X_j^2Z_i$ and use paraproduct to have the following decomposition
\begin{align}
X^2_j\circ Z_i &=2(X_j\prec X_j)\circ Z_i+(X_j\circ X_j)\circ Z_i\nonumber
\\&=2X_j(X_j\circ Z_i)+2\tilde{C}(X_j,X_j,Z_i)+(X_j\circ X_j)\circ Z_i,\label{sto4}
\end{align}
where $\tilde{C}$ is introduced in Lemma \ref{lem:com1}.
By Lemma \ref{lem:para} and Lemma \ref{lem:com1} we have
	\begin{align*}
	\|X^2_j\circ Z_i\|_{L_T^2H^{-1/2-2\kappa}}&\lesssim \|X_j\|_{L_T^2H^{\frac12-2\kappa}}\|X_j\circ Z_i\|_{C_T\bC^{-\kappa}}
	\\
	& \;+\|X_j\|_{L_T^2H^{\frac12-2\kappa}}\|X_j\|_{C_T\bC^{\frac12-\kappa}}\|Z_i\|_{C_T\bC^{-\frac12-\kappa}},
	\end{align*}
	and
\begin{align*}
	\|X_j^2\prec Z_i\|_{L_T^2H^{-\frac12-2\kappa}}  +\|X^2_j\succ Z_i\|_{L_T^2H^{-\frac12-2\kappa}}
	\lesssim\|X_j\|_{L_T^2H^{\frac12-2\kappa}}\|X_j\|_{C_T\bC^{\frac12-\kappa}}\|Z_i\|_{C_T\bC^{-\frac12-\kappa}}.
\end{align*}
Thus H\"{o}lder's inequality implies that
\begin{align*}
	\frac{1}{N^2} \sum_{i=1}^N\Big(\sum_{j=1}^N 
	&  \|X^2_jZ_i\|_{L_T^2H^{-\frac12-2\kappa}}\Big)^2
	\lesssim \frac{1}{N^2} \sum_{i=1}^N\Big(\sum_{j=1}^N\|X_j\|^2_{L_T^2H^{\frac12-2\kappa}}\Big)\Big(\sum_{j=1}^N\|X_j\circ Z_i\|^2_{C_T\bC^{-\kappa}}\Big)
	\\ &+  \frac{1}{N^2} \sum_{i=1}^N\Big(\sum_{j=1}^N\|X_j\|^2_{L_T^2H^{\frac12-2\kappa}}\Big)\Big(\sum_{j=1}^N\|X_j\|^2_{C_T\bC^{\frac12-\kappa}}\Big)\|Z_i\|^2_{C_T\bC^{-\frac12-\kappa}}.
\end{align*}
	Then by \eqref{eq:l1}, Lemma~\ref{lem:X}, Lemma~\ref{lem:X2}  and Lemma \ref{lem:sto} we obtain the  bound for $X^2_jZ_i$. The bound for $X_iX_jZ_j$ follows in the same way.
	
Moreover, by the decomposition for  $X_j\circ {\cZ}^{\<2>}_{ij}$ in \eqref{sto1} 
and   Lemmas~\ref{lem:z1}, \ref{lem:para} and \ref{lem:com},
\begin{align*}
\|X_j & \circ {\cZ}^{\<2>}_{ij}\|_{L^2_TH^{-\frac12-2\kappa}}
 \lesssim
\frac{\l}{N} \Big\|\sum_{l=1}^N \tilde{\cZ}^{\<32>}_{llj,ij}\Big\|_{L^2_TH^{-\frac12-2\kappa}}
+\frac{\l}{N}\|(\tilde{b}-b)X_i\|_{L^2_TH^{-\frac12-2\kappa}}
\\
& +\frac{\l}{N}\sum_{l=1}^N\bigg[
\|X_{\ql}\|_{L_T^2H^{\frac12-2\kappa}}\|{\cZ}^{\<22>}_{\qj l,ij}\|_{C_T\bC^{-\kappa}}
\\
&+ \Big(  
(1+2^{3\kappa L}) \|X_{\ql}\|_{L^2_TH^{\frac12-2\kappa}}  \|{\cZ}^{\<2>}_{\qj l}\|_{C_T\bC^{-1-\kappa}}
+\|X_{\ql}\|_{W^{\frac14-3\kappa,2}L^2} \|{\cZ}^{\<2>}_{\qj l}\|_{C_T\bC^{-1-\kappa}}
\Big) 
\|{\cZ}^{\<2>}_{ij}\|_{C_T\bC^{-1-\kappa}}
	\bigg]
\end{align*}
	where we used \eqref{L} and
	$$\|\UU_\leq {\cZ}^{\<2>}_{jl}\|_{\bC^{-1+2\kappa}}+\|\UU_\leq {\cZ}^{\<2>}_{ll}\|_{\bC^{-1+2\kappa}}\lesssim 2^{3\kappa L}(\|{\cZ}^{\<2>}_{jl}\|_{\bC^{-1-\kappa}}+\|{\cZ}^{\<2>}_{ll}\|_{\bC^{-1-\kappa}}).$$
Moreover, using \eqref{c:b} we get $$\frac{1}{N^2}\sum_{i=1}^N\Big(\frac1N\sum_{j=1}^N\frac{\l}{N}\|(\tilde{b}-b)X_i\|_{L^2_TH^{-\frac12-2\kappa}}\Big)^2\lesssim \frac{\l^2}{N^2}\sum_{i=1}^N\|X_i\|_{C_T\bC^{-\frac12-\kappa}}^2.$$
	Thus by Lemma \ref{lem:X2} and Lemma \ref{lem:X} we have
	\begin{align*}&\frac{1}{N^2} \sum_{i=1}^N\bigg(\sum_{j=1}^N\|X_j\circ {\cZ}^{\<2>}_{ij}\|_{L^2_TH^{-\frac12-2\kappa}}\bigg)^2
	\lesssim \l^2Q^2_N
	+\l^4Q(\mathbb{Z}),
	\end{align*}
	with $Q(\mathbb{Z})$ satisfying $(\E Q(\mathbb{Z})^q)^{1/q}\lesssim1+\l^2$ which combined with Lemma \ref{lem:sto} implies the bound for $X_j\circ {\cZ}^{\<2>}_{ij}$.
	The last bound regarding $X_i\circ {\CZ}^{\<2>}_{jj} $ then follows in the same way.
\end{proof}

\subsection{Decomposition}\label{sec:dec}

In this section we consider the following equation
\begin{equation}\label{eq:21}\LL \Phi_i
=-\frac{\l}{N}\sum_{j=1}^N\Phi_j^2\Phi_i+\xi_i,
\quad \Phi_i(0) 
\in \bC^{-\frac12-\kappa},
\end{equation}
for $\kappa>0$. For fixed $N$ and $\omega\in \Omega_0$ by using regularity structure theory \cite{Hairer14} or paracontrolled distribution method \cite{GIP15}, we easily deduce the local well-posedness of \eqref{eq:21} in $C_T\bC^{-\frac12-\kappa}$ and $\Phi_i$ is the limit of $\Phi_{i,\eps}$, which is the unique  solution to equation \eqref{eq:ap}. Furthermore, by similar arguments as in \cite{MW18, GubCMP19} global well-posedness also holds by uniform estimates, which may depend on $N$. In the following we concentrate on {\it the uniform in $N$} bounds.

With the stochastic objects at hand, we have the following decompositions: $\Phi_{i}=Z_i+X_i+Y_i$ with $Y_i$ satisfying the following equation  \footnote{This decomposition follows \cite{GH18} which uses different notation than $Z+X+Y$ here. Our choice of this notation is close to our 2D paper \cite{SSZZ20}.} 
\begin{equation}\label{eq:22}\aligned
\LL Y_i
=&-\frac{\l}{N}\sum_{j=1}^N\bigg(Y_j^2Y_i+(X_j^2+2X_jY_j)(X_i+Y_i)+Y_j^2X_i+(X_j+Y_j)^2Z_i
\\&+2(X_j+Y_j)(X_i+Y_i)Z_j+2X_j\prec \UU_\leq\cZ^{\<2>}_{ij} +X_i\prec \UU_\leq\cZ^{\<2>}_{jj}
\\&+2Y_j\prec\cZ^{\<2>}_{ij}+Y_i\prec\cZ^{\<2>}_{jj}+2(X_j+Y_j)\succcurlyeq\cZ^{\<2>}_{ij}+(X_i+Y_i)\succcurlyeq\cZ^{\<2>}_{jj}\bigg) \;,
\\Y_i(0)=&\,\Phi_i(0)-Z_i(0)-X_i(0) \;.
\endaligned\end{equation}
Here the first line in \eqref{eq:22} are the expansion of $(Y_j+X_j)^2(Y_i+X_i+Z_i)$; the terms containing $\cZ^{\<2>}_{ij}$ and $\cZ^{\<2>}_{jj}$ correspond to the remaining terms in
 the paraproduct expansion of $2(Y_j+X_j)Z_jZ_i$ and
$(Y_i+X_i)Z_j^2$, respectively.

Also, note that some products in \eqref{eq:22} are understood via renormalization.
Namely, $X_j^2Z_i$, $X_jZ_i$, $X_jX_iZ_j$, $X_j\circ \cZ^{\<2>}_{ij}$
and $X_i\circ \cZ^{\<2>}_{jj}$ are understood using \eqref{sto1}-\eqref{sto3}, \eqref{sto4} and Lemmas \ref{lem:X1}+\ref{lem:X3}. Since $\cZ^{\<2>}_{ij}\in C_T\bC^{-1-\kappa}$,
the expected regularity of $Y_i$ is $C_T\bC^{1-\kappa}$; so in view of Lemma \ref{lem:para} $Y_j\circ \cZ^{\<2>}_{ij}$ and $Y_i\circ \cZ^{\<2>}_{jj}$ are not well-defined in the classical sense and we need to use the renormalization terms $\cZ^{\<22>}_{ij,kl}$ to define these terms, which from the approximation level requires to subtract $\frac{3\l(N+2)}{N^2} \tilde{b}_\eps Y_{i,\eps}$ on the R.H.S of \eqref{eq:22}, where $Y_{i,\eps}=\Phi_{i,\eps}-X_{i,\eps}-Z_{i,\eps}$.

In the following we establish $L^2$-energy estimate for $Y_i$,
 and as explained in the introduction we follow the idea in \cite{GH18} to use the duality between $\prec$ and $\circ$, i.e.
Lemma~\ref{lem:com3} to cancel  $Y_j\circ \cZ^{\<2>}_{ij}$ and $Y_i\circ \cZ^{\<2>}_{jj}$ which would require paracontrolled ansatz and higher regularity estimate than $H^1$. 

Recall that $\mathscr{D}=m-\Delta$ and we define
\begin{equation}\label{phi}
\varphi_i\eqdef Y_i+\mathscr{D}^{-1}\frac{\l}{N}\sum_{j=1}^N
 \big(Y_{\qj} \prec\cZ^{\<2>}_{\qi j}\big)
\eqdef
Y_i+\mathscr{D}^{-1}P_i,
\end{equation}
where in the last step we defined $P_i$.

Now we turn to uniform in $N$ bounds on \eqref{eq:22} and note that $Y_i$ depends on $N$, but we omit this throughout.
Similar as the 2d case in \cite{SSZZ20}, we do $L^2$-energy estimate of $Y_i$ and take sum over $i$.
Using \eqref{phi}, we have the following decomposition.

\bl\label{den}
(Energy balance)
\begin{align}
\frac{1}{2} \sum_{i=1}^{N}\frac{\dif}{\dif t}\|Y_{i}\|_{L^{2}}^{2}+m\sum_{i=1}^{N}\| \varphi_{i}\|_{L^{2}}^{2}+\sum_{i=1}^{N}\|\nabla \varphi_{i}\|_{L^{2}}^{2}+\frac{\l}{N}\Big \|\sum_{i=1}^{N}Y_{i}^{2} \Big \|_{L^{2}}^{2}=\Theta+\Xi. \nonumber
\end{align}
Here
\begin{align*}
\Theta=
\sum_{i=1}^N \langle \mathscr{D}^{-1}P_i,P_i\rangle
-\frac{\l}{N}\sum_{i,j=1}^N \Big(2D(Y_i,{\cZ}^{\<2>}_{ij},Y_j)+D(Y_i,{\cZ}^{\<2>}_{jj},Y_i)\Big),
\end{align*}
and
\begin{align*}
\Xi=&-\frac{\l}{N}\sum_{i,j=1}^N\bigg\langle (X_j^2+2X_jY_j)(X_i+Y_i)+Y_j^2X_i+(X_j+Y_j)^2 Z_i
\\
&+2(X_j+Y_j)(X_i+Y_i)Z_j+2X_j\prec \UU_\leq {\cZ}^{\<2>}_{ij} +X_i\prec \UU_\leq {\cZ}^{\<2>}_{jj}
\\&+2(X_j+Y_j)\succ {{\cZ}^{\<2>}_{ij}}+2X_j\circ {{\cZ}^{\<2>}_{ij}}+(X_i+Y_i)\succ {\cZ}^{\<2>}_{jj}+X_i\circ {\cZ}^{\<2>}_{jj} \;\; ,\;\; Y_i\bigg\rangle.
\end{align*}

\el

\begin{proof}
We will first focus on a formal derivation for the claimed identity.
Then we will remark that although new renormalization appears to be necessary
when we take inner product,  in the end they cancel each other.

Taking inner product with $Y_i$ in $L^2$ on  \eqref{eq:22},
 we realize that the first term on the R.H.S. of \eqref{eq:22} leads to the term $-\frac{\l}N \|\sum_{i=1}^{N}Y_{i}^{2}  \|_{L^{2}}^{2}$, and it is straightforward to check that  the other terms all lead to   $\Xi$ except
  the following terms
 \begin{align}
 \label{zmm1}
-\frac{\l}{N}\sum_{i,j=1}^N\langle 2Y_j\preccurlyeq\cZ^{\<2>}_{ij}+Y_i\preccurlyeq\cZ^{\<2>}_{jj},Y_i\rangle.
 \end{align}
 We claim that we can write  \eqref{zmm1} plus $\sum_{i=1}^N\langle (\Delta-m) Y_i,Y_i\rangle$ as $\Theta-m\sum_{i=1}^{N}\| \varphi_{i}\|_{L^{2}}^{2}-\sum_{i=1}^{N}\|\nabla \varphi_{i}\|_{L^{2}}^{2}$ (see Eq.~\eqref{zmm8}) and we will prove this claim for the rest of this proof.

Using \eqref{phi} we write it as
\begin{align}\label{zmm2}
\langle (\Delta-m)Y_i,Y_i\rangle=\langle (\Delta-m) \varphi_i,\varphi_i\rangle+2\langle Y_i,P_i\rangle+\langle \mathscr{D}^{-1}P_i,P_i\rangle. \end{align}
 We will realize below that $2\langle Y_i,P_i\rangle$ cancels the irregular part  (i.e. the  paraproduct $\preccurlyeq$ part) in \eqref{zmm1}.

For the term in \eqref{zmm1}, by \eqref{phi} we have
\begin{equ}[zmm7]
-\frac{\l}{N}\sum_{j=1}^N\langle 2Y_j\preccurlyeq\cZ^{\<2>}_{ij}+Y_i\preccurlyeq\cZ^{\<2>}_{jj},Y_i\rangle
=
-\langle Y_i,P_i\rangle
-\frac{\l}{N}\sum_{j=1}^N\langle 2Y_j\circ\cZ^{\<2>}_{ij}+Y_i\circ\cZ^{\<2>}_{jj},Y_i\rangle
\end{equ}
where we note that  the first term $-\langle Y_i,P_i\rangle$ on the R.H.S.
 precisely cancels one $\langle Y_i,P_i\rangle$ on the R.H.S. of \eqref{zmm2}. Using Lemma \ref{lem:com3}, the other terms in \eqref{zmm7} containing $\circ$ can be written as
\begin{align}
\langle Y_i,2Y_j\circ\cZ^{\<2>}_{ij}\rangle &=2\langle Y_i\prec \cZ^{\<2>}_{ij},Y_j\rangle+2D(Y_i,\cZ^{\<2>}_{ij},Y_j),
		\label{com1}
\\
\langle Y_i,Y_i \circ \cZ^{\<2>}_{jj} \rangle
&=\langle Y_i\prec \cZ^{\<2>}_{jj},Y_i\rangle
	+D(Y_i,\cZ^{\<2>}_{jj},Y_i),		\label{com2}
\end{align}
where $D(f,g,h)$ is the commutator introduced in  Lemma \ref{lem:com3}.
The first terms on the R.H.S. of \eqref{com1} and \eqref{com2}  cancel the other $\langle Y_i,P_i\rangle$ from \eqref{zmm2} when taking sum.

We sum all the terms in \eqref{zmm2} and \eqref{zmm7} w.r.t. $i$ and obtain the following:
\begin{align}
&\sum_{i=1}^N\bigg[\langle (\Delta-m) Y_i,Y_i\rangle-\frac{\l}{N}\sum_{j=1}^N\langle 2Y_j\preccurlyeq\cZ^{\<2>}_{ij}+Y_i\preccurlyeq\cZ^{\<2>}_{jj},Y_i\rangle\bigg]   \label{zmm8}
\\=&\sum_{i=1}^N\big[\langle (\Delta-m) \varphi_i,\varphi_i\rangle+\langle \mathscr{D}^{-1}P_i,P_i\rangle\big]-\frac{\l}{N}\sum_{i,j=1}^N\big(2D(Y_i,\cZ^{\<2>}_{ij},Y_j)+D(Y_i,\cZ^{\<2>}_{jj},Y_i)\big).   \notag
\end{align}
This completes the derivation of the claimed identity.

\vspace{1ex}

Finally, we  remark that several terms in the above derivation should be understood in the renormalized sense.
As we have noticed above, the term $\sum_{i=1}^N 2 \la P_i,Y_i\ra$ cancels with the irregular part from  \eqref{zmm1}, so no extra renormalization is needed for them.

For the last term in \eqref{zmm2} we have
\begin{equation}\label{eq:P1}
\aligned
\langle \mathscr{D}^{-1}P_i,P_i\rangle
&=\frac{\l}{N}\sum_{j=1}^N\bigg[\langle (\mathscr{D}^{-1}P_i) \circ\cZ^{\<2>}_{ij},2Y_j\rangle-2D(Y_j,\cZ^{\<2>}_{ij},\mathscr{D}^{-1}P_i)
\\
&\qquad+\langle (\mathscr{D}^{-1}P_i)\circ\cZ^{\<2>}_{jj}, Y_i\rangle-D(Y_i,\cZ^{\<2>}_{jj},\mathscr{D}^{-1}P_i)\bigg].
\endaligned
\end{equation}
Using the definition of $\mathscr{D}^{-1}P_i$ in \eqref{phi} and the commutator $C$ introduced in Lemma \ref{lem:com2} we can write the above term as
\begin{equs}
\frac{\l}{N}\sum_{j=1}^N&\bigg(\frac{\l}{N}\sum_{l=1}^N\bigg[\langle \tilde{\cZ}^{\<22>}_{il,ij}
,4Y_lY_j\rangle+4\langle C(Y_l,\cZ^{\<2>}_{il},\cZ^{\<2>}_{ij}),Y_j\rangle
+\langle \tilde{\cZ}^{\<22>}_{ll,ij}
,2Y_iY_j\rangle+2\langle C(Y_i,\cZ^{\<2>}_{ll},\cZ^{\<2>}_{ij}),Y_j\rangle
\\
&+\langle \tilde{\cZ}^{\<22>}_{il,jj}
,2Y_lY_i\rangle+2\langle C(Y_l,\cZ^{\<2>}_{il},\cZ^{\<2>}_{jj}),Y_i\rangle
+\langle \tilde{\cZ}^{\<22>}_{ll,jj}
,Y_i^2\rangle+\langle C(Y_i,\cZ^{\<2>}_{ll},\cZ^{\<2>}_{jj}),Y_i\rangle\bigg]
\\
&-2D(Y_j,\cZ^{\<2>}_{ij},\mathscr{D}^{-1}P_i)-D(Y_i,\cZ^{\<2>}_{jj},\mathscr{D}^{-1}P_i)\bigg).		\label{eqPi}
\end{equs}
Here -- recall the renormalization of $\tilde{\cZ}^{\<22>}$ in Section~\ref{sec:ren} -- we need to subtract
\begin{equ}[e:N+8]
\frac{\l^2(N+8)}{N^2} \tilde{b}_\eps \langle Y_{i,\eps},Y_{i,\eps}\rangle+\frac{2\l^2}{N^2}\sum_{j\neq i} \tilde{b}_\eps \langle Y_{j,\eps},Y_{j,\eps}\rangle
\end{equ}
 from  the approximation level to go from \eqref{eq:P1} to \eqref{eqPi}.
 This precisely matches the renormalization from the SPDE \eqref{eq:ap}:
 indeed, summing \eqref{e:N+8} over $i$, we get $\frac{3N+6}{N^2} \l^2 \tilde{b}_\eps $ times  $\sum_i \langle Y_{i,\eps},Y_{i,\eps}\rangle$.

We thus conclude that
although new renormalization appears to be necessary when we
take inner product, in the end no extra renormalization other than the ones introduced at the level of the SPDE
is actually needed.
\end{proof}

\section{Uniform in $N$ estimates}\label{sec:uni}

In this section we prove uniform in $N$ estimates based on  Lemma \ref{den}.
The main results are Theorems~\ref{Y:T4} and \ref{Y:T4a}.
The key step to prove these theorems
 is to bound $\int_0^T(\Theta+\Xi)\dif t$ by
\begin{equation}\label{zmm4}
\aligned&\delta\Big(\sum_{j=1}^N\|\nabla  \varphi_{j}\|_{L_T^2L^2}^2+\sum_{j=1}^N\| Y_{j}\|_{L_T^2H^{1-2\kappa}}^2+\frac{\l}{N}\Big\|\sum_{i=1}^NY_{i}^2\Big\|_{L^2_TL^2}^2\Big)
\\&+C_\delta\l(1+\l^{55})\int_0^T\Big(\sum_{i=1}^N\|Y_i\|_{L^2}^2\Big)(R_N^1+R_N^2+Q_N^3)\dif s+\l(1+\l) Q_N^4
\endaligned\end{equation}
for a small constant $\delta>0$
with $R_N^1, R_N^2, Q_N^3, Q_N^4$ introduced in Lemma~\ref{lem:zz2}, Propositions \ref{theta1} and \ref{Xi1}. 
 \eqref{zmm4} will immediately follow from  Propositions \ref{theta1} and \ref{Xi1} in Section \ref{sec:4.2}. The finite moments of $R_N^1, R_N^2, Q_N^3, Q_N^4$ are bounded uniformly in $N$, which follows from  Lemma \ref{thm:renorm43} and  Lemma~\ref{lem:zmm}, \eqref{e:Q4}.

\subsection{Uniform estimates based on \eqref{zmm4}}
We prove \eqref{zmm4} in Section \ref{sec:4.2}. 
In this subsection we prove Theorem \ref{Y:T4} and Theorem \ref{Y:T4a}, assuming \eqref{zmm4}.
Before this we first prove the following two results. The first one is used to turn $\|\varphi_i\|_{L^2}^2$ on the LHS of the identity in Lemma~\ref{den} to
$\|Y_i\|_{L^2}^2$. The second one  gives uniform in $N$ estimates of various norms of  $Y_{i}, \varphi_i, \sD^{-1}P_i$ in terms of \eqref{zmm4} by using \eqref{phi}.

\bl\label{lem:zz2}
The following bound holds with $C$ independent of  $\l$ and $m\geq1$
\begin{align*}
m\sum_{i=1}^{N}\| \varphi_{i}\|_{L^{2}}^{2}
\geq
\Big(\frac{m}{2}-C\l^{\frac{2}{1-2\kappa}} (R_N^1)^{\frac1{1-2\kappa}} \Big)\sum_{i=1}^{N}\| Y_{i}\|_{L^{2}}^{2}
	+m\sum_{i=1}^{N}\|\mathscr{D}^{-1}P_i\|_{L^2}^2,
\end{align*}
with
$$
R_N^1
\eqdef \frac{1}{N^2}\sum_{i,j=1}^N\|{{\cZ}^{\<2>}_{ij}}\|_{\bC^{-1-\kappa}}^{2}
+\frac{1}{N}\sum_{j=1}^N\| {\cZ}^{\<2>}_{jj}\|_{\bC^{-1-\kappa}}^{2}+ 1\;.
$$
\el
\begin{proof} By  definition of $\varphi_i$ in \eqref{phi},
\begin{align}\label{eq:zmm}
m\sum_{i=1}^{N}\| \varphi_{i}\|_{L^{2}}^{2}=m\sum_{i=1}^{N}\| Y_{i}\|_{L^{2}}^{2}-2m\sum_{i=1}^{N}\la Y_i,\mathscr{D}^{-1}P_i\ra +m\sum_{i=1}^{N}\|\mathscr{D}^{-1}P_i\|_{L^2}^2.
\end{align}
It remains to control the second term on the RHS.
Note that for $\kappa>0$
\begin{equs}[zmm3]
&\|m\mathscr{D}^{-1}f\|_{L^2}^2=\sum_k\frac{m^2}{(m+|k|^2)^2}|\hat{f}(k)|^2
\\
&\leq m^{1+\kappa}\sum_k\frac{1}{(m+|k|^2)^{1+\kappa}}|\hat{f}(k)|^2
\leq
m^{1+\kappa}\|f\|_{H^{-1-\kappa}}^2.
\end{equs}
By \eqref{zmm3} and Lemma \ref{lem:para} we have
\begin{align*}
&\sum_{i=1}^N|\la Y_i,m\mathscr{D}^{-1}P_i\ra |\leq \sum_{i=1}^N\|Y_i\|_{L^2}\|m\mathscr{D}^{-1}P_i\|_{L^2}
\\
&\leq
m^{\frac{1+2\kappa}{2}}\sum_{i=1}^N\|Y_i\|_{L^2}   \bigg\|\frac{\l}{N}\sum_{j=1}^N
(Y_{\qj}\prec{\cZ}^{\<2>}_{\qi j} )\bigg\|_{H^{-1-2\kappa}}
\\
&\lesssim
m^{\frac{1+2\kappa}{2}}\frac{\l}{N}\sum_{i,j=1}^N\|Y_i\|_{L^2}
\|Y_{\qj}\|_{L^2}\|{\cZ}^{\<2>}_{\qi j}\|_{\bC^{-1-\kappa}}
\\
&\leq m^{\frac{1+2\kappa}{2}}\Big(\sum_{i=1}^N\|Y_i\|_{L^2}^2\Big)\bigg[\bigg(\frac{\l^2}{N^2}\sum_{i,j=1}^N\|{\cZ}^{\<2>}_{ij}\|_{\bC^{-1-\kappa}}^2\bigg)^{\frac12}
+\Big(\frac{\l}{N}\sum_{j=1}^N\| {\CZ}^{\<2>}_{jj}\|_{\bC^{-1-\kappa}}\Big)\bigg]
\\
&\leq \sum_{i=1}^N\|Y_i\|_{L^2}^2\bigg[\frac{m}{4}+C\bigg(\frac{\l^2}{N^2}\sum_{i,j=1}^N\|{\cZ}^{\<2>}_{ij}\|_{\bC^{-1-\kappa}}^2\bigg)^{\frac{1}{1-2\kappa}}
+ C\bigg(\frac{\l }{N}\sum_{j=1}^N\| {\CZ}^{\<2>}_{jj}\|_{\bC^{-1-\kappa}}\bigg)^{\frac{2}{1-2\kappa}}\bigg]
\end{align*}
where we used Young's inequality in the last step. Now the result follows.
\end{proof}

The following estimates will be useful in the sequel. Recall $R_N^1$ from Lemma \ref{lem:zz2}.

\bl\label{lem:phi1} The following holds
with $C$ independent of  $\l$ and $m\geq1$ 
\footnote{In the first bound, we keep the explicit constant $2$
for the purpose of the proof of Theorem~\ref{Y:T4};
in the third bound, we keep the explicit constant $2$ in order to derive  the condition for $m$ and $\l$ later.}
\begin{align}
\sum_{i=1}^N\|Y_i\|^2_{H^{1-2\kappa}}
& \leq
2\Big(\sum_{i=1}^{N}\| \varphi_{i}\|_{H^1}^{2}\Big)+ C \l^2  \Big(\sum_{i=1}^{N}\|Y_{i}\|^2_{L^2}\Big)R_N^1,  
\label{e:3.2.1}
\\
\sum_{i=1}^N\|\mathscr{D}^{-1}P_i\|^2_{H^{1-2\kappa}}
& \leq  C\l^2 \Big(\sum_{i=1}^{N}\|Y_{i}\|^2_{L^2}\Big)R_N^1,
\label{eq:zz3}
\\
\sum_{i=1}^N\|\varphi_i\|^2_{L^2}
&\leq
\Big(\sum_{i=1}^{N}\|Y_{i}\|^2_{L^2}\Big)(C\l^2R_N^1+2),\notag
\\
\frac{1}{N}\sum_{i,j=1}^N\|\varphi_i\varphi_j\|^2_{L^2}
&\leq   \frac{C}{N}\Big \|\sum_{i=1}^{N}Y_{i}^{2}\Big\|^2_{L^2}(\l^2 R_N^1 + 1)^2.\notag
\end{align}
\el
\begin{proof}
	Recalling  the relation $\varphi_i  = Y_i+\mathscr{D}^{-1}P_i$ in the  definition \eqref{phi}, we will see that we essentially only need to estimate $\|\mathscr{D}^{-1}P_i\|_{H^{1-2\kappa}}$ for the first three inequalities  in the lemma.
Since $\|\sD f\|_{H^\beta}\simeq\|f\|_{H^{\beta+2}}$ for $\beta\in\mR$,  by
Lemma \ref{lem:para} we have
	\begin{align}
	&\|\mathscr{D}^{-1}P_i\|_{H^{1-2\kappa}}
	\leq\frac{\l C}{N}\sum_{j=1}^{N} \Big(\|Y_{j}\|_{L^2}\|{\cZ}^{\<2>}_{ij}\|_{\bC^{-1-\kappa}}+\|Y_i\|_{L^2}\|\cZ^{\<2>}_{jj}\|_{\bC^{-1-\kappa}}\Big)\label{eq:Pi}
	\\
	&\leq C\Big(\sum_{j=1}^{N}\|Y_{j}\|^2_{L^2}\Big)^{\frac12}\Big(\frac{\l^2}{N^2}\sum_{j=1}^N\|{\cZ}^{\<2>}_{ij}\|^2_{\bC^{-1-\kappa}}\Big)^{\frac12}
	+C\|Y_i\|_{L^2}\Big(\frac{\l}{N}\sum_{j=1}^N\|\cZ^{\<2>}_{jj}\|_{\bC^{-1-\kappa}}\Big).\notag
\end{align}
The second bound for $\mathscr{D}^{-1}P_i$ follows from taking square on both sides of  \eqref{eq:Pi} and summing over $i$; this together with
$$
\|Y_i\|_{H^{1-2\kappa}}\leq \| \varphi_{i}\|_{H^1}+\|\mathscr{D}^{-1}P_i\|_{H^{1-2\kappa}}$$
also yields the first bound.
	The third bound for
	 $\varphi_i$ follows  by $\|\varphi_i\|_{L^2} \le \|Y_i\|_{L^2}+\|\mathscr{D}^{-1}P_i\|_{H^{1-2\kappa}}$
	  and simply plugging in the above bound on $\mathscr{D}^{-1}P_i$.

	Moreover, by the bound \eqref{eq:zz3} for $\mathscr{D}^{-1}P_i$ we have
\begin{align*}
	\frac{1}{N}\sum_{i,j=1}^N\|\varphi_i\varphi_j\|^2_{L^2}
	&=\frac{1}{N}\Big \|\sum_{i=1}^{N}\varphi_{i}^{2}\Big\|^2_{L^2}
	\leq \frac{C}{N}\Big \|\sum_{i=1}^{N}Y_{i}^{2}\Big\|^2_{L^2}+\frac{C}{N}\sum_{i,j=1}^N\|\mathscr{D}^{-1}P_i\|_{L^4}^2\|\mathscr{D}^{-1}P_j\|^2_{L^4}
	\\
	&\leq \frac{C}{N}\Big \|\sum_{i=1}^{N}Y_{i}^{2}\Big\|^2_{L^2}(1+\l^2 R_N^1)^2,
\end{align*}
	where we used Sobolev embedding  $\|f\|_{L^4}\lesssim \|f\|_{H^{1-2\kappa}}$ for $\kappa>0$ small enough and
	\begin{equation}\label{yi}\Big\|\sum_{i=1}^{N}Y_{i}^{2}\Big\|^2_{L^2}\geq \Big(\sum_{i=1}^{N}\|Y_{i}\|^2_{L^2}\Big)^2.\end{equation}
\end{proof}

The main result of this section is given as follows, which will be used in Section~\ref{sec:Con}.

\bt\label{Y:T4} The following holds with constant $C$ independent of  $N$, $\l$ and $m\geq1$:
\begin{align*}
&\Big(\sum_{j=1}^N\|Y_{j}(T)\|_{L^2}^2\Big)+\frac12\sum_{j=1}^N\|\nabla  \varphi_{j}\|_{L_T^2L^2}^2+m\sum_{j=1}^N\|Y_{j}\|_{L_T^2L^2}^2
\\
&\qquad +\frac{1}{8}\sum_{j=1}^N\| Y_{j}\|_{L_T^2H^{1-2\kappa}}^2+\frac{\l}{N}\Big\|\sum_{i=1}^NY_{i}^2\Big\|_{L^2_TL^2}^2
\\
& \leq \Big(\sum_{j=1}^N\|Y_j(0)\|_{L^2}^2\Big)+\l(1+\lambda) Q_N^4+\sum_{j=1}^N\|Y_{j}\|_{L_T^2L^2}^2
\\&\qquad+C\l(1+\l^{55})\int_0^T\Big(\sum_{i=1}^N\|Y_i\|_{L^2}^2\Big)(R_N^1+R_N^2+Q_N^3)\dif s\;.
\end{align*}
\et


\begin{proof}
Integrating the energy equality in Lemma~\ref{den} over time,
	and using Lemma \ref{lem:zz2} and \eqref{zmm4}, $Q_N^3\geq (R_N^1)^{\frac1{1-2\kappa}}$, we deduce
	\begin{align*}
	&\Big(\sum_{j=1}^N\|Y_{j}(T)\|_{L^2}^2\Big)+2\sum_{j=1}^N\|\nabla  \varphi_{j}\|_{L_T^2L^2}^2+m\sum_{j=1}^N\|Y_{j}\|_{L_T^2L^2}^2
	+\frac{2\lambda}{N}\Big\|\sum_{i=1}^NY_{i}^2\Big\|_{L^2_TL^2}^2
	\\
	& \leq \delta\Big(\sum_{j=1}^N\|\nabla  \varphi_{j}\|_{L_T^2L^2}^2+\sum_{j=1}^N\| Y_{j}\|_{L_T^2H^{1-2\kappa}}^2+\frac{\lambda}{N}\Big\|\sum_{i=1}^NY_{i}^2\Big\|_{L^2_TL^2}^2\Big)
	\\&\qquad+ \Big(\sum_{j=1}^N\|Y_j(0)\|_{L^2}^2\Big)+C_\delta\l(1+\l^{55})\int_0^T\Big(\sum_{i=1}^N\|Y_i\|_{L^2}^2\Big)(R_N^1+R_N^2+Q_N^3)\dif s
\\&\qquad+\l(1+\l) Q_N^4,
	\end{align*}
for some $\delta>0$ small enough.
Applying \eqref{e:3.2.1}  
to $\sum_{j=1}^N\| Y_{j}\|_{L_T^2H^{1-2\kappa}}^2$
 the result follows.
\end{proof}

Furthermore, using the dissipation effect from the term $\frac{1}{N}\|\sum_{i=1}^NY_{i}^2\|_{L^2_TL^2}^2$,  the empirical averages of the $L^2$ norms of $Y_i$ can be controlled pathwise in terms of the averages of the renormalized terms $Q(\mathbb{Z})$ with finite moment, as stated in the following theorem (which will be used in Sec~\ref{sec:Tightness}).

\bt\label{Y:T4a} The following bound holds
with  $\E Q(\mathbb{Z})\lesssim C(\l)$ where $C(\l)$ is independent of $N$
\minilab{e:3.4}
\begin{equs}
\sup_{t\in[0,T]}\frac{1}{N}\sum_{j=1}^N\|Y_{j}(t)\|_{L^2}^2
&+\frac{m}{N}\sum_{j=1}^N\| Y_{j}\|_{L_T^2L^2}^2
+\frac{1}{2N}\sum_{j=1}^N\|\nabla  \varphi_{j}\|_{L^2_TL^2}^2
+\frac{1}{8N}\sum_{j=1}^N\| Y_{j}\|_{L_T^2H^{1-2\kappa}}^2
			\label{e:3.4.1}
\\
&
+\frac{\l}{2N^2}\Big\|\sum_{i=1}^NY_{i}^2\Big\|_{L^2_TL^2}^2
\;
\leq 
Q(\mathbb{Z})+\frac{2}{N}\sum_{i=1}^N\|Y_i(0)\|_{L^2}^2.
			\label{e:3.4.2}
\end{equs}
\et
\begin{proof}
Similarly as in the proof of Theorem \ref{Y:T4} we have
\begin{align*}
&\eqref{e:3.4.1}
+\frac{\lambda}{N^2}\Big\|\sum_{i=1}^NY_{i}^2\Big\|_{L^2_TL^2}^2
 \leq 
\frac2N\sum_{j=1}^N\|Y_j(0)\|_{L^2}^2
+\l(1+\l) Q_N^4
\\
&+C\l(1+\l^{55})\int_0^T\Big(\frac1N\sum_{i=1}^N\|Y_i\|_{L^2}^2\Big)(R_N^1+R_N^2+Q_N^3)\dif s
+\frac1N\sum_{j=1}^N\|Y_{j}\|_{L_T^2L^2}^2
\end{align*}
where the last term is bounded by
 $\frac{\l}{4N^2}\|\sum_{i=1}^NY_{i}^2\|_{L_T^2L^2}^2+C(\l)$.
Using
$$\frac{1}{N^2}\Big\|\sum_{i=1}^NY_{i}^2\Big\|_{L^2}^2\geq \Big(\frac{1}{N}\sum_{i=1}^N\|Y_i\|_{L^2}^2\Big)^2,$$
Young's inequality and $\E (Q_N^3)^2\lesssim C(\l)$ from Lemma \ref{lem:zmm} below, $\E(R_N^1)^2+\E(R_N^2)^2\lesssim1$ from Lemma \ref{thm:renorm43}, and moment bound  \eqref{e:Q4} for $Q^4_N$ below, the result follows.
\end{proof}

\subsection{Proof of \eqref{zmm4}}\label{sec:4.2}
 We first consider the easier part $\Theta$ defined in Lemma \ref{den}.
\bp\label{theta1} It holds  for $\delta>0$ small that
\begin{align}\label{zmm5}
|\Theta|\leq 
\delta \sum_{i=1}^{N}\|Y_{i}\|_{H^{1-2\kappa}}^{2}+C\l(1+\l^4)\Big(\sum_{i=1}^N\|Y_i\|^{2}_{L^2}\Big)R_N^2.
\end{align}
Here $C$ is independent of $\l, N$ and $m\geq1$, and for $\theta=\frac{1+2\kappa}{2-4\kappa}$ we define
$$
\aligned R_N^2
&\eqdef 1+\Big(\frac{1}{N^2}\sum_{i,j=1}^N\|{\cZ}^{\<2>}_{ij}\|^{\frac2{1-\theta}}_{\bC^{-1-\kappa}}\Big)
+\Big(\frac{1}{N}\sum_{j=1}^N\|\CZ^{\<2>}_{jj}\|_{\bC^{-1-\kappa}}^{\frac2{1-\theta}}\Big)
\\
&\qquad +\frac{1}{N^3}\sum_{i,j,l=1}^N\Big(\| \tilde{\cZ}^{\<22>}_{il,ij}\|^{2}_{\bC^{-\kappa}}
+\| \tilde{\cZ}^{\<22>}_{ll,ij}\|^{2}_{\bC^{-\kappa}}
+\| \tilde{\cZ}^{\<22>}_{il,jj}\|^{2}_{\bC^{-\kappa}}\Big)
 +
 \frac{1}{N^2}\sum_{j,l=1}^N\| \tilde{\cZ}^{\<22>}_{ll,jj}\|^{2}_{\bC^{-\kappa}}.
\endaligned$$
\ep
\begin{proof}We estimate each term in $\Theta$. By Young's inequality we will use the first line in $R_N^2$ to control the renormalization terms evolving $\cZ^{\<2>}_{ij}$ below and use the second and the third line in $R_N^2$ to bound the renormalization terms containing $\tilde{\cZ}^{\<22>}_{ij,kl}$.
	
\newcounter{UnifNa1} 
\refstepcounter{UnifNa1} 
{\sc Step} \arabic{UnifNa1} \label{UnifN1a} \refstepcounter{UnifNa1} (Estimates of $D(Y_i,{\cZ}^{\<2>}_{ij},Y_j)$ and $D(Y_i,{\CZ}^{\<2>}_{jj},Y_i)$)

We first control the terms containing the commutator $D$ in $\Theta$.	 By Lemma \ref{lem:com3}, H\"older's inequality, and interpolation Lemma \ref{lem:interpolation}, we have
\begin{align}\no
&\Big|\frac{\l}{N}\sum_{i,j=1}^N2D(Y_i,{\cZ}^{\<2>}_{ij},Y_j)\Big|
\lesssim \frac{\l}{N}\sum_{i,j=1}^N\|Y_i\|_{H^{\frac12+\kappa}}\|{\cZ}^{\<2>}_{ij}\|_{\bC^{-1-\kappa}}\|Y_j\|_{H^{\frac12+\kappa}}
\\&\lesssim \l\Big(\frac{1}{N^2}\sum_{i,j=1}^N\|{\cZ}^{\<2>}_{ij}\|^2_{\bC^{-1-\kappa}}\Big)^{\frac12}\Big(\sum_{i=1}^N\|Y_i\|^2_{H^{\frac12+\kappa}}\Big)\no
\\&\lesssim \l\Big(\frac{1}{N^2}\sum_{i,j=1}^N\|{\cZ}^{\<2>}_{ij}\|^2_{\bC^{-1-\kappa}}\Big)^{\frac12}\sum_{i=1}^N\|Y_i\|^{2\theta}_{H^{1-2\kappa}}\|Y_i\|^{2(1-\theta)}_{L^2}\label{szz1}
\\&\lesssim\l\Big(\frac{1}{N^2}\sum_{i,j=1}^N\|{\cZ}^{\<2>}_{ij}\|^2_{\bC^{-1-\kappa}}\Big)^{\frac12}\Big(\sum_{i=1}^N\|Y_i\|^{2}_{H^{1-2\kappa}}\Big)^\theta
\Big(\sum_{i=1}^N\|Y_i\|^{2}_{L^2}\Big)^{1-\theta}\no
\\
&\le\delta \sum_{i=1}^{N}\|Y_{i}\|_{H^{1-2\kappa}}^{2}+C_\delta\l^{\frac{1}{1-\theta}} \Big(\sum_{i=1}^N\|Y_i\|^{2}_{L^2}\Big)
\Big(\frac{1}{N^2}\sum_{i,j=1}^N\|{\cZ}^{\<2>}_{ij}\|^2_{\bC^{-1-\kappa}}\Big)^{\frac1{2(1-\theta)}},\no
\end{align}
where $\theta=\frac{1+2\kappa}{2-4\kappa}\in(\frac12,\frac23)$. Similarly we have
\begin{align*}
&\Big|\frac{\l}{N}\sum_{i,j=1}^ND(Y_i,{\CZ}^{\<2>}_{jj},Y_i)\Big|\lesssim \frac{\l}{N}\sum_{i,j=1}^N\|Y_i\|^2_{H^{\frac12+\kappa}}\|\CZ^{\<2>}_{jj}\|_{\bC^{-1-\kappa}}
\\
&\le\delta\sum_{i=1}^{N}\|Y_{i}\|_{H^{1-2\kappa}}^{2}+C_\delta \l^{\frac{1}{1-\theta}}\Big(\sum_{i=1}^N\|Y_i\|^{2}_{L^2}\Big)
\Big(\frac{1}{N}\sum_{j=1}^N\|\CZ^{\<2>}_{jj}\|_{\bC^{-1-\kappa}}\Big)^{
\frac1{1-\theta}}.
\end{align*}
Therefore the terms containing $D$ in $\Theta$ can be controlled by the right hand side of \eqref{zmm5}.

{\sc Step} \arabic{UnifNa1} \label{UnifN2a} \refstepcounter{UnifNa1} (Estimate of terms in $\la \sD^{-1}P_i,P_i\ra$)

In the following we estimate each term in \eqref{eqPi}. We have three types of terms:

\textbf{I.} Terms without $C$ or $D$ such as $\langle \tilde{\cZ}^{\<22>}_{il,ij},4Y_lY_j\rangle$,

\textbf{II.} Terms with $C$, such as $\langle C(Y_i,\cZ^{\<2>}_{ll},{\cZ}^{\<2>}_{ij}),Y_j\rangle$,

\textbf{III.} Terms with $D$, such as
 $D(Y_j,{\cZ}^{\<2>}_{ij},\mathscr{D}^{-1}P_i)$.

The terms having the same type can be estimated in the same way.

\textbf{I.} For the first type
we use Lemma \ref{lem:interpolation} and Lemma \ref{lem:multi} to have
\begin{align*}
&\Big|\frac{\l^2}{N^2}\sum_{i,j,l=1}^N\langle \tilde{\cZ}^{\<22>}_{il,ij},4Y_lY_j\rangle\Big|
\lesssim\frac{\l^2}{N^2}\sum_{i,j,l=1}^N\| \tilde{\cZ}^{\<22>}_{il,ij}\|_{\bC^{-\kappa}}\|Y_l\|_{H^{2\kappa}}\|Y_j\|_{H^{2\kappa}}
\\&\lesssim\frac{\l^2}{N}\sum_{i=1}^N\Big(\frac{1}{N^2}\sum_{j,l=1}^N\| \tilde{\cZ}^{\<22>}_{il,ij}\|^2_{\bC^{-\kappa}}\Big)^{\frac12}\Big(\sum_{j=1}^N\|Y_j\|^2_{H^{2\kappa}}\Big)
\\
&\le\delta \sum_{i=1}^{N}\|Y_{i}\|_{H^{1-2\kappa}}^{2}+C_\delta \l^{\frac2{1-\theta_1}}\Big(\sum_{i=1}^N\|Y_i\|^{2}_{L^2}\Big)\Big(\frac{1}{N^3}\sum_{i,j,l=1}^N\| \tilde{\cZ}^{\<22>}_{il,ij}\|^{2}_{\bC^{-\kappa}}\Big)^{\frac1{2(1-\theta_1)}},
\end{align*}
where $\theta_1=\frac{2\kappa}{1-2\kappa}\in(0,1)$ and we used Lemma \ref{lem:multi} and Besov embedding $H^{2\kappa}\subset B^\kappa_{2,1} $  to have
$$\|Y_lY_j\|_{B^\kappa_{1,1}}\lesssim \|Y_l\|_{B^\kappa_{2,1}}\|Y_j\|_{B^\kappa_{2,1}}\lesssim \|Y_l\|_{H^{2\kappa}}\|Y_j\|_{H^{2\kappa}},$$
in the first inequality.
By the exactly same arguments the same bounds hold for
$$
\Big|\frac{\l^2}{N^2}\sum_{i,j,l=1}^N \langle \tilde{\cZ}^{\<22>}_{ll,ij},2Y_iY_j\rangle\Big|, \;\;
\Big|\frac{\l^2}{N^2}\sum_{i,j,l=1}^N\langle \tilde{\cZ}^{\<22>}_{il,jj},2Y_lY_i\rangle\Big|, \;\;
\Big| \frac{\l^2}{N^2}\sum_{i,j,l=1}^N\langle \tilde{\cZ}^{\<22>}_{ll,jj},Y_i^2\rangle\Big|
$$
with $\frac{1}{N^3}\sum_{i,j,l=1}^N\| \tilde{\cZ}^{\<22>}_{il,ij}\|^{2}_{\bC^{-\kappa}}$ on the right hand side replaced by, respectively,
 $$
\frac{1}{N^3}\sum_{i,j,l=1}^N\| \tilde{\cZ}^{\<22>}_{ll,ij}\|^{2}_{\bC^{-\kappa}},\;\;\frac{1}{N^3}\sum_{i,j,l=1}^N\| \tilde{\cZ}^{\<22>}_{il,jj}\|^{2}_{\bC^{-\kappa}},\;\;\frac{1}{N^2}\sum_{j,l=1}^N\| \tilde{\cZ}^{\<22>}_{ll,jj}\|^{2}_{\bC^{-\kappa}}.
$$
By Young's inequality the terms with $\sum_{i=1}^N\|Y_i\|^{2}_{L^2}$ 
are all bounded by $\l^2(1+\l)\Big(\sum_{i=1}^N\|Y_i\|^{2}_{L^2}\Big)R_N^2$.

\textbf{II.} By Lemma \ref{lem:com2}, interpolation Lemma \ref{lem:interpolation} and H\"older inequality,
\begin{align*}
&\Big|\frac{\l^2}{N^2}\sum_{i,j,l=1}^N\langle C(Y_l,\cZ^{\<2>}_{il},{\cZ}^{\<2>}_{ij}),Y_j\rangle\Big|
\\&\lesssim \frac{\l^2}{N^2}\sum_{i,j,l=1}^N\|\cZ^{\<2>}_{il}\|_{\bC^{-1-\kappa}}\|{\cZ}^{\<2>}_{ij}\|_{\bC^{-1-\kappa}}\|Y_l\|_{H^{\frac12+\kappa}}\|Y_j\|_{H^{\frac12+\kappa}}
\\
&\le\delta\sum_{i=1}^{N}\|Y_{i}\|_{H^{1-2\kappa}}^{2}+C_\delta\l^{\frac2{1-\theta}} \Big(\sum_{i=1}^N\|Y_i\|^{2}_{L^2}\Big)\Big(\frac{1}{N^3}\sum_{i,j,l=1}^N
\|\cZ^{\<2>}_{il}\|_{\bC^{-1-\kappa}}^{\frac1{1-\theta}}\|{\cZ}^{\<2>}_{ij}\|_{C^{-1-\kappa}}^{\frac1{1-\theta}}\Big),
\end{align*}
where we use similar argument as in \eqref{szz1} and by Young's inequality the last term can be controlled by $\l(1+\l^4)\Big(\sum_{i=1}^N\|Y_i\|^{2}_{L^2}\Big)R_N^2$.

By the exactly same arguments the same bounds hold for 
\begin{equs}
\Big|\frac{\l^2}{N^2}\sum_{i,j,l=1}^N2\langle C(Y_i,\cZ^{\<2>}_{ll},{\cZ}^{\<2>}_{ij}),Y_j\rangle\Big|, \;\; 
 \Big|\frac{\l^2}{N^2}\sum_{i,j,l=1}^N2\langle C(Y_l,\cZ^{\<2>}_{il},{\cZ}^{\<2>}_{jj}),Y_i\rangle\Big|,
\\
\mbox{and }\quad
\Big|\frac{\l^2}{N^2}\sum_{i,j,l=1}^N2\langle C(Y_i,\cZ^{\<2>}_{ll},{\cZ}^{\<2>}_{jj}),Y_i\rangle\Big| 
\end{equs}
with $\|\cZ^{\<2>}_{il}\|_{\bC^{-1-\kappa}}^{\frac1{1-\theta}}\|{\cZ}^{\<2>}_{ij}\|_{C^{-1-\kappa}}^{\frac1{1-\theta}}$ on the right hand side replaced by, respectively,
$$
\|\cZ^{\<2>}_{ll}\|_{\bC^{-1-\kappa}}^{\frac1{1-\theta}}\|{\cZ}^{\<2>}_{ij}\|_{C^{-1-\kappa}}^{\frac1{1-\theta}},\;\;\|\cZ^{\<2>}_{il}\|_{\bC^{-1-\kappa}}^{\frac1{1-\theta}}\|{\cZ}^{\<2>}_{jj}\|_{C^{-1-\kappa}}^{\frac1{1-\theta}},\;\;\|\cZ^{\<2>}_{ll}\|_{\bC^{-1-\kappa}}^{\frac1{1-\theta}}\|{\cZ}^{\<2>}_{jj}\|_{C^{-1-\kappa}}^{\frac1{1-\theta}}.
$$

As above by Young's inequality the terms containing $\sum_{i=1}^N\|Y_i\|^{2}_{L^2}$ 
can be all controlled by $\l(1+\l^4)\Big(\sum_{i=1}^N\|Y_i\|^{2}_{L^2}\Big)R_N^2$.

\textbf{III.} By Lemma \ref{lem:com3}  and setting $\theta_0=\frac{4\kappa}{1-2\kappa}\in(0,1)$, we have
\begin{align*}
&\Big|\frac{\l}{N}\sum_{i,j=1}^N
-2D(Y_j,{\cZ}^{\<2>}_{ij},\mathscr{D}^{-1}P_i)\Big|
\lesssim \frac{\l}{N}\sum_{i,j=1}^N\|Y_j\|_{H^{4\kappa}}\|{\cZ}^{\<2>}_{ij}\|_{\bC^{-1-\kappa}}\|\mathscr{D}^{-1}P_i\|_{H^{1-2\kappa}}
\\&\lesssim \l\Big(\frac{1}{N^2}\sum_{i,j=1}^N\|{\cZ}^{\<2>}_{ij}\|^2_{\bC^{-1-\kappa}}\Big)^{\frac12}\Big(\sum_{j=1}^N\|Y_j\|^2_{H^{4\kappa}}\Big)^{\frac12}\Big(\sum_{i=1}^N
\|\mathscr{D}^{-1}P_i\|^2_{H^{1-2\kappa}}\Big)^{\frac12}
\\
&\lesssim \l^2\Big(\frac{1}{N^2}\sum_{i,j=1}^N\|{\cZ}^{\<2>}_{ij}\|^2_{\bC^{-1-\kappa}}\Big)\Big(\sum_{j=1}^N\|Y_j\|^{2\theta_0}_{H^{1-2\kappa}}\|Y_j\|^{2(1-\theta_0)}_{L^2}\Big)+\delta\sum_{i=1}^N\|\mathscr{D}^{-1}P_i\|^2_{H^{1-2\kappa}}
\\
&\le C_\delta\l^{\frac2{1-\theta_0}} \Big(\sum_{i=1}^N\|Y_i\|^{2}_{L^2}\Big)\Big(\frac{1}{N^2}\sum_{i,j=1}^N\|{\cZ}^{\<2>}_{ij}\|^2_{\bC^{-1-\kappa}}\Big)^{\frac1{1-\theta_0}}
+\delta\sum_{j=1}^N \Big(\|Y_j\|^{2}_{H^{1-2\kappa}}+\|\mathscr{D}^{-1}P_j\|^2_{H^{1-2\kappa}}\Big),
\end{align*}
where we used Young's inequality and interpolation Lemma \ref{lem:interpolation} in the third inequality.

Similarly we have
\begin{align*}
&\Big|\frac{\l}{N}\sum_{i,j=1}^N
2D(Y_i,\CZ^{\<2>}_{jj},\mathscr{D}^{-1}P_i)\Big|
\lesssim \frac{\l}{N}\sum_{i,j=1}^N\|Y_i\|_{H^{4\kappa}}\|\CZ^{\<2>}_{jj}\|_{\bC^{-1-\kappa}}\|\mathscr{D}^{-1}P_i\|_{H^{1-2\kappa}}
\\
&\le C_\delta\l^{\frac2{1-\theta_0}} \Big(\sum_{i=1}^N\|Y_i\|^{2}_{L^2}\Big)\Big(\frac{1}{N}\sum_{j=1}^N\|\CZ^{\<2>}_{jj}\|^2_{\bC^{-1-\kappa}}\Big)^{\frac1{1-\theta_0}}
\! +\delta\sum_{j=1}^N\Big(\|Y_j\|^{2}_{H^{1-2\kappa}}
\!+\|\mathscr{D}^{-1}P_j\|^2_{H^{1-2\kappa}}\Big),
\end{align*}
which implies the result by the first two bounds in Lemma \ref{lem:phi1}.
\end{proof}

For $\Xi$, since we will use Lemma \ref{lem:X3} to bound the $L_T^2$ norm of the stochastic terms, we bound $\|\Xi\|_{L_T^1}$.

\bp\label{Xi1}
It holds  for $\delta>0$ small that
\begin{align*}
\int_0^T|\Xi|\dif s
& \leq \delta\sum_{i=1}^{N}\|\nabla \varphi_{i}\|_{L_T^2L^{2}}^{2}+\delta\frac{\l}{N}\Big \|\sum_{i=1}^{N}Y_{i}^{2}\Big\|^2_{L^2_TL^2}+\delta \sum_{i=1}^{N}\|Y_{i}\|_{L_T^2H^{1-2\kappa}}^{2}
\\
&\qquad+C\l(1+\l^{55})\int_0^T\Big(\sum_{i=1}^N\|Y_i\|^{2}_{L^2}\Big)Q_N^3\dif s+\l(1+\l) Q_N^4.
\end{align*}
Here $C$ is independent of $\l, N$ and can be chosen uniform for $m\geq1$ and
\begin{align*}
Q_N^3 &\eqdef R_N^4+1+\Big(\frac{1}{N}\sum_{j=1}^N\|X_j\|_{C_T\bC^{\frac{1}{2}-2\kappa}}^2\Big)^2+\Big(\frac{1}{N^2}\sum_{i,j=1}^N\|X_jZ_i\|^2_{C_T\bC^{-\frac12-\kappa}}\Big)
\\&\qquad+\Big(\frac{1}{N}\sum_{j=1}^N\|X_jZ_j\|^2_{C_T\bC^{-\frac12-\kappa}}\Big)
+\frac{1}{N^2}\sum_{i,j=1}^N\Big(\|{\cZ}^{\<2>}_{ij}\|_{C_T\bC^{-1-\kappa}}^{\frac{1}{1-\theta}}+\|{\CZ}^{\<2>}_{jj}\|_{C_T\bC^{-1-\kappa}}^{\frac{1}{1-\theta}}\Big),
\end{align*}
with $\theta=\frac{\frac12+\kappa}{1-2\kappa}$ and
\begin{align}
Q_N^4 &\eqdef
\Big(\sum_{i=1}^N\|X_i\|_{L_T^2H^{\frac{1}{2}-2\kappa}}^2\Big)
\Big[1+\frac{1}{N^2}\sum_{i,j=1}^N 2^{8\kappa L} \Big( \|{\cZ}^{\<2>}_{ij}\|^2_{C_T\bC^{-1-\kappa}}+\|{\CZ}^{\<2>}_{jj}\|^2_{C_T\bC^{-1-\kappa}}\Big)	\notag
\\
&\qquad\qquad\qquad\qquad\qquad
+ \frac{1}{N^2}\sum_{i,j=1}^N
\Big(
\|{\cZ}^{\<2>}_{ij}\|_{C_T\bC^{-1-\kappa}}^2
+ \|{\CZ}^{\<2>}_{jj}\|_{C_T\bC^{-1-\kappa}}^2\Big)\Big] 	\notag
\\
&+\Big(\frac{1}{N^2}\sum_{i=1}^N\Big(\sum_{j=1}^N\|X_j^2Z_i\|_{L^2_TH^{-\frac12-2\kappa}}\Big)^{2}\Big)
+\Big(\frac{1}{N^2}\sum_{i=1}^N\Big(\sum_{j=1}^N\|X_iX_jZ_j\|_{L^2_TH^{-\frac12-2\kappa}}\Big)^{2}\Big)
\notag\\
&+\Big(\frac{1}{N^2}\sum_{i=1}^N\Big(\sum_{j=1}^N\big(
	\|X_i\circ \CZ^{\<2>}_{jj}\|_{L^2_TH^{-\frac12-2\kappa}}
	+\|X_j\circ{\cZ}^{\<2>}_{ij}\|_{L^2_TH^{-\frac12-2\kappa}}\big)
	\Big)^{2}\Big)
 \label{zrx1}
\end{align}
where  $L$ is chosen as in \eqref{L}  and $R_N^4$ is defined in Lemma \ref{Cubic}.
\ep
Note that with all the estimates obtained above we can easily deduce for any $q\geq1$
\begin{equ}[e:Q4]
[\E (Q_N^4)^q]^{1/q}\lesssim \l^2(1+\l^6).
\end{equ}
Indeed, using Lemmas \ref{lem:X2} and \ref{lem:X3} we can bound all the terms in $Q^4_N$ that are $L_T^2$ norms and involving $X$.
These bounds together with Lemma~\ref{thm:renorm43}  and  the definition of $L$ in \eqref{L} imply \eqref{e:Q4}.

In the following lemma we use  the renormalized terms in $\mathbb{Z}$ to bound $Q_N^3$,
which will be useful in Section \ref{sec:inv} because,
$Q_N^3$ involves $(X_i)_{i=1}^N$ which are not independent for different $i$
whereas for $Q^5_N$ below it will be easier to exploit independence.

\bl\label{lem:zmm} 
One has
$Q_N^3\lesssim (1+\l^4)Q_N^5$
uniformly  in $\l$ and $N$,
with
\begin{align}\label{eq:Q}
Q_N^5\eqdef R_N^4+&\frac{1}{N^2}\sum_{i,j=1}^N(\|{\cZ}^{\<2>}_{ij}\|_{C_T\bC^{-1-\kappa}}^{\frac{1}{1-\theta}}+\|{\CZ}^{\<2>}_{jj}\|_{C_T\bC^{-1-\kappa}}^{\frac{1}{1-\theta}})
+\sum_{i=1}^3Q_N^{5i},
\end{align}
where $R_N^4$ is defined in Lemma \ref{Cubic},
 $Q_N^{51}$ is defined in the proof and $Q_N^{52}, Q_N^{53}$ are defined in the proof of Lemma \ref{lem:X1}.  
\el
\begin{proof}
In the proof we bound each term in $Q_N^3$ in terms of the renormalized terms in $\mathbb{Z}$.
Obviously nothing needs to be done for the term $R_N^4$ in $Q_N^3$.
By \eqref{bdx} in Lemma~\ref{lem:X} 
\begin{align*}
\Big(\frac{1}{N}\sum_{j=1}^N\|X_j\|_{C_T\bC^{\frac{1}{2}-2\kappa}}^2\Big)^2
\lesssim\frac{\l^4}{N^2}\sum_{i,j=1}^N\|\tilde{\cZ}^{\<30>}_{ijj}\|_{C_T\bC^{\frac{1}{2}-2\kappa}}^4\eqdef \l^4Q_N^{51}.
\end{align*}
Thus the result follows from Lemma \ref{lem:X1}.
\end{proof}

Now we focus on the estimate of $\Xi$. First we consider the cubic term $\la  Y_j^2Y_i,Z_i \ra$ in $\Xi$. Comparing to the dynamical $\Phi^4_3$ model, the dissipation from $\frac{1}{N} \|\sum_{i=1}^{N}Y_{i}^{2}\|^2_{L^2}$ is weaker, which requires further decomposition and more delicate estimates. Also unlike 2D case, the best regularity for $Y_i$ is $H^{1-}$ ($H^{1}$ in 2D case). Here we decompose $Y_i$ as $\varphi_i$ (having better regularity) and $\mathcal{D}^{-1}P_i$ (bound of which only need $L^2$-norm of $Y_i$, see Lemma \ref{lem:phi1}). For the most complicated terms (see Step 3 in the following proof) we also need to decompose $Z_i$ by localization operator and choosing $L$ to balance the competing contributions.

\bl\label{Cubic} It holds for $\delta>0$ small that
\begin{equation}\aligned\label{eq:Aa}
\Big|\frac{\l}{N}\sum_{i,j=1}^N\la  Y_j^2Y_i,Z_i \ra \Big|
&\leq \delta\sum_{i=1}^{N}\|\varphi_{i}\|_{H^1}^{2}+\delta\frac{\l}{N}\Big \|\sum_{i=1}^{N}Y_{i}^{2}\Big\|^2_{L^2}+\delta \sum_{i=1}^{N}\|Y_{i}\|_{H^{1-2\kappa}}^{2}
\\
&\quad+C_\delta \l(1+\l^{55}) \Big (\sum_{j=1}^{N}\|Y_{j}\|_{L^{2}}^{2} \Big )R_N^4,
 \endaligned\end{equation}
where  $C_\delta$ is independent of $\l, N$ and can be chosen uniform for $m\geq1$ and
$$\aligned R_N^4=(R_N^1)^{25} \Big (\frac{1}{N}\sum_{j=1}^{N}\|Z_j\|_{\bC^{-\frac12-\kappa}}^{2}+1  \Big )^{9}.\endaligned$$
\el
\begin{proof}

 We use \eqref{phi} to get a decomposition for $Y_i$ and we have
\begin{align*}
\frac{\l}{N}\sum_{i,j=1}^N\la  Y_j^2Y_i,Z_i \ra
&=\frac{\l}{N}\sum_{i,j=1}^N\bigg[\la  \varphi_j^2\varphi_i, Z_i \ra +\la  [\mathscr{D}^{-1}P_j]^2Y_i,Z_i \ra -2\la  [\mathscr{D}^{-1}P_j]\varphi_j\varphi_i,Z_i \ra
\\&\qquad\qquad-\la  \varphi_j^2[\mathscr{D}^{-1}P_i],Z_i \ra +2\la  [\mathscr{D}^{-1}P_j]\varphi_j[\mathscr{D}^{-1}P_i],Z_i \ra \bigg].
\end{align*}
In the following we show that each term can be bounded by the RHS of \eqref{eq:Aa}.
\newcounter{UnifNa} 
\refstepcounter{UnifNa} 

{\sc Step} \arabic{UnifNa} \label{UnifN1a} \refstepcounter{UnifNa} (Estimate of $\la  \varphi_j^2\varphi_i, Z_i \ra $.)

In this step we use Lemma \ref{lem:dual+MW} with $s=\frac12+\kappa$ to have
\begin{align}
&\Big|\frac{\l}{N}\sum_{i,j=1}^N\la  \varphi_j^2\varphi_i,Z_i \ra\Big| =\Big|\frac{\l}{N}\sum_{i,j=1}^N\la  \varphi_i^2\varphi_j,Z_j \ra\Big| \nonumber
\\
&\lesssim\frac{\l}{N}\sum_{i,j=1}^{N} \big ( \|\nabla(\varphi_{i}^2\varphi_{j})\|_{L^{1}}^{\frac{1}{2}+\kappa} \|\varphi_{i}^{2} \varphi_{j}\|_{L^{1}}^{\frac{1}{2}-\kappa}+\|\varphi_{i}^{2} \varphi_{j}\|_{L^{1}}\big )\|Z_j\|_{\bC^{-\frac12-\kappa}}			\label{eq:a}
\\
&\lesssim\frac{\l}{N}\sum_{j=1}^{N}\Big[ \Big ( \sum_{i=1}^N\|\nabla(\varphi_{i}^2\varphi_{j})\|_{L^{1}}\Big)^{\frac{1}{2}+\kappa} \Big(\sum_{i=1}^N\|\varphi_{i}^{2} \varphi_{j}\|_{L^{1}}\Big)^{\frac{1}{2}-\kappa}+\sum_{i=1}^N\|\varphi_{i}^{2} \varphi_{j}\|_{L^{1}}\Big ]\|Z_j\|_{\bC^{-\frac12-\kappa}}   \nonumber
\end{align}
where we used H\"older's inequality for the summation in $i$ in the last step.
By H\"older's inequality it holds that
\begin{align}
\sum_{i=1}^N\|\varphi_{i}^{2} \varphi_{j}\|_{L^{1}}\leq \sum_{i=1}^N\|\varphi_i\varphi_j\|_{L^2}\|\varphi_i\|_{L^2}
\leq \Big(\sum_{i=1}^N\|\varphi_i\varphi_j\|_{L^2}^2\Big)^{\frac12}\Big(\sum_{i=1}^N\|\varphi_i\|_{L^2}^2\Big)^{\frac12},\label{eq:a2}
\end{align}
and
\begin{align}
\sum_{i=1}^N\|\nabla(\varphi_{i}^{2} \varphi_{j})\|_{L^{1}}
&\lesssim \sum_{i=1}^N\|\varphi_{i}\varphi_{j}\nabla \varphi_{i}\|_{L^{1}}+\Big\|\sum_{i=1}^N\varphi_{i}^{2}\nabla \varphi_{j}\Big\|_{L^{1}} \label{eq:a1}
\\
& \lesssim \Big(\sum_{i=1}^N\|\varphi_i\varphi_j\|_{L^2}^2\Big)^{\frac12}\Big(\sum_{i=1}^N\|\varphi_i\|_{H^1}^2\Big)^{\frac12}+\Big\|\sum_{i=1}^N\varphi_{i}^{2}\Big\|_{L^{2}}\|\nabla \varphi_{j}\|_{L^2}.\nonumber
\end{align}
Substituting \eqref{eq:a2}--\eqref{eq:a1}   into \eqref{eq:a} we obtain
that \eqref{eq:a}
 is bounded by
\begin{align*}
&\frac{\l}{N}\sum_{j=1}^{N} \Big(\sum_{i=1}^N\|\varphi_i\varphi_j\|_{L^2}^2\Big)^{\frac{1}{2}} \Big(\sum_{i=1}^N\|\varphi_{i}\|_{H^1}^2\Big)^{\frac{1+2\kappa}{4}}
\Big(\sum_{i=1}^N\|\varphi_{i}\|_{L^2}^2\Big)^{\frac{1-2\kappa}{4}}\|Z_j\|_{\bC^{-\frac12-\kappa}}
\\&+\frac{\l}{N}\sum_{j=1}^{N} \Big\|\sum_{i=1}^N\varphi_i^2\Big\|_{L^2}^{\frac{1}{2}+\kappa} \|\nabla \varphi_j\|_{L^2}^{\frac{1}{2}+\kappa}
\Big(\sum_{i=1}^N\|\varphi_i\varphi_j\|_{L^2}^2\Big)^{\frac{1-2\kappa}4}\Big(\sum_{i=1}^N\|\varphi_i\|_{L^2}^2\Big)^{\frac{1-2\kappa}4}
\|Z_j\|_{\bC^{-\frac12-\kappa}}\\&
+\frac{\l}{N}\sum_{j=1}^{N}\Big(\sum_{i=1}^N\|\varphi_i\varphi_j\|_{L^2}^2\Big)^{\frac12}\Big(\sum_{i=1}^N\|\varphi_i\|_{L^2}^2\Big)^{\frac12}
\|Z_j\|_{\bC^{-\frac12-\kappa}},\end{align*}
which by H\"older's inequality for the summation in $j$ is bounded by
\begin{align*}
&\l\Big(\frac{1}{N}\sum_{i,j=1}^N\|\varphi_i\varphi_j\|_{L^2}^2\Big)^{\frac{1}{2}} \Big(\sum_{i=1}^N\|\varphi_{i}\|_{H^1}^2\Big)^{\frac{1+2\kappa}{4}}
\Big(\sum_{i=1}^N\|\varphi_{i}\|_{L^2}^2\Big)^{\frac{1-2\kappa}{4}}\Big(\frac{1}{N}\sum_{j=1}^{N} \|Z_j\|_{\bC^{-\frac12-\kappa}}^2\Big)^{\frac12}
\\&
+\l\Big(\frac{1}{N}\sum_{i,j=1}^N\|\varphi_i\varphi_j\|_{L^2}^2\Big)^{\frac12}\Big(\sum_{i=1}^N\|\varphi_i\|_{L^2}^2\Big)^{\frac12}
\Big(\frac{1}{N}\sum_{j=1}^{N}\|Z_j\|_{\bC^{-\frac12-\kappa}}^2\Big)^{\frac12}.
\end{align*}
Using Lemma \ref{lem:phi1}, we can
bound
$\sum_{i=1}^N\|\varphi_i\|_{L^2}^2$ and $\frac{1}{N}\sum_{i,j=1}^N\|\varphi_i\varphi_j\|_{L^2}^2$,
which gives the desired bound  \eqref{eq:Aa} for  \eqref{eq:a} by H\"older's inequality, namely,  the second term in the last expression is bounded by $C$ times
\begin{equs}
 {}&\l\Big(\frac{1}N\|\sum_{i=1}^NY_i^2\|_{L^2}^2 \Big)^{\frac12}
\Big(\sum_{i=1}^N\|Y_i\|_{L^2}^2 \Big)^{\frac12}
\Big(\frac{1}{N}\sum_{j=1}^{N}\|Z_j\|_{\bC^{-\frac12-\kappa}}^2\Big)^{\frac12}
(1+\l^2 R_N^1)^{3/2}
\\
&\leq
\frac{\delta\l}N\|\sum_{i=1}^NY_i^2\|_{L^2}^2
 +
C_\delta\l(1+\l^2)^3 \Big(\sum_{i=1}^N\|Y_i\|_{L^2}^2 \Big)
\Big(\frac{1}{N}\sum_{j=1}^{N}\|Z_j\|_{\bC^{-\frac12-\kappa}}^2\Big)
( R_N^1)^3,
\end{equs}
and the first term is bounded in the same way.

\vspace{2ex}

{\sc Step} \arabic{UnifNa} \label{UnifN1a1} \refstepcounter{UnifNa} (Estimates of $\la  [\mathscr{D}^{-1}P_j]^2Y_i, Z_i \ra $ and
$\la  [\mathscr{D}^{-1}P_j]\varphi_j[\mathscr{D}^{-1}P_i], Z_i \ra $.)

We  use  \eqref{ine:Lei} and Lemma \ref{lem:emb} and Sobolev embedding $H^{\frac34}\subset L^4$ to have
\begin{align}\label{zmm01}
\|fgh\|_{B^{\frac12+\kappa}_{1,1}}\lesssim \|fg\|_{L^2}\|h\|_{H^{\frac34}}+
\|fg\|_{B^{\frac12+\kappa}_{\frac43,1}}\|h\|_{L^4}\lesssim
\|f\|_{H^{1-2\kappa}}\|g\|_{H^{1-2\kappa}}\|h\|_{H^{\frac34}},
\end{align}
which combined with (ii) in Lemma \ref{lem:multi} gives
\begin{align}\label{nL2}
&\Big|\frac{\l}{N}\sum_{i,j=1}^N\la  [\mathscr{D}^{-1}P_j]^2Y_i, Z_i \ra \Big|\nonumber\lesssim
\frac{\l}{N}\sum_{i,j=1}^N\|[\mathscr{D}^{-1}P_j]^2Y_i\|_{B^{\frac12+\kappa}_{1,1}}\|Z_i\|_{\bC^{-\frac12-\kappa}}
\nonumber
\\
&\lesssim
\frac{\l}{N}\sum_{i,j=1}^N\|\mathscr{D}^{-1}P_j\|_{H^{1-2\kappa}}^2\|Y_i\|_{H^{\frac34}}\|Z_i\|_{\bC^{-\frac12-\kappa}}
\nonumber\\
&\lesssim
\frac{\l}{\sqrt{N}}\Big (\sum_{j=1}^{N}\|Y_{j}\|_{L^{2}}^{2} \Big )\l^2R_N^1\Big (\sum_{j=1}^{N}\| Y_{j}\|_{H^{1-2\kappa}}^{2} \Big )^{\frac{\theta}2}\Big (\sum_{j=1}^{N}\| Y_{j}\|_{L^2}^{2} \Big )^{\frac{1-\theta}2} \Big (\frac{1}{N}\sum_{j=1}^{N}\|Z_j\|_{\bC^{-\frac12-\kappa}}^{2}  \bigg )^{\frac12}\nonumber
\\
&\le C_\delta \l^{\frac5{1-\theta}}
(1+R_N^1)^{\frac2{1-\theta}} \Big (1+\frac{1}{N}\sum_{j=1}^{N}\|Z_j\|_{\bC^{-\frac12-\kappa}}^{2}  \Big )^{\frac1{1-\theta}}\Big (\sum_{j=1}^{N}\|Y_{j}\|_{L^{2}}^{2} \Big )\nonumber
\\
&\qquad\qquad +\delta \l \Big \|\frac{1}{\sqrt{N}}\sum_{i=1}^{N}Y_{i}^{2} \Big \|_{L^{2}}^{2}+\delta\Big (\sum_{j=1}^{N}\| Y_{j}\|_{H^{1-2\kappa}}^{2} \Big ),
\end{align}
where $\theta=\frac{3}{4(1-2\kappa)}$ and we used Lemma \ref{lem:phi1} to bound $\|\mathscr{D}^{-1}P_j\|_{H^{1-2\kappa}}^2$ in the third inequality and \eqref{yi} in the last inequality.
Similarly we use \eqref{zmm01} and Lemma \ref{lem:phi1} to have
\begin{align}\label{nL3}
&\Big|\frac{\l}{N}\sum_{i,j=1}^N\la  [\mathscr{D}^{-1}P_j]\varphi_j[\mathscr{D}^{-1}P_i], Z_i \ra \Big|\nonumber
\\&\lesssim \l
 \bigg \|\frac{1}{\sqrt{N}}\sum_{i=1}^{N}Y_{i}^{2} \bigg \|_{L^{2}}\l^2 R_N^1\bigg (\sum_{j=1}^{N}\| \varphi_{j}\|_{H^{1}}^{2} \bigg )^{\frac{\theta}2}\bigg ( \sum_{j=1}^{N}\| \varphi_{j}\|_{L^2}^{2} \bigg )^{\frac{1-\theta}2} \bigg (\frac{1}{N}\sum_{j=1}^{N}\|Z_j\|_{\bC^{-\frac12-\kappa}}^{2}  \bigg )^{\frac12}\nonumber
 \\
& \le
C_\delta \l^{\frac5{1-\theta}}(1+\l^2)(R_N^1)^{\frac{3-\theta}{1-\theta}} \Big (\frac{1}{N}\sum_{j=1}^{N}\|Z_j\|_{\bC^{-\frac12-\kappa}}^{2}  \Big )^{\frac1{1-\theta}}\Big (\sum_{j=1}^{N}\|Y_{j}\|_{L^{2}}^{2} \Big )\nonumber
\\&\qquad\qquad\qquad+\delta\l  \Big \|\frac{1}{\sqrt{N}}\sum_{i=1}^{N}Y_{i}^{2} \Big \|_{L^{2}}^{2}+\delta\Big (\frac{1}{N} \sum_{j=1}^{N}\| \varphi_{j}\|_{H^{1}}^{2} \Big ).
\end{align}

{\sc Step} \arabic{UnifNa} \label{UnifN4} \refstepcounter{UnifNa} (Estimates for $\la  [\mathscr{D}^{-1}P_j]\varphi_j\varphi_i,Z_i \ra $ and $\la  [\mathscr{D}^{-1}P_i]\varphi^2_j, Z_i \ra $)

For the last two terms we use the localization operator $\UU_>$ and $\UU_\leq$ introduced in \eqref{zmm02} to separate $Z_i=\UU_>Z_i+\UU_\leq Z_i$ with $L$ chosen below.
By the definition of $\UU_>$ and $\UU_\leq$ we know  
\begin{align}\label{zmm03}
\|\UU_>Z_i\|_{\bC^{-1+3\kappa}}\lesssim \|Z_i\|_{\bC^{-\frac12-\kappa}}2^{(-\frac12+4\kappa) L},\quad
\|\UU_\leq Z_i\|_{L^\infty}\lesssim \|Z_i\|_{\bC^{-\frac12-\kappa}}2^{(\frac12+2\kappa) L},
\end{align}
Using (ii) in Lemma \ref{lem:multi} followed by \eqref{ine:Lei} and Lemma \ref{lem:emb} and $H^{1-2\kappa}\subset H^{\frac34}\subset L^4$,  we find
\begin{align}
&\Big|\frac{\l}{N}\sum_{i,j=1}^N\la  [\mathscr{D}^{-1}P_j]\varphi_j\varphi_i,\UU_>Z_i \ra\Big|
\lesssim\frac{\l}{N}\sum_{i,j=1}^{N} \|\varphi_{i}\varphi_{j}\|_{B^{1-\kappa}_{\frac43,2}}\|\mathscr{D}^{-1}P_j\|_{H^{1-2\kappa}}\|\UU_>Z_i\|_{\bC^{-1+3\kappa}} \nonumber
\\
&\lesssim
\l\Big (\sum_{i=1}^{N}  \|\varphi_{i}\|_{H^1}^2\Big)^{\frac12}\Big(\frac{1}{N}\sum_{i=1}^{N} \| \varphi_{i}\|_{H^{\frac34}}^{2}  \Big )^{\frac12}\Big (\frac{1}{N}\sum_{i,j=1}^{N} \|\UU_>Z_i\|^2_{\bC^{-1+3\kappa}}\|\mathscr{D}^{-1}P_j\|^2_{H^{1-2\kappa}}\Big )^{\frac12} \nonumber
\\
&\lesssim
\l(\frac{1}{\sqrt{N}})^{\frac38}
\Big \|\frac{1}{\sqrt{N}}\sum_{i=1}^{N}Y_{i}^{2} \Big \|_{L^{2}}^{\frac58}
\Big (\sum_{j=1}^{N}\| \varphi_{j}\|_{H^1}^{2} \Big )^{\frac78}   \nonumber
\\
&\qquad\qquad \times \Big (\frac{1}{N}\sum_{j=1}^{N} \|Z_j\|_{\bC^{-\frac12-\kappa}}^{2}  \Big )^{\frac12}2^{(-\frac12+4\kappa) L}(1+\l^2 R_N^1)^{5/8}, \label{nL41}
\end{align}
where we used Lemma \ref{lem:interpolation}, Lemma \ref{lem:phi1} and \eqref{zmm03} in the last step.
By H\"older's inequality, \eqref{zmm03} and Lemma \ref{lem:phi1} we obtain that
\begin{align}\label{nL42}
&\Big|\frac{\l}{N}\sum_{i,j=1}^N\la  (\mathscr{D}^{-1}P_j)\varphi_j\varphi_i,\UU_\leq Z_i \ra \Big|\leq \frac{\l}{N}\sum_{i,j=1}^N\| \mathscr{D}^{-1}P_j\|_{L^2}\|\varphi_j\varphi_i\|_{L^2}\|\UU_\leq Z_i\|_{L^\infty}
\nonumber
\\
&\lesssim
 \l\Big \|\frac{1}{\sqrt{N}}\sum_{i=1}^{N}\varphi_{i}^{2} \Big \|_{L^{2}}\Big (\sum_{j=1}^{N}\| \mathscr{D}^{-1}P_j\|_{L^2}^{2} \Big )^{\frac12}\Big (\frac{1}{N}\sum_{j=1}^{N}\|Z_j\|_{\bC^{-\frac12-\kappa}}^{2}  \Big )^{\frac12}2^{(\frac12+2\kappa)L}
\\
&\lesssim
\l \Big \|\frac{1}{\sqrt{N}}\sum_{i=1}^{N}Y_{i}^{2} \Big \|_{L^{2}}\Big (\sum_{j=1}^{N}\|Y_{j}\|_{L^{2}}^{2} \Big )^{\frac12}(1+\l^2 R_N^1)^{\frac32} \Big (\frac{1}{N}\sum_{j=1}^{N}\|Z_j\|_{\bC^{-\frac12-\kappa}}^{2}  \Big )^{\frac12}2^{(\frac12+2\kappa)L}. \nonumber
\end{align}

Now we choose $L$ to  balance the competing contributions:
\begin{equation}
2^{L(-\frac12+4\kappa)}= \bigg (\sum_{j=1}^{N}\|Y_{j}\|_{L^{2}}^{2} \bigg )^{7/16}\bigg (\sum_{i=1}^{N}\| \varphi_{i}\|_{H^{1}}^{2} \bigg )^{-7/16}. \nonumber
\end{equation}
This choice of $L$ leads to
\begin{align}\label{nL4}
&\Big|\frac{\l}{N}\sum_{i,j=1}^N\la  \mathscr{D}^{-1}P_j\varphi_j\varphi_i, Z_i \ra \Big|
\lesssim
\l \Big \|\frac{1}{\sqrt{N}}\sum_{i=1}^{N}Y_{i}^{2} \Big \|_{L^{2}}
\Big (\frac{1}{N}\sum_{j=1}^{N}\|Z_j\|_{\bC^{-\frac12-\kappa}}^{2}  \Big )^{\frac12}
\nonumber
\\
&\qquad\qquad\qquad  \qquad
\times \bigg[
 \Big (\sum_{j=1}^{N}\|Y_{j}\|_{L^{2}}^{2} \Big )^{\frac12-\frac{7a}{16}}\Big (\sum_{j=1}^{N}\| \varphi_{j}\|_{H^1}^{2} \Big )^{\frac{7a}{16}}(1+\l^2 R_N^1)^{\frac32} \nonumber
\\
&\qquad\qquad\qquad\qquad\qquad
+ \Big (\sum_{j=1}^{N}\|Y_{j}\|_{L^{2}}^{2} \Big )^{\frac1{16}}\Big (\sum_{j=1}^{N}\| \varphi_{j}\|_{H^1}^{2} \Big )^{\frac7{16}}(1+\l^2 R_N^1)^{\frac58}
\bigg]\nonumber
\\&\le
C_\delta\l(1+\l^{55}) (R_N^1)^{25} \Big (1+\frac{1}{N}\sum_{j=1}^{N}\|Z_j\|_{\bC^{-\frac12-\kappa}}^{2}  \Big )^{9}\Big (\sum_{j=1}^{N}\|Y_{j}\|_{L^{2}}^{2} \Big ) \nonumber
\\
&\qquad\qquad +\delta \l \Big \|\frac{1}{\sqrt{N}}\sum_{i=1}^{N}Y_{i}^{2} \Big \|_{L^{2}}^{2}+\delta\Big ( \sum_{j=1}^{N}\| \varphi_{j}\|_{H^1}^{2} \Big),
\end{align}
where $a=\frac{\frac12+2\kappa}{\frac12-4\kappa}$ and the second line comes from \eqref{nL42} and the third line corresponds to \eqref{nL41}.

 Similar argument also implies that $|\frac{\l}{N}\sum_{i,j=1}^N\la  [\mathscr{D}^{-1}P_i]\varphi^2_j, Z_i \ra |$ can be bounded by the right hand side of \eqref{nL4}.
Combining the above estimates we obtain the result.
\end{proof}

\begin{proof}[Proof of Proposition \ref{Xi1}]
The desired estimate for $\frac{\l}{N}\sum_{i,j=1}^N\langle Y_j^2Y_i,Z_i\ra$ has been deduced in Lemma \ref{Cubic}. It remains to consider the other terms in $\Xi$.
We write the remaining terms in $\Xi$ as
 \begin{align*}
I_1^N\eqdef&-\frac{\l}{N}\sum_{i,j=1}^N\Big\langle (X_j^2+2X_jY_j)(X_i+Y_i)+Y_j^2X_i\;,\;Y_i\Big\rangle,
\\I_2^N\eqdef&-\frac{\l}{N}\sum_{i,j=1}^N\Big\langle(X_j^2+2X_jY_j)Z_i
+2(X_j+Y_j)X_iZ_j+2X_jY_iZ_j\;,\;Y_i\Big\rangle,
\\
I_3^N
\eqdef&
-\frac{\l}{N}\sum_{i,j=1}^N\Big\langle 
X_{\qj}\prec \UU_\leq{\cZ}^{\<2>}_{\qi j} 
+(X+Y)_{\qj}\succ{\cZ}^{\<2>}_{\qi j}
+X_{\qj}\circ {\cZ}^{\<2>}_{\qi j}
\;,\;Y_i\Big\rangle.
\end{align*}

\newcounter{UnifNb} 
\refstepcounter{UnifNb} 
{\sc Step} \arabic{UnifNb} \label{UnifN1b} \refstepcounter{UnifNb} (Estimate of $I_1^N$)

By symmetry with respect to $i$ and $j$ we write $I_1^N$ as
\begin{align*}
I_1^N=&-\frac{\l}{N}\sum_{i,j=1}^N\langle X_j^2X_i,Y_i\rangle
-\frac{\l}{N}\sum_{i,j=1}^N\langle X_j^2Y_i,Y_i\rangle-\frac{2\l}{N}\sum_{i,j=1}^N\langle X_jX_iY_j,Y_i\rangle\\&-\frac{3\l}{N}\sum_{i,j=1}^N\langle X_jY_j,Y_i^2\rangle\eqdef \sum_{k=1}^4I_{1k}^N,
\end{align*}
and we will prove that for every $k=1,\dots 4$,
\begin{align}
\|I_{1k}^N\|_{L_T^1}
&\le C_\delta\l \Big(\sum_{i=1}^N\|X_i\|^2_{L_T^2H^{\frac12-2\kappa}}\Big)
+\delta\l\Big \|\frac{1}{\sqrt{N}}\sum_{i=1}^NY_i^2\Big\|^2_{L_T^2L^2}
		\label{zmm011}
\\
&+C_\delta\l \Big(\frac{1}{N}\sum_{j=1}^N\| X_j\|_{C_T\bC^{\frac12-\kappa}}^2+1\Big)^2\Big(\sum_{i=1}^N\|Y_i\|^2_{L_T^2L^2}\Big).
\nonumber
\end{align}
We consider each term separately. For $I_{11}^N$,
applying H\"older's inequality we have
	\begin{align*}
\|\langle X_j^2X_i,Y_i\rangle\|_{L_T^1}\lesssim \int_0^T\|X_j\|_{L^\infty}^2\|X_i\|_{L^2}\|Y_i\|_{L^2}\dif s\lesssim \|X_j\|_{C_TL^\infty}^2\|X_i\|_{L_T^2L^2}\|Y_i\|_{L_T^2L^2}
	\end{align*}
so by embedding $\bC^{\frac12-\kappa} \subset L^\infty$ we have
\begin{align*}
\|I_{11}^N\|_{L_T^1}
&\lesssim \frac{\l}{N}\sum_{i,j=1}^N\| X_j\|_{C_T\bC^{\frac12-\kappa}}^2\|X_i\|_{L_T^2L^2}\|Y_i\|_{L_T^2L^2}
\\
&\lesssim \l\Big(\frac{1}{N}\sum_{j=1}^N\| X_j\|_{C_T\bC^{\frac12-\kappa}}^2\Big)\Big(\sum_{i=1}^N\|X_i\|^2_{L_T^2L^2}\Big)^{\frac12}\Big(\sum_{i=1}^N\|Y_i\|^2_{L_T^2L^2}\Big)^{\frac12},
\end{align*}
which by Young's inequality gives \eqref{zmm011} for $I_{11}^N$.
Applying H\"{o}lder's inequality  in a similar way as above we find
\begin{align*}
\|I_{12}^N\|_{L_T^1} & \lesssim \frac{\l}{N}\sum_{i,j=1}^N\| X_j\|_{C_T\bC^{\frac12-\kappa}}^2\|Y_i\|^2_{L_T^2L^2},
\\
\|I_{13}^N\|_{L_T^1} &\lesssim \frac{\l}{N}\sum_{i,j=1}^N\| X_j\|_{C_T\bC^{\frac12-\kappa}}\| X_i\|_{C_T\bC^{\frac12-2\kappa}}\|Y_i\|_{L_T^2L^2}\|Y_j\|_{L_T^2L^2}
\\&\lesssim \l\Big(\frac{1}{N}\sum_{j=1}^N\| X_j\|_{C_T\bC^{\frac12-\kappa}}^2\Big)\Big(\sum_{i=1}^N\|Y_i\|^2_{L_T^2L^2}\Big),
\end{align*}
which give \eqref{zmm011} for $I_{12}^N$ and $I_{13}^N$.
Moreover by H\"older's inequality we have
\begin{align*}
\|I_{14}^N\|_{L_T^1} & \lesssim \l\Big \|\frac{1}{\sqrt{N}}\sum_{i=1}^NY_i^2\Big\|_{L_T^2L^2}\Big(\sum_{j=1}^N\|Y_j\|^2_{L_T^2L^2}\Big)^{\frac12}
\Big(\frac{1}{N}\sum_{j=1}^N\|X_j\|^2_{C_T\bC^{\frac{1}{2}-\kappa}}\Big)^{\frac12}
\\ & \lesssim \delta\l \Big \|\frac{1}{\sqrt{N}}\sum_{i=1}^NY_i^2\Big\|^2_{L_T^2L^2}+\l\Big(\sum_{i=1}^N\|Y_i\|^2_{L_T^2L^2}\Big)
\Big(\frac{1}{N}\sum_{j=1}^N\|X_j\|^2_{C_T\bC^{\frac{1}{2}-\kappa}}\Big),
\end{align*}
which implies \eqref{zmm011}.

{\sc Step} \arabic{UnifNb} \label{UnifN1b} \refstepcounter{UnifNb} (Estimate of $I_2^N$)
By symmetry with respect $i$ and $j$ we write $I_2^N$ as
\begin{align*}
I_2^N=&-\frac{\l}{N}\sum_{i,j=1}^N\langle X_j^2Z_i,Y_i\rangle
-\frac{4\l}{N}\sum_{i,j=1}^N\langle X_jY_jZ_i,Y_i\rangle
-\frac{2\l}{N}\sum_{i,j=1}^N\langle X_jX_iZ_j,Y_i\rangle
\\&-\frac{2\l}{N}\sum_{i,j=1}^N\langle X_jY_iZ_j,Y_i\rangle
\eqdef \sum_{i=1}^4I_{2i}^N.
\end{align*}
In the following we show that  the $L_T^1$-norm of each term can be bounded by
\begin{align}\label{zmm011-2}
&\delta\Big(\sum_{i=1}^N\|Y_i\|^2_{L_T^2H^{1-2\kappa}}\Big)+C_\delta\l^2 Q_N^{41}
\\
&+C_\delta\l^2\Big(\sum_{j=1}^N\|Y_j\|^2_{L_T^2L^2}\Big)\Big[\Big(\frac{1}{N^2}\sum_{i,j=1}^N\|X_jZ_i\|^2_{C_T\bC^{-\frac12-\kappa}}\Big)+
\Big(\frac{1}{N}\sum_{j=1}^N\|X_jZ_j\|^2_{C_T\bC^{-\frac12-\kappa}}\Big)\Big],
\notag
\end{align}
with $Q_N^{41}$ given by the third line in the definition of $Q_N^4$ in \eqref{zrx1}.
Using (ii) in Lemma \ref{lem:multi} followed by H\"older's inequality for the summation over $i$ we have
\begin{align*}
\|I_{21}^N\|_{L_T^1}
&\lesssim \frac{\l}{N}\sum_{i,j=1}^N\|X_j^2Z_i\|_{L_T^2H^{-\frac12-2\kappa}}\|Y_i\|_{L_T^2H^{1-2\kappa}}
\\
&\lesssim\l\Big(\frac{1}{N^2}\sum_{i=1}^N
\Big(\sum_{j=1}^N\|X_j^2Z_i\|_{L^2_TH^{-\frac12-2\kappa}}\Big)^{2}\Big)^{\frac12}
\Big(\sum_{i=1}^N\|Y_i\|^2_{L_T^2H^{1-2\kappa}}\Big)^{\frac12},
\end{align*}
which by Young's inequality gives the desired bound for $\|I_{21}^N\|_{L_T^1}$. Similarly,
\begin{align*}
&\|I_{23}^N\|_{L_T^1}\lesssim\l\Big(\frac{1}{N^2}\sum_{i=1}^N
\Big(\sum_{j=1}^N\|X_jX_iZ_j\|_{L^2_TH^{-\frac12-2\kappa}}\Big)^{2}\Big)^{\frac12}
\Big(\sum_{i=1}^N\|Y_i\|^2_{L_T^2H^{1-2\kappa}}\Big)^{\frac12},
\end{align*}
which implies the bound for $\|I_{23}^N\|_{L_T^1}$.
Using (ii) in Lemma \ref{lem:multi} followed by H\"older's inequality for the summation over $i, j$ and \eqref{ine:Lei}, we deduce
\begin{align*}
\|I_{22}^N\|_{L_T^1}
&\lesssim \l \Big(\frac{1}{N^2}\sum_{i,j=1}^N\|X_jZ_i\|^2_{C_T\bC^{-\frac12-\kappa}}\Big)^{\frac12}\Big(\sum_{i,j=1}^N\|Y_iY_j\|_{L_T^1B^{\frac12+\kappa}_{1,1}}^2\Big)^{\frac12}
\\
&\lesssim \l \Big(\frac{1}{N^2}\sum_{i,j=1}^N\|X_jZ_i\|^2_{C_T\bC^{-\frac12-\kappa}}\Big)^{\frac12}\Big(\sum_{j=1}^N\|Y_i\|^2_{L_T^2H^{1-2\kappa}}\Big)^{\frac12}\Big(\sum_{i=1}^N\|Y_j\|^2_{L_T^2L^2}\Big)^{\frac12}
\end{align*}
which by Young's inequality gives the required bound for $\|I_{22}^N\|_{L_T^1}$.
Similarly we have
\begin{align*}
\|I_{24}^N\|_{L_T^1}
&\lesssim\l \Big(\frac{1}{N}\sum_{j=1}^N\|X_jZ_j\|^2_{C_T\bC^{-\frac12-\kappa}}\Big)^{\frac12}\Big(\sum_{i=1}^N\|Y_i^2\|_{L_T^1B^{\frac12+\kappa}_{1,1}}\Big)
\\
&\lesssim \l \Big(\frac{1}{N}\sum_{i,j=1}^N\|X_jZ_j\|^2_{C_T\bC^{-\frac12-\kappa}}\Big)^{\frac12}\Big(\sum_{j=1}^N\|Y_i\|^2_{L_T^2H^{1-2\kappa}}\Big)^{\frac12}\Big(\sum_{i=1}^N\|Y_i\|^2_{L_T^2L^2}\Big)^{\frac12},
\end{align*}
which implies the bound for $I_{24}^N$.


{\sc Step} \arabic{UnifNb} \label{UnifN1b} \refstepcounter{UnifNb} (Estimate of $I_3^N$)
We write $I_3^N =\sum_{i=1}^4I_{3i}^{N} $
with
\begin{align*}
I_{31}^N\eqdef&-\frac{\l}{N}\sum_{i,j=1}^N
\Big\langle X_{\qj} \prec \UU_\leq{\cZ}^{\<2>}_{\qi j}
\;,\;Y_i\Big\rangle,
\qquad
I_{32}^N
\eqdef
-\frac{\l}{N}\sum_{i,j=1}^N\Big\langle X_{\qj}\succ{\cZ}^{\<2>}_{\qi j}
\;,\;Y_i\Big\ra,
\\
I_{33}^N\eqdef&-\frac{\l}{N}\sum_{i,j=1}^N\Big\langle X_{\qj}\circ {\cZ}^{\<2>}_{\qi j}\;,\;Y_i\Big\rangle,
\qquad\qquad
I_{34}^N\eqdef
-\frac{\l}{N}\sum_{i,j=1}^N\Big\la Y_{\qj} \succ{\cZ}^{\<2>}_{\qi j} \;,\;Y_i\Big\ra.
\end{align*}
In the following we bound the $L_T^1$-norm of  each term by
\begin{align}\label{zmm011-3}
&\delta\Big(\sum_{i=1}^N\|Y_i\|^2_{L_T^2H^{1-2\kappa}}\Big)+\l^2 Q_N^{4}
\\
&+\l^{\frac1{1-\theta}}\Big(\sum_{j=1}^N\|Y_j\|^2_{L_T^2L^2}\Big)\Big(\frac{1}{N^2}\sum_{i,j=1}^N(\|{\cZ}^{\<2>}_{ij}\|_{C_T\bC^{-1-\kappa}}^{\frac1{1-\theta}}+\|{\CZ}^{\<2>}_{jj}\|_{C_T\bC^{-1-\kappa}}^{\frac1{1-\theta}})\Big) \notag
\end{align}
with $\theta=\frac{\frac12+\kappa}{1-2\kappa}$.
Recall that a product with the $[i],[j]$ notation expands into two terms.
By definition of $\UU_\leq$ we have
$$\|\UU_\leq{\cZ}^{\<2>}_{ij}\|_{C_T\bC^{-1+3\kappa}}\lesssim
2^{4\kappa L}\|{\cZ}^{\<2>}_{ij}\|_{C_T\bC^{-1-\kappa}},$$
with $L$ given in \eqref{L}, which
combined with Lemma \ref{lem:para} implies that for the first term in $I_{31}^N$
\begin{align*}
&\int_0^T\Big|\frac{\l}{N}\sum_{i,j=1}^N\langle X_j\prec \UU_\leq{\cZ}^{\<2>}_{ij} ,Y_i\rangle\Big|\dif s
\\&\lesssim \frac{\l}{N}\sum_{i,j=1}^N\| X_j\|_{L^2_TH^{\frac12-2\kappa}} \|\UU_\leq{\cZ}^{\<2>}_{ij}\|_{C_T\bC^{-1+3\kappa}}\|Y_i\|_{L^2_TH^{1-2\kappa}}
\\&\lesssim 2^{4\kappa L}\l\Big(\frac{1}{N^2}\sum_{i,j=1}^N\|{\cZ}^{\<2>}_{ij}\|^2_{C_T\bC^{-1-\kappa}}\Big)^{\frac12}\Big(\sum_{i,j=1}^N\| X_j\|^2_{L^2_TH^{\frac12-2\kappa}}\|Y_i\|^2_{H^{1-2\kappa}}\Big)^{\frac12}.
\end{align*}
Similarly, for the second term in $I_{31}^N$,
\begin{align*}
&\int_0^T\Big|\frac{\l}{N}\sum_{i,j=1}^N\langle X_i\prec \UU_\leq{\CZ}^{\<2>}_{jj},Y_i\rangle\Big|\dif s
\\
&\lesssim 2^{4\kappa L}\l\Big(\frac{1}{N}\sum_{i,j=1}^N\|{\cZ}^{\<2>}_{jj}\|_{C_T\bC^{-1-\kappa}}\Big)\Big(\sum_{i=1}^N\| X_i\|^2_{L^2_TH^{\frac12-2\kappa}}\Big)^{\frac12}\Big(\sum_{i=1}^N\|Y_i\|^2_{H^{1-2\kappa}}\Big)^{\frac12},
\end{align*}
which by Young's inequality shows that   $\|I_{31}^N\|_{L_T^1}$ is bounded by the first line of $Q_N^4$ in \eqref{zrx1} multiplied by $\l^2$ and $\delta\sum_{i=1}^N\|Y_i\|^2_{L_T^2H^{1-\kappa}}$.
Moreover, by Lemma \ref{lem:para} and Lemma \ref{lem:multi} we have
\begin{align*}
\|I_{32}^N\|_{L_T^1}
&\lesssim
\frac{\l}{N}\sum_{i,j=1}^N
\|X_{\qj}\|_{L^2_TH^{\frac12-2\kappa}}\|{\cZ}^{\<2>}_{\qi j}\|_{C_T\bC^{-1-\kappa}}
\|Y_i\|_{L^2_TH^{1-2\kappa}}
\\
&\lesssim \l\Big[\Big(\frac{1}{N^2}\sum_{i,j=1}^N\|{\cZ}^{\<2>}_{ij}\|_{C_T\bC^{-1-\kappa}}^2\Big)^{\frac12}+
\Big(\frac{1}{N}\sum_{j=1}^N\|{\cZ}^{\<2>}_{jj}\|_{C_T\bC^{-1-\kappa}}^2\Big)^{\frac12}\Big]
\\
&\qquad\times\Big(\sum_{i=1}^N\|X_i\|^2_{L_T^2H^{\frac12-2\kappa}}\Big)^{\frac12} \Big(\sum_{i=1}^N\|Y_i\|^2_{L_T^2H^{1-2\kappa}}\Big)^{\frac12},
\end{align*}
 which by Young's inequality shows that   $\|I_{32}^N\|_{L_T^1}$ is bounded by the second line of $Q_N^4$ in \eqref{zrx1} multiplied by $\l^2$ and $\delta\sum_{i=1}^N\|Y_i\|^2_{L_T^2H^{1-2\kappa}}$. Also we have
\begin{align*}
\|I_{33}^N\|_{L_T^1}
&\lesssim
\frac{\l}{N}\sum_{i,j=1}^N
\|X_{\qi}\circ \CZ^{\<2>}_{\qj j}\|_{L^2_TH^{-\frac12-2\kappa}}
\|Y_i\|_{L^2_TH^{1-2\kappa}}
\\
&\lesssim
\l\Big[
\frac{1}{N^2}\sum_{i=1}^N\Big(\sum_{j=1}^N
\|X_{\qi}\circ \CZ^{\<2>}_{\qj j}\|_{L^2_TH^{-\frac12-2\kappa}}
\Big)^{2}
\Big]^{\frac12}\Big(\sum_{i=1}^N\|Y_i\|^2_{L^2_TH^{1-2\kappa}}\Big)^{\frac12},
\end{align*}
which by Young's inequality shows that   $\|I_{33}^N\|_{L_T^1}$ is bounded by the last line of $Q_N^4$ in \eqref{zrx1} multiplied by $\l^2$ and $\delta\sum_{i=1}^N\|Y_i\|^2_{L_T^2H^{1-2\kappa}}$. 

Furthermore, using Lemma \ref{lem:para} and  Lemma \ref{lem:multi} (ii),
\begin{align*}
&\|I_{34}^N\|_{L_T^1}
\lesssim\frac{\l}{N}\sum_{i,j=1}^N\| Y_j\|_{L_T^2H^{\frac12+\kappa}}\| Y_i\|_{L_T^2H^{\frac12+\kappa}}\Big(\|{\cZ}^{\<2>}_{ij}\|_{C_T\bC^{-1-\kappa}}+\|{\CZ}^{\<2>}_{jj}\|_{C_T\bC^{-1-\kappa}}\Big)
\\&\lesssim\l\Big(\frac{1}{N^2}\sum_{i,j=1}^N(\|{\cZ}^{\<2>}_{ij}\|_{C_T\bC^{-1-\kappa}}^2+\|{\CZ}^{\<2>}_{jj}\|_{C_T\bC^{-1-\kappa}}^2)\Big)^{\frac12}\Big(\sum_{i=1}^N\| Y_i\|^{2}_{H^{1-2\kappa}}\Big)^\theta\Big(\sum_{i=1}^N\| Y_i\|^{2}_{L^2}\Big)^{1-\theta},
\end{align*}
with $\theta=\frac{\frac12+\kappa}{1-2\kappa}$, where we used H\"older inequality for the sum over $i$ and $j$.
Thus the required bound for $\|I_{34}^N\|_{L_T^1}$ follows by Young's inequality.

Proposition \ref{Xi1} now follows from \eqref{zmm011}, \eqref{zmm011-2}
and \eqref{zmm011-3} and Lemma \ref{Cubic}.
\end{proof}


\section{Convergence of measures and  tightness of observables}\label{sec:inv}

Now we return to the measure $\nu^N$ in \eqref{e:Phi_i-measure} which is invariant under \eqref{eq:main}.  In fact, by standard argument (c.f. \cite{HM18}) the solutions $(\Phi_i)_{1\leq i\leq N}$ to \eqref{eq:21} form a Markov process on $(\bC^{-\frac12-\kappa})^{N}$ which, by strong Feller property in \cite{HM18} and irreducibility in \cite{HS19}, will turn out to admit a unique invariant measure. 

 One goal in this section is to 
  show that for sufficiently large mass or small $\l$, in the limit $N \to \infty$, the marginals of $\nu^{N}$ are simply products of the Gaussian measure $\nu$, which is the unique invariant measure to the linear dynamics \eqref{eq:li1}. Furthermore, the convergence rate is given in Theorem \ref{th:main}, which can be seen as a very non elementary version of the Projective Central Limit Theorem. The idea in this section follows similarly as in our previous work \cite[Sections~5-6.1]{SSZZ20}. Since the 3D case is far more complicated than its counterpart in 2D, our work  heavily relies on the computations from Section~\ref{sec:sto} to Section~\ref{sec:uni} for the remainder $X$ and $Y$, but we leverage these estimates with consequences of stationarity.

\br  Using lattice approximations (see e.g. \cite{GH18}, \cite{HM18a, ZZ18})
one can show that the measure $\dif\nu^N(\Phi)$ indeed has the form \eqref{e:Phi_i-measure}, with suitable renormalization that is consistent with \eqref{eq:ap}, i.e.
formally \footnote{Namely, the integral,  $\nabla$ and  $\mathcal D\Phi$
	should be approximated by lattice summation, finite difference and Lebesgue measure.}
it is
$$\exp\Big(-\int \sum_{j=1}^N|\nabla \Phi_j|^2+\Big(m-\frac{N+2}N\l a_\eps+\frac{3(N+2)}{N^2}\l^2 \tilde b_\eps\Big)\sum_{j=1}^N\Phi_j^2+\frac{\l}{2N}\Big(\sum_{j=1}^N\Phi_j^2\Big)^2\dif x\Big)\mathcal D \Phi$$
divided by a normalization constant. 
\er 

\subsection{Tightness of the measure}
\label{sec:Tightness}

It will be convenient to have a stationary coupling of the linear and non-linear dynamics \eqref{eq:li1} and \eqref{eq:21}, which is stated in the following lemma {and the proof follows essentially the same as in \cite[Lemma 5.7]{SSZZ20}.}

\bl\label{lem:zz1} For $(m,\l)\in (0,\infty)\times [0,\infty)$, there exists a unique invariant measure $\nu^N$ on $(\bC^{-\frac{1}{2}-\kappa})^N$ to \eqref{eq:21}. Furthermore, there exists a  stationary process $(\tilde{\Phi}_i^N, Z_i)_{1\leq i\leq N}$ such that the components $\tilde{\Phi}_i^N, Z_i$ are stationary  solutions to \eqref{eq:21} and \eqref{eq:li1} respectively and
$$\E\|\tilde{\Phi}_i^N(0)-Z_i(0)\|_{L^2}^2\lesssim1.$$
Here the implicit constant may depend on $\l, m$ and $N$.
\el
\begin{proof}
Let $ \Phi_i$ and $\bar Z_i$ be solutions to \eqref{eq:21} and \eqref{eq:li1} with general initial conditions, respectively. We also recall that $Z_i$ is the stationary solution to \eqref{eq:li1}.
By the general results of \cite{HM18}, 
$(\Phi_i, \bar Z_i)_{1\leq i\leq N}$ is a Markov process on $(\bC^{-\frac12-\kappa})^{2N}$, and we denote by $(P_t^N)_{t\geq0}$ the associated Markov semigroup.  To derive the desired structural properties about the limiting measure, we will follow the Krylov-Bogoliubov construction with a specific choice of initial condition that allows to exploit the uniform estimate from Appendix~\ref{sec:extra}.
Namely, we denote by $\Phi_{i}$ the solution to \eqref{eq:21} starting from
 $Z_{i}(0)$ where $Z_i$ is  the stationary solution to \eqref{eq:li1}, so that the process $\Phi_i-\bar Z_i$ starts from the origin. In this case $\bar Z_i$ is the same as the stationary solution $Z_i$.
By Lemma \ref{lem:A2}, for every $T\geq 1$ and $\kappa>0$
\begin{align}
\int_0^T\mathbf{E}\Big(\frac{1}{N}\sum_{i=1}^N\| (\Phi_i-Z_i)(t)\|_{L^2}^2\Big)\dif t
+\int_0^T\mathbf{E}\Big(\frac{1}{N}\sum_{i=1}^N\| (\Phi_i- Z_i)(t)\|_{\bC^{-\frac{1+\kappa}2}}^2\Big)\dif t & \lesssim T,
\label{sszz2}
\end{align}
where the implicit constant is independent of $T$.  
By \eqref{sszz2} together with Lemma~\ref{thm:renorm43}
we have
\begin{align*}
\int_0^T\mathbf{E}\Big(\frac{1}{N}\sum_{i=1}^N\|\Phi_i(t)\|_{\bC^{-\frac{1+\kappa}2}}^2\Big)\dif t
+\int_0^T\mathbf{E}\Big(\frac{1}{N}\sum_{i=1}^N\|Z_i(t)\|_{\bC^{-\frac{1+\kappa}2}}^2\Big)\dif t\lesssim T.
\end{align*}

Defining $R_t^N:=\frac{1}{t}\int_0^tP_r^N\dif r$, the above estimates and the compactness of the embedding $\bC^{-\frac{1+\kappa}2} \hookrightarrow \bC^{-\frac12-\kappa}$ imply the induced laws of $\{R_t^N \}_{t\geq0}$ started from $({Z}_i(0),{Z}_i(0))$ are tight on $(\bC^{-\frac12-\kappa})^{2N}$.
Furthermore, by  continuity w.r.t. initial data, it is easy to check that $(P_t^N)_{t\geq0}$ is  Feller on $(\bC^{-\frac12-\kappa})^{2N}$.
By Krylov-Bogoliubov existence theorem (see \cite[Corollary 3.1.2]{DZ96}) , these laws converge weakly in $(\bC^{-\frac12-\kappa})^{2N}$ along a subsequence $t_k\to\infty$ to an invariant measure $\pi_N$ for $(P_t^N)_{t\geq0}$. The desired stationary process $(\tilde\Phi_i^N, Z_i)_{1\leq i\leq N}$ is defined to be the unique solution to \eqref{eq:21} and \eqref{eq:li1} obtained by sampling the initial datum $(\varphi_{i},z_{i})_{i}$ from $\pi_{N}$. By \eqref{sszz2} we deduce
\begin{align*}
\E^{\pi_N}&\|\Phi_i(0)-Z_i(0)\|_{L^2}^2=\E^{\pi_N}\sup_{ f }\la\Phi_i(0)-Z_i(0),f\ra^2
\\
&=\E\sup_{f}\lim_{T\to\infty}\bigg[\frac1T\int_0^T\la \Phi_i(t)-Z_i(t),f\ra\dif t\bigg]^2
\\
&\lesssim \lim_{T\to\infty}\frac1T\int_0^T\E\|\Phi_i(t)-Z_i(t)\|_{L^2}^2\dif t\lesssim1,
\end{align*}
where $\E^{\pi_N}$ denotes the expectation w.r.t. $\pi_N$ and $\sup_{f}$ is over smooth functions $f$ with $\|f\|_{L^2}\leq 1$. 

 Finally, we project onto the first component and consider the map $\bar{\Pi}_1:\mathcal{S}'(\mathbb{T}^3)^{2N}\rightarrow \mathcal{S}'(\mathbb{T}^3)^N$ defined through $\bar{\Pi}_1(\Phi,Z)=\Phi$.  Letting $\nu_N=\pi_N\circ \bar{\Pi}_1^{-1}$ yields an invariant measure to \eqref{eq:21}, and uniqueness follows from the general results of \cite{HM18} and \cite{HS19}.
\end{proof}

By lattice approximation it is possible to prove
 that the field $\Phi$ under $\nu^N$ is $O(N)$ invariant, translation invariant and satisfies reflection positivity
  (see e.g. \cite{GH18} for $\Phi^4_3$). 
 As we do not prove convergence of lattice approximations for sigma models in this paper, 
 we  give a proof of $O(N)$ and translation invariances 
using ergodicity of the dynamics. 

\bp The field $\Phi$ under $\nu^N$ is $O(N)$ invariant and translation invariant  in law.
\ep
\begin{proof} 
The equation \eqref{eq:ap} is $O(N)$ invariant and translation invariant, namely, for any $A\in O(N)$ and solution $\Phi_\eps$, $A\Phi_\eps$  satisfies the same equation with driven noise $A\xi_\eps \stackrel{law}{=} \xi_\eps$ and initial data $A\Phi_\eps(0)$, 
and  similar property holds for translations of $\Phi_\eps$.
		 By uniqueness of the solutions to \eqref{eq:ap}, the solution $\Phi_\eps(t)$ starting from zero satisfy $O(N)$ invariance and translation invariance. Hence, the limit $\Phi(t)$ of $\Phi_\eps(t)$ as $\eps\to0$ is $O(N)$ and translation invariant.
	As  $\nu^N$ is the unique invariant measure to \eqref{eq:main}, the  law of  $\Phi(t)$   converges to $\nu^N$ as $t\to\infty$. Hence,  the field $\Phi$ under $\nu^N$ is $O(N)$ and translation invariant.
\end{proof}

\begin{remark}\label{rem:lat}Lemma \ref{lem:zz1} also gives a dynamical / SPDE construction of  the measure $\nu^N$.
\end{remark}

Next we study tightness of the marginal laws of $\nu^{N}$ over $\mathcal{S}'(\mathbb{T}^3)^N$. 
Recall that $\nu^{N,i}\eqdef \nu^N\circ \Pi_i^{-1}$ is the marginal law of the $i^{th}$ component
 and $\nu^{N}_k\eqdef \nu^N\circ (\Pi^{(k)})^{-1}$  is the marginal law of the first $k$ components,
  with $\Pi_i$ and $\Pi^{(k)}$ defined in \eqref{e:Pr-Pi} and \eqref{e:Pr-Pik}, respectively. 

\bt\label{th:con} For $(m,\l)\in (0,\infty)\times [0,\infty)$,
$\{\nu^{N,i}\}_{N\geq1}$ form a tight set of probability measures on $H^{-\frac{1}{2}-\kappa}$ for $\kappa>0$. 
\et

\begin{proof}
	Let  $(\tilde \Phi_i, Z_i)_{1 \leq i \leq N}$ be 
	as in Lemma \ref{lem:zz1}.    
	By the compact embedding of $H^{\frac12-2\kappa} \hookrightarrow H^{-\frac12-\kappa}$ for $\kappa>0$ small enough and the stationarity of $\tilde{\Phi}_{i}$, it suffices to show that the second moment of $\|\tilde{\Phi}_i(0)-Z_i(0)\|_{H^{\frac12-2\kappa}}$ is bounded uniformly in $N$.  To this end,
	let $Y_i=\tilde{\Phi}_i-Z_i-X_i$ with $X_i$ defined in \eqref{eq:211}.
It holds that
\begin{equation}\label{bdd:1}
\aligned
\E\|\tilde \Phi_i&(0)  -Z_i(0)\|_{H^{\frac{1}{2}-2\kappa}}^2
=\frac{1}{t}\int_0^t\E\|\tilde \Phi_i(s)-Z_i(s)\|_{H^{\frac{1}{2}-2\kappa}}^2 \dif s
\\
&=\frac{1}{t}\int_0^t\E\|X_i(s)+Y_i(s)\|_{H^{\frac{1}{2}-2\kappa}}^2 \dif s
\\
&\leq \frac{2}{t}\int_0^t\E\|X_i(s)\|_{H^{\frac{1}{2}-2\kappa}}^2 \dif s+\frac{2}{t}\int_0^t\E\|Y_i(s)\|_{H^{\frac{1}{2}-2\kappa}}^2 \dif s.
\endaligned
\end{equation}
For the first term,
by \eqref{bdnx} of Lemma~\ref{lem:X2}, we have
\begin{align}\label{bdd:2}
\frac{1}{t}\int_0^t\E\|X_i(s)\|_{H^{\frac{1}{2}-2\kappa}}^2 \dif s=\frac{1}{tN}\int_0^t\E\sum_{i=1}^N\|X_i(s)\|_{H^{\frac{1}{2}-2\kappa}}^2 \dif s
\lesssim \frac{1}{tN}\E Q_N^0\lesssim \frac{C_t}{tN},
\end{align}
with $Q_N^0$ defined in Lemma \ref{lem:sto} and $\E Q_N^0\lesssim C_t$. 
For the second term, by $O(N)$ symmetry in law
\begin{align}\label{bdd:22}
\frac{1}{t}\int_0^t\E\|Y_i(s)\|_{H^{\frac{1}{2}-2\kappa}}^2 \dif s=\frac{1}{tN}\int_0^t\E\Big(\sum_{i=1}^N\|Y_i(s)\|_{H^{\frac{1}{2}-2\kappa}}^2\Big) \dif s .
\end{align}
Using Theorem \ref{Y:T4a}\footnote{Since we do not need to choose $m$ and $\l$ in the proof, we can use Theorem \ref{Y:T4a} for $m>0,\l\geq0$ with the proportional constant depending on $m,\l$. } we deduce that
\begin{align*}
\frac{1}{8N}\int_0^t\E\sum_{i=1}^N\|Y_i(s)\|_{H^{1-2\kappa}}^2 \dif s
&+\frac{m}{N}\int_0^t\E\sum_{i=1}^N\|Y_i(s)\|_{L^2}^2 \dif s+\frac{\l}{2N^2}\int_0^t\E\Big\|\sum_{i=1}^NY_i^2\Big\|_{L^2}^2\dif s\no
\\
& \leq 4\E\|\tilde{\Phi}_i(0)-Z_i(0)\|_{L^2}^2+\frac{4}{N}\sum_{i=1}^N\E\|X_i(0)\|_{L^2}^2+C.
\end{align*}
By definition of $X_i$ in \eqref{eq:211} (and definition of $\cI$ in Section~\ref{sec:heat}) we have $X_i(0)=-\frac\l{N}\sum_{j=1}^N\tilde\cZ^{\<30>}_{ijj}(0)$, which gives 
\begin{align*}
\frac1N\sum_{i=1}^N\E\|X_i(0)\|_{L^2}^2\lesssim \frac{\l^2}{N^2}\sum_{i,j=1}^N\E\|\tilde{\cZ}^{\<30>}_{ijj}(0)\|_{L^2}^2\lesssim1.
\end{align*}
Then we conclude that
\begin{align*}
&\E\|\tilde{\Phi}_i(0)-Z_i(0)\|_{H^{\frac{1}{2}-2\kappa}}^2\leq C_t+\frac{64}{t}\E\|\tilde{\Phi}_i(0)-Z_i(0)\|_{L^2}^2 \;.
\end{align*}
Lemma \ref{lem:zz1} implies that $\E\|\tilde{\Phi}_i(0)-Z_i(0)\|_{L^2}^2$ is finite.  Choosing $t=128$ and using $\|\cdot\|_{L^2}\leq \|\cdot\|_{H^{\frac{1}{2}-2\kappa}}$, we find $\E\|\tilde{\Phi}_i(0)-Z_i(0)\|_{H^{\frac{1}{2}-2\kappa}}^2\leq C_t$. The result then follows by the bound for $Z_i(0)$ from Lemma \ref{thm:renorm43}. 
\end{proof}

\subsection{Convergence of measures}
\label{sec:Con}

In the following we prove convergence of the measures to the unique invariant measure 
(which requires $m$ large enough or $\l$ small)
using the estimate in Theorem \ref{Y:T4}.

Define the $(H^{-\frac12-\kappa})^k$-Wasserstein distance for the measures on $(H^{-\frac12-\kappa})^k$
\begin{equ}[e:W2]
\mathbb{W}_{2,k}(\nu_1,\nu_2):=\inf_{\pi\in\mathscr{C}(\nu_1,\nu_2)}\left(\int\|\phi-\psi\|_{(H^{-\frac12-\kappa})^k}^2\pi(\dif \phi,\dif \psi)\right)^{1/2},
\end{equ}
where $\mathscr{C}(\nu_1,\nu_2)$ denotes the set of all couplings of $\nu_1, \nu_2$ satisfying $\int\|\phi\|_{(H^{-\frac12-\kappa})^k}^2\nu_i(\dif \phi)<\infty$ for $i=1,2$. Recall that $\nu=\cN(0,\frac12(m-\Delta)^{-1})$.

\bt\label{th:main} Let $(m,\l)\in [1,\infty)\times [0,\infty)$. There exists $c_0>0$ such that for $m\geq 1+\l(1+\l^{59}) c_0$ 
\begin{equation}
\mathbb{W}_{2,k}(\nu^{N}_k,\nu^{\otimes k}) \leq C_kN^{-\frac{1}{2}}. \label{West}
\end{equation}
\et
\begin{proof}
By Lemma \ref{lem:zz1} we may construct a stationary coupling $(\tilde \Phi_i, Z_i)$ of $\nu_N$ and $\nu$ whose components satisfy \eqref{eq:21} and \eqref{eq:li1}, respectively.  The stationarity of the joint law of $(\tilde \Phi_i, Z_i)$ implies that also $\tilde \Phi_i-Z_i$ is stationary.  We now claim that
\begin{equation}
\E \|\tilde \Phi_i(0)-Z_i(0)\|_{H^{\frac12-2\kappa}}^{2} \leq CN^{-1} \label{se8},
\end{equation}
which implies \eqref{West} by definition of the Wasserstein metric and the embedding $H^{\frac12-2\kappa} \hookrightarrow H^{-\frac12-\kappa}$.

Recall the inequality \eqref{bdd:1}, for which the RHS is bounded by \eqref{bdd:2}+\eqref{bdd:22}.

Using Theorem \ref{Y:T4} and Lemma \ref{lem:zmm} (with $R_N^1, R_N^2, Q_N^5$  introduced in Lemma \ref{lem:zz2}, Proposition \ref{theta1}  and Lemma \ref{lem:zmm} respectively), we deduce that
\begin{align}
&\frac18\int_0^t\E\Big(\sum_{i=1}^N\|Y_i(s)\|_{H^{1-2\kappa}}^2\Big) \dif s+(m-1)\int_0^t\E\Big(\sum_{i=1}^N\|Y_i(s)\|_{L^2}^2\Big) \dif s+\frac{\l}{N}\E\int_0^t\Big\|\sum_{i=1}^NY_i^2\Big\|_{L^2}^2\dif s\no
\\
&\leq\E\Big(\sum_{i=1}^N\|Y_i(0)\|_{L^2}^2\Big)+C\int_0^t\E\Big[\Big(\sum_{i=1}^N\|Y_i(s)\|_{L^2}^2\Big)\Big(\l(1+\l^{59})(R_N^1+R_N^2+Q_N^5)\Big)\Big] \dif s+C(t,\l)\no
\\
&\leq 2N\E\|\tilde{\Phi}_i(0)-Z_i(0)\|_{L^2}^2+2\sum_{i=1}^N\E\|X_i(0)\|_{L^2}^2
+C(t,\l)
 \no
\\
&\quad+C\int_0^t\E\Big(\sum_{i=1}^N\|Y_i(s)\|_{L^2}^2\Big)\E\Big(\l(1+\l^{59})(R_N^1+R_N^2+Q_N^5)\Big) \dif s  \label{sz1}
\\
&\quad+C\int_0^t\E\Big[\Big(\sum_{i=1}^N\|Y_i(s)\|_{L^2}^2\Big)\l(1+\l^{59})\Big| R_N^1+R_N^2+Q_N^5-\E[R_N^1+R_N^2+Q_N^5]\Big|\Big] \dif s. \no
\end{align}
Here $C$ is independent of $\l, m$ and $N$.
By definition of $X_i$ in \eqref{eq:211}   and similar cancelation as in the proof of Lemma \ref{lem:sto}
we have 
\begin{align*}\sum_{i=1}^N\E\|X_i(0)\|_{L^2}^2=\frac{\l^2}{N^2}\sum_{i=1}^N\E\Big\|\sum_{j=1}^N\tilde{\cZ}^{\<30>}_{ijj}(0)\Big\|_{L^2}^2\lesssim \frac{1}{N^2}\sum_{i,j=1}^N\E\|\tilde{\cZ}^{\<30>}_{ijj}(0)\|_{L^2}^2\lesssim1.
\end{align*}
Note that by Lemma~\ref{thm:renorm43}
 we can easily deduce for $T>0$
 $$
\E R_N^1 + \E R_N^2 +\E Q_N^5 \lesssim1   .
$$
Here we recall the definition of $Q_N^5$ depend on $T$ and $\E Q_N^5$ is increasing w.r.t. $T$.
The  last line of \eqref{sz1} is controlled by
\begin{align*}
&\frac{\l}{2N}\int_0^t\E \Big[\Big(\sum_{i=1}^N\|Y_i(s)\|_{L^2}^2\Big)^2\Big]\dif s
\\
&+C(\l) N\sum_{k=1,2}\int_0^t\E \Big[ (R_N^k-\E[R_N^k])^2\Big]\dif s
+C(\l) N\int_0^t\E \Big[ (Q_N^5-\E[Q_N^5] )^2\Big]\dif s .
\end{align*}
We claim that the last two terms  here
-- although apparently have a large factor $N$ -- are actually bounded
by constant independent of $N$.

Indeed
 $R_N^1 - \E R_N^1$,  $R_N^2 - \E R_N^2$ and $Q_N^5 - \E Q_N^5$
are all  summations of terms of the form
$$
 \frac{1}{N^l} \sum_{i_1\cdots i_l=1}^N M_{i_1,\cdots,i_l}
$$
for different choices of $l$  (namely $l=1,2$ for  $R_N^1-\E R_N^1$,  $l=1,2,3$ for  $R_N^2-\E R_N^2$,
and  $l=2,3,4,5$ for  $Q_N^5 - \E Q_N^5$),
where each $M_{i_1,\cdots,i_l}$ is mean-zero, has bounded second moment,
 and
they satisfy the independence assumption in Lemma~\ref{lem:M} because
 they only involve the stochastic objects  in \eqref{e:defZ}
  constructed from the independent noises $(\xi_i)_{i=1}^N$.
So by  Lemma~\ref{lem:M} below the claim is proved.

%
Recall that $C$ is independent of $\l, m$ and $N$. Choosing
$$
m>C\l(1+\l^{59})(\E[R_N^0+R_N^2]+\E[Q_N^5])+1,
$$
 with $T=64$ in $Q_N^5$ which is uniformly bounded w.r.t. $m$ by Lemma \ref{thm:renorm43},   we have for $t\leq 64$
\begin{align}
\frac18\int_0^t & \E\sum_{i=1}^N\|Y_i(s)\|_{H^{1-2\kappa}}^2 \dif s+\int_0^t\frac{\l}{N}\E\Big\|\sum_{i=1}^NY_i^2(s)\Big\|_{L^2}^2\dif s  \nonumber
\\
& \leq C(t,\l)+2N\E\|\tilde{\Phi}_i(0)-Z_i(0)\|_{L^2}^2.	\label{eq:yi1}
\end{align}
Combining \eqref{bdd:1} and \eqref{eq:yi1} we conclude that for $t\leq 64$
\begin{align*}
&\E\|\tilde{\Phi}_i(0)-Z_i(0)\|_{H^{\frac{1}{2}-2\kappa}}^2\leq \frac{C(t,\l)}{Nt}+\frac{32}{t}\E\|\tilde{\Phi}_i(0)-Z_i(0)\|_{L^2}^2.
\end{align*}
By Lemma \ref{lem:zz1} we have $\E\|\tilde{\Phi}_i(0)-Z_i(0)\|_{L^2}^2$ finite.
Choosing $t=64$ and using $\|\cdot\|_{L^2}\leq \|\cdot\|_{H^{\frac{1}{2}-2\kappa}}$ we obtain
the claim \eqref{se8}.
\end{proof}

\bl  \label{lem:M}
Let $l$ be a fixed positive integer and  $(M_{i_1,\cdots,i_l} : i_1,\cdots,i_l \in \{1,\cdots,N\})$
be a collection of mean-zero random variables
such that $\E [M_{i_1,\cdots,i_l}^2] \lesssim 1$ uniformly in $N$ for any $ i_1,\cdots,i_l \in \{1,\cdots,N\}$,
and assume that
$M_{i_1,\cdots,i_l} $ and $M_{j_1,\cdots,j_l}$ are independent when
the $2l$ indices $i_1,\cdots,i_l,j_1,\cdots,j_l$ are all different. Then we have
$$
\E \Big[ \Big( \frac{1}{N^l} \sum_{i_1\cdots i_l=1}^N M_{i_1,\cdots,i_l} \Big)^2\Big] \le C/N
$$
where $C$ only depends on $l$ and is independent of $N$.
\el

\begin{proof}
Writing the LHS as
$$
\frac{1}{N^{2l}} \sum_{i_1\cdots i_l=1}^N \sum_{j_1\cdots j_l=1}^N
\E[ M_{i_1,\cdots,i_l}  M_{j_1,\cdots,j_l} ]
 $$
we see that the expectation is zero
when $i_1,\cdots,i_l,j_1,\cdots,j_l$ are all different by the mean-zero and the independence assumptions.
When these indices are not all different,
 the number of summands is $N^{2l} - \frac{N!}{(N-2l)!} \le C N^{2l-1}$
 where $C$ only depends on $l$ but independent of $N$,
  and each summand is bounded by
our moment bound assumption, so we obtain the claimed bound.
\end{proof}

\subsection{Observables}\label{sec:Ob}

In this section we write the stationary solutions constructed in Lemma \ref{lem:zz1} as $(\Phi_i,Z_i)$.
In the following we study the  observables \eqref{e:twoObs}.
They are defined as follows.
By Lemma \ref{lem:zz1}  we  decompose ${\Phi}_i=X_i+Y_i+Z_i$ with $({\Phi}_i, Z_i)$ stationary and $X_i$ introduced in \eqref{eq:211}.
With this we define
\begin{align}\label{eq:o1}
\frac{1}{\sqrt{N}}\sum_{i=1}^N
\Wick{\Phi_i^2}=\frac{1}{\sqrt{N}}\sum_{i=1}^N \Big(X_i^2+Y_i^2+{\cZ}^{\<2>}_{ii}+2X_iY_i+2X_iZ_i+2Y_iZ_i \Big).
\end{align}
%

By \eqref{eq:yi1} and \eqref{se8} we also have
\bl\label{lem:o1} Let $(m,\l)\in[1,\infty)\times [0,\infty)$. For $m\geq c_0\l(1+\l^{59})+1$ with $c_0$ as in Theorem \ref{th:main}. It holds that
\begin{align*}
&\int_0^t\E\sum_{i=1}^N\|Y_i(s)\|_{H^{1-2\kappa}}^2 \dif s+\int_0^t\frac{1}{N}\E\Big\|\sum_{i=1}^NY_i^2(s)\Big\|_{L^2}^2\dif s
\lesssim 1,
\end{align*}
where the proportional constant may depend on $\l$ and is independent of $N$.
\el

\begin{proof}[Proof of Theorem~\ref{th:1.3}]
By stationarity we have
\begin{align*}
\E\Big\|\frac{1}{\sqrt{N}}\sum_{i=1}^N\Wick{\Phi_i^2}\Big\|_{B^{-1-2\kappa}_{1,1}}
=\frac{1}{T}\E\Big\|\frac{1}{\sqrt{N}}\sum_{i=1}^N\Wick{\Phi_i^2}\Big\|_{L_T^1B^{-1-2\kappa}_{1,1}}.
\end{align*}
In the following we use the above equality to derive
\begin{align}\label{szz3}
\E\Big\|\frac{1}{\sqrt{N}}\sum_{i=1}^N\Wick{\Phi_i^2}\Big\|_{B^{-1-2\kappa}_{1,1}}\lesssim1,
\end{align}
which implies the desired result by the compact embedding $B^{-1-2\kappa}_{1,1}\subset B^{-1-3\kappa}_{1,1}$.

\vspace{1ex}

We consider  $L_T^1B^{-1-2\kappa}_{1,1}$-norm of each term in \eqref{eq:o1}. Using Lemma \ref{lem:multi} and \eqref{bdnx} we find
\begin{align*}
\E\Big\|\frac{1}{\sqrt{N}}\sum_{i=1}^NX_i^2\Big\|_{L_T^1B^{\frac{1}{2}-3\kappa}_{1,1}}\lesssim \E\frac{1}{\sqrt{N}}\sum_{i=1}^N\|X_i\|_{L_T^2H^{\frac{1}{2}-2\kappa}}^2\lesssim\frac{1}{\sqrt{N}}.
\end{align*}
Similarly, by  Lemma \ref{lem:multi} and Lemma \ref{lem:o1} and Lemma \ref{lem:X2} we have
\begin{align*}
\E\Big\|\frac{1}{\sqrt{N}}\sum_{i=1}^NY_i^2\Big\|_{L_T^1B^{\frac{1}{2}-3\kappa}_{1,1}} & \lesssim \E\frac{1}{\sqrt{N}}\sum_{i=1}^N\|Y_i\|_{L_T^2H^{\frac{1}{2}-2\kappa}}^2\lesssim\frac{1}{\sqrt{N}},
\\
\E\Big\|\frac{1}{\sqrt{N}}\sum_{i=1}^NX_iY_i\Big\|_{L_T^1B^{\frac{1}{2}-3\kappa}_{1,1}}
 &\lesssim \E\frac{1}{\sqrt{N}}\Big(\sum_{i=1}^N\|Y_i\|_{L_T^2H^{\frac{1}{2}-2\kappa}}^2\Big)^{\frac12}
\Big(\sum_{i=1}^N\|X_i\|_{L_T^2H^{\frac{1}{2}-2\kappa}}^2\Big)^{\frac12}
\lesssim\frac{1}{\sqrt{N}}.
\end{align*}

Furthermore, by independence and similar argument as in the proof of Lemma \ref{lem:sto} we deduce
\begin{align*}
\E\Big\|\frac{1}{\sqrt{N}}\sum_{i=1}^N{\cZ}^{\<2>}_{ii}\Big\|_{L_T^2B^{-1-\kappa}_{2,2}}^2\lesssim \E\frac{1}{N}\sum_{i=1}^N\|{\cZ}^{\<2>}_{ii}\|_{L_T^2B^{-1-\kappa}_{2,2}}^2\lesssim 1.
\end{align*}
Using Lemma \ref{lem:multi} and Lemma \ref{lem:o1} we obtain
\begin{align*}
\E\Big\|\frac{1}{\sqrt{N}}\sum_{i=1}^NY_iZ_i\Big\|_{L_T^1B^{-\frac{1}{2}-2\kappa}_{1,1}}
\lesssim \E\frac{1}{\sqrt{N}}\Big(\sum_{i=1}^N\|Y_i\|_{L_T^2H^{1-2\kappa}}^2\Big)^{\frac12}
\Big(\sum_{i=1}^N\|Z_i\|_{L_T^2H^{-\frac{1}{2}-\kappa}}^2\Big)^{\frac12}\lesssim1.
\end{align*}
\vspace{1ex}

It only remains to consider $X_iZ_i$.
By Lemma \ref{lem:para} we have
\begin{align*}
\E\frac{1}{\sqrt{N}}\sum_{i=1}^N
\|X_i\prec Z_i  &+X_i\succ Z_i  \|_{L_T^2B^{-\frac12-2\kappa}_{2,2}}
\lesssim \E\frac{1}{\sqrt{N}}\sum_{i=1}^N\|X_i\|_{L_T^2H^{\frac12-2\kappa}}\|Z_i\|_{C_T\bC^{-\frac12-\kappa}}
\\&\lesssim \E\sum_{i=1}^N\|X_i\|_{L_T^2H^{\frac12-2\kappa}}^2+\frac{1}{N}\E\sum_{i=1}^N\|Z_i\|_{C_T\bC^{-\frac12-\kappa}}^2\lesssim1.
\end{align*}
Here we use Lemma \ref{lem:X2} in the last inequality.
For $X_i\circ Z_i$ we recall \eqref{sto3} that
\begin{align*}
X_i\circ Z_i=&-\frac{\l}{N} \sum_{j=1}^N \bigg[\tilde{\cZ}^{\<31>}_{ijj,i}+\cI(2X_j\prec \UU_> {\cZ}^{\<2>}_{ij}+X_i\prec \UU_>\CZ^{\<2>}_{jj})\circ Z_i\bigg],
\end{align*}
which by Lemma \ref{lem:para} implies that
\begin{align*}
\E\frac{1}{\sqrt{N}} & \sum_{i=1}^N\|X_i\circ Z_i\|_{L_T^1B^{-2\kappa}_{1,1}}
\lesssim \E\frac{1}{N^{3/2}} \sum_{i=1}^N \bigg\|\sum_{j=1}^N\tilde{\cZ}^{\<31>}_{ijj,i}\bigg\|_{L_T^1B^{-\kappa}_{2,2}}
\\
&+\E\frac{1}{N^{3/2}} \sum_{i,j=1}^N \|X_j\|_{L_T^2H^{\frac12-2\kappa}}\|{\cZ}^{\<2>}_{ij}\|_{C_T\bC^{-1-\kappa}}\| Z_i\|_{C_T\bC^{-\frac12-\kappa}}
\\
&+\E\frac{1}{N^{3/2}} \sum_{i,j=1}^N\|X_i\|_{L_T^2H^{\frac12-2\kappa}} \|\CZ^{\<2>}_{jj}\|_{C_T\bC^{-1-\kappa}}\| Z_i\|_{C_T\bC^{-\frac12-\kappa}}.
:=\sum_{i=1}^3 J_i^N.
\end{align*}

Regarding $J_2^N$ we can bound each summand by
$$
  \frac{1}{N} \|X_j\|_{L_T^2H^{\frac12-2\kappa}}^2
\| Z_i\|_{C_T\bC^{-\frac12-\kappa}}^2
+
 \frac{1}{N^2}\|{\cZ}^{\<2>}_{ij} \|_{C_T\bC^{-1-\kappa}}^2
$$
and after taking expectation and summation it is equal to
\begin{align*}
\E\Big(\sum_{j=1}^N\|X_j\|_{L_T^2H^{\frac12-2\kappa}}^2\Big)
\Big(\frac{1}{N}\sum_{i=1}^N\| Z_i\|_{C_T\bC^{-\frac12-\kappa}}^2\Big)
+\Big(\frac{1}{N^2}\E\sum_{i,j=1}^N\|{\cZ}^{\<2>}_{ij}\|_{C_T\bC^{-1-\kappa}}^2\Big).
\end{align*}
Applying Cauchy--Schwarz to the  product in the first term
followed by \eqref{bdnx}  and Lemma~\ref{thm:renorm43} and Lemma \ref{lem:X2} we have $J_2^N\lesssim 1$.
For the term $J_3^N$, each summand is bounded by
$$
  \frac{1}{N} \|X_i\|_{L_T^2H^{\frac12-2\kappa}}^2
\|{\cZ}^{\<2>}_{jj} \|_{C_T\bC^{-1-\kappa}}^2
+
 \frac{1}{N^2}
 \| Z_i\|_{C_T\bC^{-\frac12-\kappa}}^2
$$
and then $J_3^N\lesssim 1$ follows in the same way.

The first term $J_1^N$  is bounded by
\begin{align*}
\E\frac{1}{\sqrt{N}} \Big\|\sum_{j=1}^N\tilde{\cZ}^{\<31>}_{ijj,i}\Big\|_{L_T^1B^{-\kappa}_{2,2}}
\lesssim T\Big(\E\frac{1}{N} \Big\|\sum_{j=1}^N\tilde{\cZ}^{\<31>}_{ijj,i}\Big\|_{B^{-\kappa}_{2,2}}^2\Big)^{\frac12}.
\end{align*}
The quantity in the parenthesis is equal to
 \begin{align*}
\E\frac{1}{N} \sum_{j_1,j_2=1}^N\bigg\la \Lambda^{-\kappa}\tilde{\cZ}^{\<31>}_{ij_1j_1,i}, \Lambda^{-\kappa}\tilde{\cZ}^{\<31>}_{i j_2 j_2,i}\Big\ra.
\end{align*}
For the case  $j_1\neq j_2$, the expectation is zero.  We  thus conclude
that $J_1^N\lesssim1$. Combining all the calculations above \eqref{szz3} follows.
\end{proof}

\br
We believe that nontrivial formulae for the correlations of the observables in the large $N$ limit can be calculated explicitly as in 2D case (see \cite[Section 6]{SSZZ20}). 
However, the solutions to  the SPDEs \eqref{eq:main} and  the related field $\nu^N$ are much more singular  compared to its counterpart in 2D case. The uniform in $N$ estimates in 3D,  which are the main part in Sections \ref{sec:Sto}-\ref{sec:uni}, are much more complicated than the 2D case. Moreover, as $\Phi^3$ is not a random distribution in 3D, it would require more effort to interpret the new terms when applying Dyson--Schwinger equation. 
	In the 2D case $L^p$ energy estimates are enough to derive the convergence of the correlations of the observable, but in 3D case we need further decomposition and uniform in $N$ Schauder estimate to bound the cubic term. These requires more complicated estimates than that in  Sections \ref{sec:Sto}-\ref{sec:uni}, so we leave it to the future studies.
\er


 \appendix
\renewcommand{\appendixname}{Appendix~\Alph{section}}
\renewcommand{\theequation}{A.\arabic{equation}}

\section{Notations and Besov spaces}
\label{sec:pre}

\subsection{Besov spaces}\label{sub:1}
We denote by $(\Delta_{i})_{i\geq -1}$ the Littlewood--Paley blocks for a dyadic partition of unity.
The Besov spaces $B^\alpha_{p,q}$ on the torus with $\alpha\in \R$, $p,q\in[1,\infty]$ are defined as
the completion of $C^\infty$ with respect to the norm
$$\|u\|_{B^\alpha_{p,q}}\eqdef(\sum_{j\geq-1}2^{j\alpha q}\|\Delta_ju\|_{L^p}^q)^{1/q}.$$

The following embedding results will  be frequently used. 

\bl\label{lem:emb} Let $1\leq p_1\leq p_2\leq\infty$ and $1\leq q_1\leq q_2\leq\infty$, and let $\alpha\in\mathbb{R}$. Then $B^\alpha_{p_1,q_1} \subset B^{\alpha-d(1/p_1-1/p_2)}_{p_2,q_2}$.
Here  $\subset$ means  continuous and dense embedding. (cf. \cite[Lemma~A.2]{GIP15})

\el

We also recall the following interpolation lemma.

\bl\label{lem:interpolation}
Suppose that $\theta\in (0,1)$, $\alpha_1,\alpha_2\in \mR$. Then for $\alpha=\theta \alpha_1+(1-\theta)\alpha_2$
$$\|f\|_{H^\alpha}\lesssim \|f\|_{H^{\alpha_2}}^{1-\theta}\|f\|_{H^{\alpha_1}}^\theta.$$
(cf. \cite[Theorem 4.3.1]{Tri78})
\el
\subsection{Smoothing effect of heat flow}
\label{sec:heat}

We recall the following  smoothing effect of the heat flow $P_t=e^{t(\Delta-m)}$, $m>0$ (e.g. \cite[Lemma~A.7]{GIP15}, \cite[Proposition~A.13]{MW18}).
\vskip.10in
\bl\label{lem:heat}  Let $u\in B^{\alpha}_{p,q}$ for some $\alpha\in \mathbb{R}, p,q\in [1,\infty]$. Then for every $\delta\geq0$, $t\in[0,T]$
$$\|P_tu\|_{B^{\alpha+\delta}_{p,q}}\lesssim t^{-\delta/2}\|u\|_{B^{\alpha}_{p,q}},$$
where the proportionality constants are uniform for $m\geq1$.
If $0\leq \beta-\alpha\leq 2$, then
$$\|(\mathrm{I}-P_t)u\|_{B^{\alpha}_{p,q}}\lesssim t^{\frac{\beta-\alpha}2}\|u\|_{B^{\beta}_{p,q}},$$
where the proportionality constants are uniform for $m\geq1$.
\el
We also define $(\cI f)(t,x):=(\sL^{-1}f)(t,x)=\int_0^tP_{t-s}f\dif s$.

\begin{lemma}\label{lemma:sch} (\cite[Lemma~A.9]{GIP15},\cite[Lemma~2.8, Lemma~2.9]{ZZZ20})
	Let $\alpha\in\R$.
	Then the following bounds hold uniformly over $0\leq t\leq T$
	\begin{equation}\label{eq:1vw}
		\|\cI f(t)\|_{\bC^{2+\alpha}}\lesssim \|f\|_{L_T^{\infty}\bC^{\alpha}}.
	\end{equation}
	If $0\leqslant 2+\alpha < 2$ then
	\begin{equs}
		\|\cI f(t)\|_{C_T^{(2+\alpha)/2}L^\infty}\lesssim \|f\|_{L_T^\infty\bC^{\alpha}},
	\end{equs}
	where the proportionality constants are uniform for $m\geq1$.
\end{lemma}


\bl\label{lem:z1} For $\beta\in \mR$ we have
$$
\|\cI f\|_{L_T^2H^\beta}\lesssim \|f\|_{L_T^2H^{\beta-2}}
\qquad \mbox{and} \qquad
\|\cI f\|_{W_T^{\frac14-3\kappa,2}L^2}\lesssim  \|f\|_{L^2_TH^{-\frac32-2\kappa}}$$
for $0<\kappa<1/20$, where the proportionality constants are uniform for $m\geq1$.
\el
\begin{proof} Let $\{e_k(x)\}=\{2^{-\frac32}e^{\iota \pi k\cdot x}, k\in \mZ^3\}$ on $\mathbb{T}^3$.
	\begin{align*}
		\|\cI f\|^2_{L_T^2H^\beta}
		&=\int_0^T\sum_k(|k|^2+m)^{\beta}\Big|\Big\langle \int_0^tP_{t-s}f\dif s,e_k\Big\rangle\Big|^2\dif t
		\\
		&=\int_0^T\sum_k(|k|^2+m)^{\beta}\Big| \int_0^te^{-(t-s)(|k|^2+m)}\langle f,e_k\rangle\dif s\Big|^2\dif t
		\\
		&\lesssim \int_0^T\sum_k(|k|^2+m)^{\beta-1}\int_0^te^{-(t-s)(|k|^2+m)}|\langle f,e_k\rangle|^2\dif s\dif t
		\\
		&=\sum_k(|k|^{2}+m)^{\beta-1}\int_0^T\int_s^Te^{-(t-s)(|k|^2+m)}\dif t|\langle f,e_k\rangle|^2\dif s
		\lesssim \|f\|_{L_T^2H^{\beta-2}}^2,
	\end{align*}
	where we use H\"older's inequality in the first inequality.
	Thus the first result follows.
	
	By the definition of $W_T^{\frac14-3\kappa,2}L^2$ we have
	\begin{align*}
		\|\cI f\|_{W_T^{\frac14-3\kappa,2}L^2}^2 =\int_0^T\|\cI f(s)\|_{L^2}^2\dif s+\int_0^T\int_0^T\frac{\|\cI f(t)-\cI f(s)\|^2_{L^2}}{|t-s|^{1+2(\frac14-3\kappa)}}\dif s\dif t
		=I_1+I_2.
	\end{align*}
	$I_1$ can be easily controlled by the first result and we consider the $I_2$ part.
	Using the smoothing effect of the heat kernel given in Lemma \ref{lem:heat} we have
	\begin{align*}
		I_2
		&\lesssim  \int_0^T\int_0^t\frac{\|(P_{t-s}-\mathrm{I})\cI f(s)\|^2_{L^2}}{|t-s|^{1+2(\frac14-3\kappa)}}\dif s\dif t+\int_0^T\int_0^t\frac{\left(\int_s^t(t-r)^{-\frac34-\kappa}\|f(r)\|_{H^{-\frac32-2\kappa}}\dif r\right)^2}{|t-s|^{1+2(\frac14-3\kappa)}}\dif s\dif t
		\\
		&\lesssim \int_0^T\int_0^t\frac{\|\cI f(s)\|^2_{H^{\frac12-2\kappa}}}{|t-s|^{1-2\kappa}}\dif s\dif t+\int_0^T\int_0^t\frac{\int_s^t(t-r)^{-\frac34-\kappa}\|f(r)\|^2_{H^{-\frac32-2\kappa}}dr}{|t-s|^{1+\frac14-5\kappa}}\dif s\dif t
		\\
		&\lesssim  \|\cI f\|^2_{L^2_TH^{\frac12-2\kappa}}+ \|f\|^2_{L^2_TH^{-\frac32-2\kappa}},
	\end{align*}
	where we used the fact that the second term on the right is controlled by
	\begin{align*}
		&\int_0^T\int_r^T\int_0^r(t-r)^{-\frac34-\kappa}\|f(r)\|^2_{H^{-\frac32-2\kappa}}|t-s|^{-1-\frac14+5\kappa}\dif s \dif t \dif r
		\\\lesssim&\int_0^T\int_r^T (t-r)^{-1+4\kappa}\|f(r)\|^2_{H^{-\frac32-2\kappa}}\dif t\dif r
	\end{align*}
	in the last step. Thus the second result follows from the first result.
\end{proof}

\subsection{Paraproducts and commutators}

We recall the following paraproduct introduced by Bony (see \cite{Bon81}). In general, the product $fg$ of two distributions $f\in \bC^\alpha, g\in \bC^\beta$ is well defined if and only if $\alpha+\beta>0$. In terms of Littlewood-Paley blocks, the product $fg$ of two distributions $f$ and $g$ can be formally decomposed as
$$fg=f\prec g+f\circ g+f\succ g,$$
with $$f\prec g=g\succ f=\sum_{j\geq-1}\sum_{i<j-1}\Delta_if\Delta_jg, \quad f\circ g=\sum_{|i-j|\leq1}\Delta_if\Delta_jg.$$
We also denote
\begin{align*}
	\succcurlyeq \, \eqdef \, \succ+\circ,
	\qquad \preccurlyeq\,\eqdef\, \prec+\circ.
\end{align*}

The following results on paraproduct on  Besov space is from \cite{Bon81} (see also \cite[Lemma~2.1]{GIP15},  \cite[Proposition~A.7]{MW18}).

\begin{lemma}\label{lem:para}
	Let  $\beta\in\R$, $p, p_1, p_2, q\in [1,\infty]$ such that $\frac{1}{p}=\frac{1}{p_1}+\frac{1}{p_2}$. Then we have
	\begin{equs}
		\|f\prec g\|_{B^\beta_{p,q}}
		&\lesssim\|f\|_{L^{p_1}}\|g\|_{B^{\beta}_{p_2,q}}
		\\
			\|f\prec g\|_{B^{\alpha+\beta}_{p,q}}
			&\lesssim\|f\|_{B^{\alpha}_{p_1,q}}\|g\|_{B^{\beta}_{p_2,q}}   \qquad (\mbox{for }\alpha<0)
				\\
				\|f\circ g\|_{B^{\alpha+\beta}_{p,q}}
				&\lesssim\|f\|_{B^{\alpha}_{p_1,q}}\|g\|_{B^{\beta}_{p_2,q}} \qquad
				(\mbox{for }\alpha+\beta>0).
			\end{equs}
		\end{lemma}
		
		Furthermore, we have the following multiplicative inequality.
		
		\bl\label{lem:multi}
		(i) Let $\alpha,\beta\in\mathbb{R}$ and $p, p_1, p_2, q\in [1,\infty]$ be such that $\frac{1}{p}=\frac{1}{p_1}+\frac{1}{p_2}$.
		The bilinear map $(u, v)\mapsto uv$
		extends to a continuous map from ${B}^\alpha_{p_1,q}\times {B}^\beta_{p_2,q}$ to ${B}^{\alpha\wedge\beta}_{p,q}$  if $\alpha+\beta>0$. (cf. \cite[Corollary~3.21]{MW17})
		
		(ii) (Duality.) Let $\alpha\in (0,1)$, $p,q\in[1,\infty]$, $p'$ and $q'$ be their conjugate exponents, respectively. Then the mapping  $\langle u, v\rangle\mapsto \int uv \dif x$  extends to a continuous bilinear form on $B^\alpha_{p,q}\times B^{-\alpha}_{p',q'}$, and one has $|\langle u,v\rangle|\lesssim \|u\|_{B^\alpha_{p,q}}\|v\|_{B^{-\alpha}_{p',q'}}$ (cf.  \cite[Proposition~3.23]{MW17}).
		
		(iii) (Fractional Leibniz estimate) Let $s\geq0$, $p, p_1, p_2, p_3, p_4, q\in [1,\infty]$ satisfy $\frac1p=\frac1{p_1}+\frac1{p_2}=\frac1{p_3}+\frac1{p_4}$. Then it holds that \cite[Proposition~A.2]{CGW2020}
		\begin{align}\label{ine:Lei}
			\|uv\|_{B^s_{p,q}}\lesssim \|u\|_{B^s_{p_1,q}}\|v\|_{L^{p_2}}
			+\|u\|_{L^{p_3}}\|v\|_{B^s_{p_4,q}}.
		\end{align}
		\el

		By duality and \cite[Proposition~3.25]{MW17} we easily deduce the following (cf. \cite[Lemma~2.5]{SSZZ20}):

		\begin{lemma}\label{lem:dual+MW}
			For $s \in (0,1)$
			\begin{equation*}
				|\langle g,f\rangle|\lesssim \big (\|\nabla g \|_{L^{1}}^{s}\|g\|_{L^{1}}^{1-s}+ \|g\|_{L^{1}} \big )\|f\|_{\bC^{-s}}  .
			\end{equation*}
		\end{lemma}

		We also recall the following commutator estimate (\cite[Lemma 2.4]{GIP15}, \cite[Proposition A.9]{MW18}).
		
		\begin{lemma}\label{lem:com1}
			Let $\alpha\in (0,1)$ and $\beta,\gamma\in \R$ such that $\alpha+\beta+\gamma>0$ and $\beta+\gamma<0$, $p,p_1,p_2\in [1,\infty]$, $\frac{1}{p}=\frac{1}{p_1}+\frac{1}{p_2}$. Then there exist a trilinear bounded operator $\tilde C(f,g,h):B^\alpha_{p_1,\infty}\times \bC^\beta\times B^\gamma_{p_2,\infty}\to B^{\alpha+\beta+\gamma}_{p,\infty}$ satisfying
			$$
			\|\tilde{C}(f,g,h)\|_{B^{\alpha+\beta+\gamma}_{p,\infty}}\lesssim \|f\|_{B^\alpha_{p_1,\infty}}\|g\|_{\bC^\beta}\|h\|_{B^\gamma_{p_2,\infty}}
			$$
			and for smooth functions $f,g,h$
			$$
			\tilde{C}(f,g,h)=(f\prec g)\circ h - f(g\circ h).
			$$
		\end{lemma}
		
		The following lemmas are from \cite[Lemma A.13, Lemma A.14]{GH18}.
		In \cite{GH18} these lemmas are stated in the discrete setting,
		but the same arguments lead to the following statements in the continuum
		setting.
		
		\begin{lemma}\label{lem:com3}
			Let $\alpha,\beta,\gamma\in \R$ such that $\alpha+\beta+\gamma>0$ and $\beta+\gamma<0$. Then there exist a trilinear bounded operator $D(f,g,h):H^\alpha\times \bC^\beta\times H^\gamma\to \R$ satisfying
			$$
			|D(f,g,h)|\lesssim \|f\|_{H^\alpha}\|g\|_{\bC^\beta}\|h\|_{H^\gamma}
			$$
			and for smooth functions $f,g,h$
			$$
			D(f,g,h)=\la f,g\circ h\ra  - \la f\prec g, h\ra .
			$$
		\end{lemma}
		
		\begin{lemma}\label{lem:com2}
			Let $\alpha,\beta,\gamma\in \R$ such that $\alpha\in (0,1)$,  $\alpha+\beta+\gamma+2>0$ and $\beta+\gamma+2<0$. Then there exist  trilinear bounded operators
			$$C(f,g,h):H^\alpha\times \bC^\beta\times \bC^{\gamma+\delta}\to H^{\beta+\gamma+2},$$
			$$\bar{C}(f,g,h):\bC^\alpha\times \bC^\beta\times \bC^{\gamma+\delta}\to \bC^{\beta+\gamma+2}$$
			satisfying for $\delta>0$
			\begin{equs}
				\|C(f,g,h)\|_{H^{\beta+\gamma+2}}
				&\lesssim \|f\|_{H^\alpha}\|g\|_{\bC^\beta}\|h\|_{\bC^{\gamma+\delta}}
				\\
				\|\bar{C}(f,g,h)\|_{C_T\bC^{\beta+\gamma+2}}
				&\lesssim (\|f\|_{C_T\bC^\alpha}+\|f\|_{C_T^{\alpha/2} L^\infty})\|g\|_{C_T\bC^\beta}\|h\|_{C_T\bC^{\gamma+\delta}}
			\end{equs}
			and for smooth functions $f,g,h$
			\begin{equs}
				C(f,g,h) & = ((m-\Delta)^{-1}(f\prec g) )\circ h- f(h\circ (m-\Delta)^{-1}g),
				\\
				\bar{C}(f,g,h) & =\cI(f\prec g)\circ h-f\cI(g)\circ h
				\qquad(\mbox{with } \cI=\LL^{-1}).
			\end{equs}
		\end{lemma}
		
		We also prove the following estimate for commutators:
		
		\bl\label{lem:com}
		For $\bar C$ as in Lemma \ref{lem:com2}, $T>0$,
		$f\in L_T^2H^{\frac{1}{2}-2\kappa}\cap W^{\frac{1}{4}-3\kappa,2}_TL^2$,
		and  $g, h\in C_T\bC^{-1-\kappa}$ with $0<\kappa<1/22$, it holds that
		\begin{align*}
			\|\bar{C}(f,g,h)\|_{L^2_TL^2}\lesssim (\|f\|_{L^2_TH^{\frac12-2\kappa}}+\|f\|_{W_T^{\frac14-3\kappa,2}L^2})\|g\|_{C_T\bC^{-1-\kappa}}\|h\|_{C_T\bC^{-1-\kappa}}.\end{align*}
		Here the proportionality constants are uniform for $m\geq1$ and may depend on $T$.
		\el
		\begin{proof}
			We write $\bar{C}(f,g,h)(t)$ as
			\begin{align*}
				\int_0^t & P_{t-s}(\delta_{st}f\prec g_s)\dif s\circ h_t
				+\int_0^t(P_{t-s}[f_t\prec g_s]-f_t\prec P_{t-s}g_s)\dif s\circ h_t
				\\& +\Big((f\prec\cI g)\circ h-f(\cI g\circ h)\Big)(t)
				\qquad \eqdef I_1(t)+I_2(t)+I_3(t),
			\end{align*}
			with $\delta_{st}f=f_s-f_t$.
			By Lemmas \ref{lem:com1} and \ref{lemma:sch} we have
			\begin{align*}
				\|I_3\|^2_{L^2_TL^2}
				&\lesssim \|h\|^2_{C_T\bC^{-1-\kappa}}\|\mathcal{I}g\|^2_{C_T\bC^{1-\kappa}}\|f\|_{L_T^2H^{\frac{1}{2}-2\kappa}}^2
				\\
				&\lesssim \|h\|^2_{C_T\bC^{-1-\kappa}}\|g\|^2_{C_T\bC^{-1-\kappa}}\|f\|_{L_T^2H^{\frac{1}{2}-2\kappa}}^2.\end{align*}
			Using Lemma \ref{lem:para}, Lemma \ref{lem:heat} followed by H\"older's inequality we find
			\begin{align*}
				\|I_1\|^2_{L^2_TL^2}
				&\lesssim \|h\|^2_{C_T\bC^{-1-\kappa}}\|\cI(\delta_{st}f\prec g)\|_{L_T^2H^{1+2\kappa}}^2
				\\
				&\lesssim \|h\|^2_{C_T\bC^{-1-\kappa}}\|g\|^2_{C_T\bC^{-1-\kappa}}\int_0^T\Big[\int_0^t\frac{\|\delta_{st}f\|_{L^2}}{|t-s|^{1+2\kappa}}\dif s\Big]^2\dif t
				\\
				&\lesssim \|h\|^2_{C_T\bC^{-1-\kappa}}\|g\|^2_{C_T\bC^{-1-\kappa}}\int_0^T \! \int_0^t\frac{\|\delta_{st}f\|^2_{L^2}}{|t-s|^{1+5\kappa}}\,\dif s\dif t
				\\
				&\lesssim
				\|f\|^2_{W_T^{\frac14-3\kappa,2}L^2}\|g\|^2_{C_T\bC^{-1-\kappa}}\|h\|^2_{C_T\bC^{-1-\kappa}}.
			\end{align*}
			Moreover, Lemma \ref{lem:para} and \cite[Proposition A.16]{MW18} imply that
			\begin{align*}
				\|I_2\|^2_{L^2_TL^2}
				&\lesssim \|h\|^2_{C_T\bC^{-1-\kappa}}\Big\|\int_0^t \Big(P_{t-s}[f_t\prec g_s]-f_t\prec P_{t-s}g_s \Big)\dif s\Big\|_{L_T^2H^{1+2\kappa}}^2
				\\&\lesssim \|h\|^2_{C_T\bC^{-1-\kappa}}\|g\|^2_{C_T\bC^{-1-\kappa}}\int_0^T\Big[\int_0^t\frac{\|f_t\|_{H^{\frac12-2\kappa}}}{|t-s|^{1-2\kappa}}\dif s\Big]^2\dif t
				\\&\lesssim\|f\|^2_{L^2_TH^{\frac12-2\kappa}}\|g\|^2_{C_T\bC^{-1-\kappa}}\|h\|^2_{C_T\bC^{-1-\kappa}},
			\end{align*}
			where we use H\"older's inequality in the last step.
		\end{proof}

\renewcommand{\theequation}{B.\arabic{equation}}
\section{Extra estimates}
\label{sec:extra}

 By similar argument as in Section \ref{sec:uni} we deduce
\bl\label{lem:a1}
(Energy balance)
\begin{align}\label{zmm}
\frac{1}{2} \frac1N\sum_{i=1}^{N}\frac{\dif}{\dif t}\|Y_{i}\|_{L^{2}}^{2}+\frac{1}{N^2}\bigg \|\sum_{i=1}^{N}Y_{i}^{2} \bigg \|_{L^{2}}^{2}+\frac1N\sum_{i=1}^{N}\|Y_{i}\|_{H^{1-2\kappa}}^{2}\lesssim R_N,
\end{align}
for $R_N$ given in the proof. Here the implicit constant may depend on $\l$ and $m$.
\el
\begin{proof}
	Using Lemma \ref{den}
	we find
	\begin{align*}
	\frac{1}{2} \frac1N\sum_{i=1}^{N}\frac{\dif}{\dif t}\|Y_{i}\|_{L^{2}}^{2}+\frac{1}{N^2}\bigg \|\sum_{i=1}^{N}Y_{i}^{2} \bigg \|_{L^{2}}^{2}+\frac mN\sum_{i=1}^{N}\|\varphi_{i}\|_{L^2}^{2}+\frac 1N\sum_{i=1}^{N}\|\nabla\varphi_{i}\|_{L^2}^{2}
	\leq \frac1N\Theta+\frac1N\Xi.
	\end{align*}
In Proposition \ref{theta1} and Lemma \ref{Cubic} we already deduce the required bound for $\frac1N\Theta$ and the cubic terms in $\frac1N\Xi$.
In the following  we consider the remaining terms in  $\frac1N\Xi$. In the proof of Proposition \ref{Xi1} we give estimate for $\|\Xi\|_{L_T^1}$.
Following  the same argument  and estimate at a fixed time, we find
\begin{align*}
&\frac1N\Theta+\frac1N\Xi\lesssim \delta \frac 1N\sum_{i=1}^{N}\|Y_{i}\|_{H^{1-2\kappa}}^{2}+\frac 1N\sum_{i=1}^{N}\|Y_{i}\|_{L^2}^{2}(\bar{Q}_N^3+R_N^2)+\frac{\bar{Q}_N^4}N
\\ & \lesssim \delta \frac 1N\sum_{i=1}^{N}\|Y_{i}\|_{H^{1-2\kappa}}^{2}
+\delta\frac{1}{N^2}\bigg \|\sum_{i=1}^{N}Y_{i}^{2} \bigg \|_{L^{2}}^{2}+
(\bar{Q}_N^3+R_N^2)^2+\frac1N\bar{Q}_N^4,
\end{align*}
with
\begin{align*}
\bar{Q}_N^3 &\eqdef R_N^4+1+\Big(\frac{1}{N}\sum_{j=1}^N\|X_j\|_{\bC^{\frac{1}{2}-2\kappa}}^2\Big)^2+\Big(\frac{1}{N^2}\sum_{i,j=1}^N\|X_jZ_i\|^2_{\bC^{-\frac12-\kappa}}\Big)
\\&\qquad+\Big(\frac{1}{N}\sum_{j=1}^N\|X_jZ_j\|^2_{\bC^{-\frac12-\kappa}}\Big)
+\frac{1}{N^2}\sum_{i,j=1}^N\Big(\|{\cZ}^{\<2>}_{ij}\|_{\bC^{-1-\kappa}}^{\frac{1}{1-\theta}}+\|{\CZ}^{\<2>}_{jj}\|_{\bC^{-1-\kappa}}^{\frac{1}{1-\theta}}\Big),
\end{align*}
with $\theta=\frac{\frac12+\kappa}{1-2\kappa}$ and
\begin{align*}
\bar{Q}_N^4 &\eqdef
\Big(\sum_{i=1}^N\|X_i\|_{H^{\frac{1}{2}-2\kappa}}^2\Big)
\Big[1+\frac{1}{N^2}\sum_{i,j=1}^N 2^{8\kappa L} \Big( \|{\cZ}^{\<2>}_{ij}\|^2_{\bC^{-1-\kappa}}+\|{\CZ}^{\<2>}_{jj}\|^2_{\bC^{-1-\kappa}}\Big)	\notag
\\
&+\Big(\frac{1}{N^2}\sum_{i,j=1}^N\|{\cZ}^{\<2>}_{ij}\|_{\bC^{-1-\kappa}}^2\Big)
+\Big(\frac{1}{N}\sum_{j=1}^N\|{\CZ}^{\<2>}_{jj}\|_{\bC^{-1-\kappa}}^2\Big)\Big] 	\notag
\\&+\Big(\frac{1}{N^2}\sum_{i=1}^N\Big(\sum_{j=1}^N\|X_j^2Z_i\|_{H^{-\frac12-2\kappa}}\Big)^{2}\Big)
+\Big(\frac{1}{N^2}\sum_{i=1}^N\Big(\sum_{j=1}^N\|X_iX_jZ_j\|_{H^{-\frac12-2\kappa}}\Big)^{2}\Big)
\notag\\
&+\Big(\frac{1}{N^2}\sum_{i=1}^N\Big(\sum_{j=1}^N
\|X_i\circ \CZ^{\<2>}_{jj}\|_{H^{-\frac12-2\kappa}}
+\|X_j\circ{\cZ}^{\<2>}_{ij}\|_{H^{-\frac12-2\kappa}}
\Big)^{2}\Big)	
+\frac{t^{-\gamma}}{N^2}\sum_{i=1}^N\|X_i\|_{L^2}^2,
\end{align*}
for $\gamma>0$.
Using the first inequality in  Lemma \ref{lem:phi1} the result holds with $R_N$ given by
$(\bar{Q}_N^3+R_N^2+R_N^1)^2+\frac1N\bar{Q}_N^4+1$.
	\end{proof}

  We denote by $\Phi_{i}$ the solution to \eqref{eq:21} starting from the stationary solution $Z_{i}(0)$, so that the process $\Phi_i-Z_i$ starts from the origin as in the proof of Lemma \ref{lem:zz1}.
By using Lemma \ref{lem:a1} and \cite[Lemma 3.8]{TW18}, we obtain the following result.

\bl \label{lem:A2}
For every $T\geq 1$ and $\kappa>0$, the following holds uniformly in $T$
\begin{align*}
\int_0^T\mathbf{E}\Big[\frac{1}{N}\sum_{i=1}^N\| (\Phi_i-Z_i)(t)\|_{L^2}^2\Big]\dif t & \lesssim T,
\\
\int_0^T\mathbf{E}\Big[\frac{1}{N}\sum_{i=1}^N\| (\Phi_i-Z_i)(t)\|_{\bC^{-\frac12-\kappa}}^2\Big]\dif t & \lesssim T.
\end{align*}
\el
\begin{proof}
	On every interval $[s,s+2]$, $s\geq1$ we decompose $\Phi_i-Z_i=X_i^s+Y_i^s$ with $X_i^s$ satisfies
	\begin{align*}X_i^s(t)=&-\frac{\l}N\int_s^t\sum_{j=1}^NP_{t-r}(X_j^s\prec \UU_>\cZ_{ij}^{\<2>}+X_i^s\prec\UU_>\cZ_{jj}^{\<2>})\dif r+\frac{\l}N\sum_{j=1}^N\tilde\cZ^{\<30>}_{ijj}(t).\end{align*}
	and
	$Y_i^s$ satifies \eqref{eq:22} with $X_i$ replaced by $X_i^s$ and the initial condition  $Y_i^s(s)=\Phi_i(s)-Z_i(s)+\frac{\l}N\sum_{j=1}^N\tilde{\cZ}^{\<30>}_{ijj}(s)$.
	By similar argument as in Lemma \ref{lem:X} we deduce
	\begin{align*}\mathbf{E}\Big(\frac{1}{N}\sum_{i=1}^N\| X_i^s(t)\|_{C([s,s+2];\bC^{\frac12-\kappa})}^2\Big)\lesssim 1,
	\end{align*}
	with the proportinal constant independent of $s$, $t$ and initial condition. We find Lemma \ref{lem:a1} also holds with $Y_i$, and $X_i$   replaced by  $Y_i^s$, $X_i^s$, respectively. For every $\tau\in\mZ$ in Lemma \ref{thm:renorm43} we know
	$$\sup_{s\geq0}\E [\sup_{s\leq t\leq s+2}\tau]\lesssim1,$$
	which implies that
		$$\sup_{s\geq1}\E \sup_{s\leq t\leq s+2} R_N^s(t)\lesssim1,$$
		for $R_N^s$ given as $R_N$ in the proof of Lemma \ref{lem:a1} with $X_i$ replaced by $X_i^s$.
	 Now by  \cite[Lemma 3.8]{TW18} and \eqref{yi}  we deduce for $1\leq s\leq t\leq s+1$
	\begin{align*}
	\mathbf{E}\Big(\frac{1}{N}\sum_{i=1}^N\| Y_i^s(t)\|_{L^2}^2\Big)\lesssim (t-s)^{-1/2},
	\end{align*}
	with the proportinal constant independent of $s$ and initial condition,
	which by taking integration over $t$ for \eqref{zmm} implies that for $s\geq 1$
		\begin{align}\label{zrx}
	\int_{s+\frac12}^{s+2}\mathbf{E}\Big(\frac{1}{N}\sum_{i=1}^N\| Y_i^s(t)\|_{H^{1-2\kappa}}^2\Big)\dif t\lesssim1,
	\end{align}
	with the proportinal constant independent of $s$.
	Then we know for $T\geq 1$
	\begin{align*}
	&\int_0^T\mathbf{E}\Big(\frac{1}{N}\sum_{i=1}^N\| (\Phi_i-Z_i)(t)\|_{L^2}^2\Big)\dif t
	\\\lesssim &\sum_{s=1}^{T-1}\int_{s+\frac12}^{s+\frac32}\mathbf{E}\Big(\frac{1}{N}\sum_{i=1}^N\| X_i^s(t)\|_{L^2}^2\Big)\dif t+
	\sum_{s=1}^{T-1}\int_{s+\frac12}^{s+\frac32}\mathbf{E}\Big(\frac{1}{N}\sum_{i=1}^N\| Y_i^s(t)\|_{L^2}^2\Big)\dif t
	\\&+\int_{0}^{\frac32}\mathbf{E}\Big(\frac{1}{N}\sum_{i=1}^N\| X_i^s(t)\|_{L^2}^2\Big)\dif t+
	\int_{0}^{\frac32}\mathbf{E}\Big(\frac{1}{N}\sum_{i=1}^N\| Y_i^s(t)\|_{L^2}^2\Big)\dif t
	\lesssim T,
	\end{align*}
	where we used Theorem \ref{Y:T4a} for the bound of the integral from $0$ to $\frac32$. 
 Similarly we obtain
 \begin{align*}
 &\int_0^T\mathbf{E}\Big(\frac{1}{N}\sum_{i=1}^N\| (\Phi_i-Z_i)(t)\|_{\bC^{-\frac12-2\kappa}}^2\Big)\dif t
 \lesssim T
 \end{align*} by using \eqref{zrx} and Besov embedding Lemma \ref{lem:emb}. Since $\kappa$ is arbitrary, the second result follows.
	\end{proof}
	
\section{Notation index}
\label{sec:extra1}
We collect some frequently used notations of this paper, with  their meanings and the pages where they first occur.
 Remark that all the objects named with $Q$ or $R$ are explicit combinations of norms of  $ \mathbb{Z}$
(see \eqref{e:defZ}), except that $Q^3_N$ and $Q^4_N$ also depend on $X$ (defined in  \eqref{eq:211}).
	
 \begin{center}
\renewcommand{\arraystretch}{1.1}
\begin{longtable}{lll}
\toprule
Symbol & Place  introduced & Page\\
\midrule
\endfirsthead
\toprule
Symbol & Place being introduced & Page\\
\midrule
\endhead
\bottomrule
\endfoot
\bottomrule
\endlastfoot
$\phi_i$, $P_i$ &
Eq.~\eqref{phi} &
\pageref{phi}\\
$Q^0_N$, $Q^1_N$, $Q^2_N$ &
Lemma~\ref{lem:sto} &
\pageref{lem:sto} \\
$Q^3_N$, $Q^4_N$ &
Proposition~\ref{Xi1} &
\pageref{Xi1} \\
$Q^5_N$, $Q^{51}_N$ &
Lemma~\ref{lem:zmm} &
\pageref{eq:Q} \\
$Q^{52}_N,Q^{53}_N$ &
Lemma~\ref{lem:X1} &
\pageref{lem:X1} \\
$Q^6_N$  &
Lemma~\ref{lem:X1} &
\pageref{page Q6} \\
$R_N^0$ &
Lemma~\ref{lem:X} &
\pageref{lem:X} \\
$R_N^1$ &
Lemma~\ref{lem:zz2} &
\pageref{lem:zz2} \\
$R_N^2$ &
Proposition~\ref{theta1} &
\pageref{theta1} \\
$R_N^4$ &
Lemma~\ref{Cubic} &
\pageref{Cubic} \\
$X_i$ &
Eq. \eqref{eq:211} &
\pageref{eq:211} \\
\end{longtable}
 \end{center}

\bibliographystyle{alphaabbr}
\bibliography{Reference}

\end{document}